\documentclass[a4paper,11pt]{article}
\usepackage[english]{babel}

\usepackage[utf8]{inputenc}
\usepackage[T1]{fontenc}

\usepackage{lmodern}
\usepackage{amssymb,amsmath,amsfonts}
\usepackage{xspace}
\usepackage{graphicx}
\usepackage{hyperref}
\usepackage{url}
\usepackage{stmaryrd}
\usepackage{mathrsfs}
\usepackage[babel=true]{csquotes}
\usepackage{amsthm}
\usepackage{subfigure}
\usepackage{mathrsfs}
\usepackage{float}

\theoremstyle{plain}
\newtheorem{thm}{Theorem}[section]
\newtheorem{lm}[thm]{Lemma}
\newtheorem{prop}[thm]{Proposition}
\newtheorem{cor}[thm]{Corollary}
\newtheorem*{thmintro}{Theorem}

\theoremstyle{definition}
\newtheorem{dfn}[thm]{Definition}

\def\rmq{\emph{Remark.}\xspace}

\def\dpar#1#2{\frac{\partial #1}{\partial #2}}

\def\Hil{\mathcal{H}\xspace}

\def\N{\mathbb{N}\xspace}
\def\Z{\mathbb{Z}\xspace}
\def\R{\mathbb{R}\xspace}
\def\C{\mathbb{C}\xspace}
\def\P{\mathbb{P}\xspace}
\def\T{\mathbb{T}\xspace}
\def\Sch{\mathscr{S}\xspace}
\def\SBar{\mathfrak{S}\xspace}
\def\scal#1#2{\left\langle #1,#2\right\rangle}
\def\norm#1{\left\|#1\right\|}
\def\classe#1#2{\mathcal{C}^{#1}#2}
\def\Barg{\mathcal{B}\xspace}

\hypersetup{colorlinks=true, linkcolor=blue}

\begin{document}

\title{Singular Bohr-Sommerfeld conditions for 1D Toeplitz operators: hyperbolic case}

\author{Yohann Le Floch\footnote{Universit\'e de Rennes 1, IRMAR, UMR 6625,
Campus de Beaulieu, b\^atiments 22 et 23, 263 avenue du G\'en\'eral Leclerc,
CS 74205, 35042  RENNES C\'edex, France; email: yohann.lefloch@univ-rennes1.fr}
}

\maketitle

\begin{abstract} 
In this article, we state the Bohr-Sommerfeld conditions around a singular value of hyperbolic type of the principal symbol of a self-adjoint semiclassical Toeplitz operator on a compact connected K\"{a}hler surface. These conditions allow the description of the spectrum of the operator in a fixed size neighbourhood of the singularity. We provide numerical computations for three examples, each associated to a different topology.
\end{abstract}

\section{Introduction}

Let $M$ be a compact, connected K\"{a}hler manifold of complex dimension $1$, with fundamental 2-form $\omega$. Assume that $M$ is endowed with a prequantum bundle $L$, that is a Hermitian, holomorphic line bundle whose Chern connection has curvature $-i\omega$. Let $K$ be another Hermitian holomorphic line bundle and define the quantum Hilbert space $\Hil_{k}$ as the space of holomorphic sections of $L^{\otimes k} \otimes K$, for every positive integer $k$. We consider (Berezin-)Toeplitz operators (see for instance \cite{BouGui,BorPauUri2,Cha1,MaMa}) acting on $\Hil_{k}$. The semiclassical limit corresponds to $k \rightarrow +\infty$.

The usual Bohr-Sommerfeld conditions \cite{Cha2}, recalled in section \ref{subsect:BSreg}, describe the intersection of the spectrum of a selfadjoint Toeplitz operator to a neighbourhood of any regular value of its principal symbol, in terms of geometric quantities (actions). A natural question is whether one can write Bohr-Sommerfeld conditions near a singular value of the principal symbol. In the case of a nondegenerate singularity of elliptic type, it was answered positively in \cite{LF}, and the result is quite simple: roughly speaking, the singular Bohr-Sommerfeld conditions are nothing but the limit of the regular Bohr-Sommerfeld conditions when the energy goes from regular to singular. The hyperbolic case is much more complicated, because the topology of a neighbourhood of the singular level is. For instance, in the case of one hyperbolic point, the critical level looks like a figure eight, and crossing it has the effect of adding (or removing) one connected component from the regular level.

Let us mention that the case of Toeplitz operators is very close to the case of pseudodifferential operators. In this setting, the problem of describing the spectrum of a selfadjoint operator near a singular level of hyperbolic type was handled by Colin de Verdi\`ere and Parisse in a series of articles \cite{CdV3,CdV4,CdVP}. In this article, we use analogous techniques to write hyperbolic Bohr-Sommerfeld conditions in the context of Toeplitz operators. The novelty is that they can be applied in this context.

\subsection{Main result}

Let $A_k$ be a self-adjoint Toeplitz operator on $M$; its normalized symbol $a_{0} + \hbar a_{1} + \ldots$ is real-valued. Assume that $0$ is a critical value of the principal symbol $a_0$, that the level set $\Gamma_{0} = a_{0}^{-1}(0)$ is connected and that every critical point contained in $\Gamma_{0}$ is non-degenerate and of hyperbolic type. Let $S = \{ s_{j} \}_{1 \leq j \leq n}$ be the set of these critical points.
$\Gamma_{0}$ is a compact graph embedded in $M$, and each of its vertices has local degree $4$. At each vertex $s_{j}$, we denote by $e_{m}$, $m=1,2,3,4$, the local edges, labeled with cyclic order $(1,3,2,4)$ (with respect to the orientation of $M$ near $s_{j}$) and such that $e_{1},e_{2}$ (resp. $e_{3},e_{4}$) correspond to the local unstable (resp. stable) manifolds. Cut $n + 1$ edges of $\Gamma_{0}$, each one corresponding to a cycle $\gamma_{i}$ in a basis $(\gamma_{1}, \ldots, \gamma_{n + 1})$ of $\text{H}_{1}(\Gamma_{0},\Z)$, in such a way that the remaining graph is a tree $T$. Our main result is the following:
\begin{thmintro}[theorem \ref{thm:cdvp}, theorem \ref{thm:holsing}]
$0$ is an eigenvalue of $A_{k}$ up to $O(k^{-\infty})$ if and only if the following system of $3n+1$ linear equations with unknowns $(x_{\alpha} \in \C_{k})_{\alpha \in \{\mathrm{edges \ of \ } T\}}$ has a non-trivial solution:
\begin{enumerate}
\item if the edges $(\alpha_{1}, \alpha_{2}, \alpha_{3}, \alpha_{4})$ connect at $s_{j}$, then
\begin{equation*} \begin{pmatrix} x_{\alpha_{3}} \\ x_{\alpha_{4}} \end{pmatrix} = T_{j} \begin{pmatrix} x_{\alpha_{1}} \\ x_{\alpha_{2}}, \end{pmatrix} \end{equation*}
\item if $\alpha$ and $\beta$ are the extremities of a cut cycle $\gamma_{i}$, then 
\begin{equation*} x_{\alpha} = \exp\left( i k \theta(\gamma_{i},k) \right) x_{\beta}, \end{equation*}
where the following orientation is assumed: $\gamma_{i}$ can be represented as a closed path starting on the edge $\alpha$ and ending on the edge $\beta$.
\end{enumerate}
Moreover, $T_{j}$ is a matrix depending only on a semiclassical invariant $\varepsilon_{j}(k)$ of the system at the singular point $s_{j}$, and $\theta(\gamma,k)$ admits an asymptotic expansion in non-positive powers of $k$. The first two terms of this expansion involve regularizations of the geometric invariants (actions and index) appearing in the usual Bohr-Sommerfeld conditions. 

\end{thmintro}
For spectral purposes, we use this theorem by replacing $A_{k}$ by $A_{k} - E$ for $E$ varying in fixed size neighbourhood of the singular level. Away from the critical energy, we recover the regular Bohr-Sommerfeld conditions. 

This is very similar to the results of Colin de Verdi\`ere and Parisse \cite{CdVP}, but the novelty lies in the framework that had to be set in order to extend their techniques to the Toeplitz setting, and also in the geometric invariants that are specific to this context.

\subsection{Structure of the article}

As said earlier, the case of Toeplitz operators is very close to the case of pseudodifferential operators; in mathematical terms, there is a microlocal equivalence between Toeplitz operators and pseudodifferential operators. When the phase space is the whole complex plane, this equivalence is realized by the Bargmann transform. Hence, we will use some of the results in the pseudodifferential setting in this work.  
This is why the article is organized as follows: first, we construct a microlocal normal form for $A_{k}$ near each critical point $s_{j}, 1 \leq j \leq n$ on Bargmann spaces. Then, we use the Bargmann transform and the study of Colin de Verdi\`ere and Parisse \cite{CdV2} to describe the space of microlocal solutions of $A_{k}$ near $s_{j}$. Finally, we adapt the reasoning of Colin de Verdi\`ere and Parisse \cite{CdVP} and Colin de Verdi\`ere and V{\~u} Ng{\d{o}}c \cite{SanCdV} to obtain the singular Bohr-Sommerfeld conditions (in section \ref{sect:BS}). We give numerical evidence in the last section.

\section{Preliminaries and notations}

\subsection{Symbol classes}

We introduce rather standard symbol classes. Let $d$ be a positive integer. For $u$ in $\C^{d} \simeq \R^{2d}$, let $m(u) = \left(1 + \|u\|^2 \right)^{\frac{1}{2}}$. For every integer $j$, we define the symbol class $\mathcal{S}_{j}^d$ as the set of sequences of functions of $\classe{\infty}{(\C^{d})}$ which admit an asymptotic expansion of the form $a(.,k) = \sum_{\ell \geq 0} k^{-\ell}a_{\ell}$ in the sense that
\begin{itemize}
\item $\forall \ell \in \N \quad  \forall \alpha, \beta \in \N^{2d} \quad \exists \ C_{\ell,\alpha,\beta} > 0 \quad |\partial_{z}^{\alpha} \partial_{\bar{z}}^{\beta} a_{\ell}| \leq C_{\ell,\alpha,\beta} m^j$,
\item $\forall L \in \N^* \quad \forall \alpha, \beta \in \N^{2d} \quad \exists  \ C_{L,\alpha} > 0 \quad \left| \partial_{z}^{\alpha} \partial_{\bar{z}}^{\beta} \left(a - \sum_{\ell=0}^{L-1} k^{-\ell} a_{\ell} \right)  \right| \leq C_{L,\alpha,\beta} k^{-L} m^j$.
\end{itemize}
We set $\mathcal{S}^d = \bigcup_{j \in \Z} \mathcal{S}_{j}^d$. If, in this definition, we only consider symbols independent of $z$, we obtain the class $\C_{k}$ of constant symbols; we will also sometimes speak of ``admissible constants''.

\subsection{Function spaces}

Using standard notations, we denote by $\Sch(\R)$ the Schwartz space of functions $f \in \classe{\infty}{(\R)}$ such that for all $j,p \in \N$, $\sup_{t \in \R} \ |t^j f^{(p)}(t)| < +\infty$, by $\mathscr{D}'(\R)$ the space of distributions on $\R$, and by $\mathscr{S}'(\R)$ the space of tempered distributions on $\R$ (the dual space of $\Sch(\R)$). We recall that 
\begin{equation*} \Sch(\R) = \underset{j \in \N}{\bigcap} \Sch_{j}(\R), \end{equation*}
where $\Sch_{j}(\R)$ is the space of functions $f$ of $\classe{j}{(\R)}$ with $\|f\|_{\Sch_{j}}$ finite, with
\begin{equation*} \|f\|_{\Sch_{j}} = \underset{0 \leq p \leq j}{\max} \left( \underset{t \in \R}{\sup} \left| (1 + t^2)^{(j-p)/2} f^{(p)}(t) \right| \right). \end{equation*}
The topology of $\Sch(\R)$ is defined by the countable family of semi-norms $\|.\|_{\Sch_{j}}, j \in \N$.

We recall the definition of Bargmann spaces \cite{Bar, Bar2}, which are spaces of square integrable functions with respect to a Gaussian weight:
\begin{equation*} \Barg_{k} = \left\{ f\psi^k; f:\C \mapsto \C \ \text{holomorphic}, \int_{\R^2} |f(z)|^2 \exp(-k|z|^2) \ d\lambda(z) < +\infty \right\} \end{equation*}
with $\psi:\C \rightarrow \C, z \mapsto  \exp\left( -\frac{1}{2} |z|^2 \right)$, $\psi^k: \C \to \C^{\otimes k}$ its $k$-th tensor power, and $\lambda$ the Lebesgue measure on $\R^2$. Furthermore, we introduce the subspace
\begin{equation} \mathfrak{S}_k = \left\{ \varphi \in \Barg_k; \ \forall j \in \N \quad \sup_{z \in \C} \left( |\varphi(z)|(1 + |z|^2)^{j/2}  \right) < + \infty  \right\} \label{eq:sbark}\end{equation}	
of $\Barg_{k}$, with topology induced by the obvious associated family of semi-norms.

\subsection{Weyl quantization and pseudodifferential operators}

We briefly recall some standards notations and properties of the theory of pseudodifferential operators (for details, see \emph{e.g.} \cite{CdV2,DS,Zwo}), replacing the usual small parameter $\hbar$ by $k^{-1}$.

\subsubsection{Pseudodifferential operators}

A pseudodifferential operator in one degree of freedom is an operator (possibly unbounded) acting on $L^2(\R)$ which is the Weyl quantization of a symbol $a(.,k) \in \mathcal{S}^1$:
\begin{equation*} A_{k} = \left( \text{Op}_{k}^{W}(a)u \right)(x) = \frac{k}{2\pi} \int_{\R^2} \exp\left( i k(x-y)\xi \right) a\left( \frac{x+y}{2},\xi,k) u(y) \, dy \, d\xi \right) \end{equation*}
The sequence $a(.,k)$ is a sequence of functions defined on the cotangent space $T^*\R \simeq \R^2$; the leading term $a_{0}$ in its asymptotic expansion is the principal symbol of $A_{k}$. $A_{k}$ is said to be \emph{elliptic} at $(x_{0},\xi_{0}) \in T^*\R$ if $a_{0}(x_{0}, \xi_{0}) \neq 0$.

\subsubsection{Wavefront set}

\begin{dfn}
A sequence $u_{k}$ of elements of $\mathscr{D}'(\R)$ is said to be \emph{admissible} if for any pseudodifferential operator $P_{k}$ whose symbol is compactly supported, there exists an integer $N \in \Z$ such that $\norm{P_{k}u_{k}}_{L^2} = O(k^N)$. 
\end{dfn}
We recall the standard definition of the wavefront set $\text{WF}(u_{k})$ of an admissible sequence of distributions.
\begin{dfn}
\label{dfn:wavefront}
Let $u_{k}$ be an admissible sequence of $\mathscr{D}'(\R)$. A point $(x_{0}, \xi_{0})$ does not belong to $\text{WF}(u_{k})$ if and only if there exists a pseudodifferential operator $P_{k}$, elliptic at $(x_{0}, \xi_{0})$, such that $\| P_{k}u_{k} \|_{L^2} = O(k^{-\infty})$.
\end{dfn}
One can refine these definitions in the case where $u_{k}$ belong to $\Sch(\R)$.
\begin{dfn}
A sequence $(u_{k})_{k \geq 1}$ of elements of $\Sch(\R)$ is said to be
\begin{itemize}
\item $\Sch$-\textit{admissible} if there exists $N$ in $\Z$ such that every Schwartz semi-norm of $u_{k}$ is $O(k^{N})$,
\item $\Sch$-\textit{negligible} if it is admissible and every Schwartz semi-norm of $u_{k}$ is $O(k^{-\infty})$. We write $u_{k} = O_{\Sch}(k^{-\infty})$.
\end{itemize}
\end{dfn}
Now, instead of using the $L^2$-norm in definition \ref{dfn:wavefront}, one can actually consider the semi-norms $\|.\|_{\Sch_{j}}$.
\begin{lm}
A point $(x_{0}, \xi_{0})$ does not belong to $\text{WF}(u_{k})$ if and only if there exists a pseudodifferential operator $P_{k}$, elliptic at $(x_{0}, \xi_{0})$, such that $P_{k}u_{k} = O_{\Sch}(k^{-\infty})$.
\label{lm:pseudoneg}\end{lm}
\begin{proof}
The sufficient condition comes from the previous definition, so we only prove the necessary condition. We only adapt a standard argument used when one wants to deal with $\classe{j}{}$-norms (see \cite[proposition IV$-8$]{Rob}). Assume that $(x_{0}, \xi_{0})$ does not belong to $\text{WF}(u_{k})$; there exists a pseudodifferential operator $P_{k}$, elliptic at $(x_{0}, \xi_{0})$, such that $\| P_{k}u_{k} \|_{L^2} = O(k^{-\infty})$. Consider a compactly supported smooth function $\chi$ equal to one in a neighbourhood of $(x_{0},\xi_{0})$ and set $Q_{k} = \text{Op}^{\text{W}}(\chi) P_{k}$. For every $R \in \R[X]$ and every integer $j > 0$, $k^{-j} \frac{d^j}{dx^j} R \text{Op}^{\text{W}}(\chi)$ is a pseudodifferential operator of order 0, hence bounded $L^2(\R) \rightarrow L^2(\R)$ by a constant $C > 0$ (by Calderon-Vaillancourt theorem). Thus, one has $\| k^{-j} \frac{d^j}{dx^j} RQ_{k}u_{k} \|_{L^2} \leq C \| P_{k}u_{k} \|_{L^2} = O(k^{-\infty})$. Hence, $\| RQ_{k}u_{k} \|_{H^s}  = O(k^{-\infty})$ for every integer $s > 0$; Sobolev injections then yield that every $\classe{j}{}$-norm of $RQ_{k}u_{k}$ is $O(k^{-\infty})$. Since this holds for every polynomial $R$, we obtain the result.
\end{proof}

\subsection{Geometric quantization and Toeplitz operators}

We also recall the standard definitions and notations in the Toeplitz setting. Unless otherwise mentioned, \enquote{smooth} will always mean $\classe{\infty}{}$, and a section of a line bundle will always be assumed to be smooth. The space of sections of a bundle $E \rightarrow M$ will be denoted by $\classe{\infty}{(M,E)}$. 
Let $M$ be a connected compact K\"{a}hler manifold, with fundamental 2-form $\omega \in \Omega^2(M,\R)$. Assume $M$ is endowed with a prequantum bundle $L \rightarrow M$, that is a Hermitian holomorphic line bundle whose Chern connection $\nabla$ has curvature $-i \omega$. 
Let $K \rightarrow M$ be a Hermitian holomorphic line bundle. For every positive integer $k$, define the quantum space $\Hil_{k}$ as:
\begin{equation*} \Hil_{k} = H^0(M,L^k \otimes K) = \left\{ \text{holomorphic sections of } L^k \otimes K \right\}. \end{equation*}
The space $\Hil_{k}$ is a subspace of the space $L^2(M,L^k \otimes K)$ of sections of finite $L^2$-norm, where the scalar product is given by
\begin{equation*} \langle \varphi, \psi \rangle = \int_{M} h_{k}(\varphi,\psi) \mu_{M} \end{equation*}
with $h_{k}$ the Hermitian product on $L^k \otimes K$ induced by those of $L$ and $K$, and $\mu_{M}$ the Liouville measure on $M$. Since $M$ is compact, $\Hil_{k}$ is finite dimensional, and is thus given a Hilbert space structure with this scalar product.

\subsubsection{Admissible and negligible sequences}

Let $(s_{k})_{k \geq 1}$ be a sequence such that for each $k$, $s_{k}$ belongs to $\classe{\infty}{(M,L^k \otimes K)}$. We say that $(s_{k})_{k \geq 1}$ is 
\begin{itemize}
\item \textit{admissible} if for every positive integer $\ell$, for every vector fields $X_{1},\ldots,X_{\ell}$ on $M$, and for every compact set $C \subset M$, there exist a constant $c > 0$ and an integer $N$ such that
\begin{equation*} \forall m \in C \quad \| \nabla_{X_{1}} \ldots \nabla_{X_{\ell}} s_{k}(m) \| \leq c k^{N}, \end{equation*}
\item \textit{negligible} if for every positive integers $\ell$ and $N$, for every vector fields $X_{1},\ldots,X_{\ell}$ on $M$, and for every compact set $C \subset M$, there exists a constant $c > 0$ such that
\begin{equation*} \forall m \in C \quad \| \nabla_{X_{1}} \ldots \nabla_{X_{\ell}} s_{k}(m) \| \leq c k^{-N}. \end{equation*}
\end{itemize}
In a standard way, one can then define the microsupport $\text{MS}(u_{k})$ of an admissible sequence $u_{k}$ and the notion of microlocal equality.

\subsubsection{Toeplitz operators}

Let $\Pi_{k}$ be the orthogonal projector of $L^2(M,L^k \otimes K)$ onto $\Hil_{k}$. A \textit{Toeplitz operator} is any sequence $(T_{k}: \Hil_{k} \rightarrow \Hil_{k})_{k \geq 1}$ of operators of the form
\begin{equation*} T_{k} = \Pi_{k} M_{f(.,k)} + R_{k} \end{equation*}
where $f(.,k)$ is a sequence of $\classe{\infty}{(M)}$ with an asymptotic expansion $f(.,k) = \sum_{\ell \geq 0} k^{-\ell} f_{\ell}$ for the $\classe{\infty}{}$ topology, $M_{f(.,k)}$ is the operator of multiplication by $f(.,k)$ and $\norm{R_{k}} = O(k^{-\infty})$. 
Define the contravariant symbol map
\begin{equation*} \sigma_{\text{cont}} : \mathcal{T} \rightarrow \classe{\infty}{(M)}[[\hbar]] \end{equation*}
sending $T_{k}$ into the formal series $\sum_{\ell \geq 0} \hbar^{\ell} f_{\ell}$. We will mainly work with the \textit{normalized symbol}
\begin{equation*} \sigma_{\text{norm}} = \left( \text{Id} + \frac{\hbar}{2} \Delta \right) \sigma_{\text{cont}} \end{equation*}
where $\Delta$ is the holomorphic Laplacian acting on $\classe{\infty}{(M)}$; unless otherwise mentioned, when we talk about a subprincipal symbol, this refers to the normalized symbol.

\subsubsection{The case of the complex plane}
\label{subsubsect:toepcomp}

Let us briefly recall how to adapt the previous constructions to the case of the whole complex plane. We consider the K\"{a}hler manifold $\C \simeq \R^2$ with coordinates $(x,\xi)$, standard complex structure and symplectic form $\omega_{0} = d\xi \wedge dx$. Let $L_0 = \R^2 \times \C \rightarrow \R^2$ be the trivial fiber bundle with standard Hermitian metric $h_{0}$ and connection $\nabla^0$ with $1$-form $\frac{1}{i} \alpha$, where $\alpha_u(v) = \frac{1}{2} \omega_0(u,v)$; endow $L_{0}$ with the unique holomorphic structure compatible with $h_{0}$ and $\nabla^0$.
For every positive integer $k$, the quantum space at order $k$ is
\begin{equation*} \Hil_k^0 = H^0(\R^2,L_{0}^k) \cap L^2(\R^2,L_{0}^k), \end{equation*}
and it turns out that $\Hil_k^0 = \Barg_k$ (if we choose the holomorphic coordinate $z=\frac{x-i\xi}{\sqrt{2}}$).
One can define the algebra of Toeplitz operators and the various symbols in a similar way than in the compact case; see \cite{LF} for details. We will call $\mathcal{T}_{j}$ the class of Toeplitz operators with symbol in $\mathcal{S}_{j}^1$. 

Let us give more details about the microsupport in this setting. We start by recalling the following inequality in Bargmann spaces \cite[equation $(1.7)$]{Bar}.
\begin{lm}
Let $\phi_{k} \in \Barg_{k}$. Then for every complex variable $z$
\begin{equation*} |\phi_{k}(z)| \leq \left( \frac{k}{2\pi} \right)^{1/2} \| \phi_{k} \|_{\Barg_{k}}. \end{equation*}
Similarly, for every vector fields $X_{1}, \ldots, X_{p}$ on $\C$, there exists a polynomial $P \in \R[x_{1},x_{2}]$ with positive values such that for every $z \in \C$
\begin{equation*} |(\nabla_{X_{1}} \ldots \nabla_{X_{p}}\phi_{k})(z)| \leq P(|z|,k)^{1/2} \| \phi_{k} \|_{\Barg_{k}}. \end{equation*}
\label{lm:ineqBarg}\end{lm}

\begin{proof}
The first claim is proved in \cite{Bar} in the case $k=1$; the general case then comes from a change of variables. The second claim can be proved in the same way.
\end{proof}

\begin{lm}
Let $u_{k}$ be a sequence of elements of $\Barg_{k}$ and $\Omega$ a bounded open subset of $\C$. Assume that $\|u_{k}\|_{L^2(\Omega)} = O(k^{-\infty})$; then for any compact subset $K$ of $\Omega$, $u_{k}$ and all its covariant derivatives are uniformly $O(k^{-\infty})$ on $K$.
\label{lm:l2neg}\end{lm}
\begin{proof}
Choose a compactly supported smooth function $\eta$ which is positive, vanishing outside $\Omega$ and with constant value $1$ on $K$ and set $v_{k} = \text{Op}(\eta) u_{k}$. One has
\begin{equation*} \| v_{k} \|_{\Barg_{k}} = \| \Pi_{k}^0 \eta u_{k} \|_{\Barg_{k}} \leq  \| \eta u_{k} \|_{L^2} \leq \| u_{k} \|_{L^2(\Omega)}  \end{equation*}
since $\Pi_{k}^0$ is continuous $L^2 \rightarrow L^2$ with norm smaller than $1$. Hence, $\| v_{k} \|_{\Barg_{k}} = O(k^{-\infty})$. By lemma \ref{lm:ineqBarg}, this implies that $v_{k}$ and its covariant derivatives are uniformly $O(k^{-\infty})$ on $K$; since $u_{k} = v_{k} + O(k^{-\infty})$ on $K$, the same holds for $u_{k}$. 
\end{proof}
\begin{lm}
Let $(u_{k})_{k \geq 1}$ be an admissible sequence of elements of $\Barg_{k}$ and $z_{0} \in \C$. Then $z_{0} \notin \mathrm{MS}(u_{k})$ if and only if there exists a Toeplitz operator $T_{k} \in \mathcal{T}_{0}$, elliptic at $z_{0}$, such that $\| T_{k}u_{k} \|_{\Barg_{k}} = O(k^{-\infty})$. 
\label{lm:toepneg}\end{lm}
\begin{proof}
Assume that $z_{0} \notin \text{MS}(u_{k})$. There exists a neighbourhood $\mathcal{U}$ of $z_{0}$ such that $u_{k}$ is negligible on $\mathcal{U}$. Choose a compactly supported function $\chi \in \classe{\infty}{(\C,\R)}$ with support $K$ contained in $\mathcal{U}$ and such that $\chi(z_{0}) = 1$; and set $T_{k} = \text{Op}(\chi)$. One has for $z_{1} \in \C$
\begin{equation*} (T_{k}u_{k})(z_{1}) = \frac{k}{2\pi} \int_{K} \exp\left( -\frac{k}{2} \left( |z_{1}|^2 + |z_{2}|^2 - 2z_{1}\bar{z}_{2}\right) \right) \chi(z_{2}) u_{k}(z_{2}) \ d\mu(z_{2}), \end{equation*}
which gives
\begin{equation*} |(T_{k}u_{k})(z_{1})| \leq  \frac{k}{2\pi} \underset{K}{\sup} |u_{k}|  \int_{K} \exp\left( -\frac{k}{2} |z_{1} - z_{2}|^2 \right)  \ d\mu(z_{2}). \end{equation*}
This allows to estimate the norm of $T_{k}u_{k}$:
\begin{equation*} \|T_{k}u_{k}\|_{\Barg_{k}}^2 \leq \left( \frac{k}{2\pi}\right)^2 \left(\underset{K}{\sup} |u_{k}| \right)^2 \int_{\C} \int_{K} \exp\left( -k |z_{1} - z_{2}|^2 \right)  \ d\mu(z_{1}) d\mu(z_{2}). \end{equation*}
Hence
\begin{equation*} \|T_{k}u_{k}\|_{\Barg_{k}}^2 \leq \left( \frac{k}{2\pi}\right)^2 \left(\underset{K}{\sup} |u_{k}| \right)^2 \mu(K) \int_{\C}  \exp\left( -k |z_{1}|^2 \right)  \ d\mu(z_{1}) \end{equation*}
and the necessary condition is proved since the integral is $O(k^{-1/2})$.

Conversely, assume that there exists a Toeplitz operator $T_{k} \in \mathcal{T}_{0}$ elliptic at $z_{0}$ such that $\| T_{k}u_{k} \|_{\Barg_{k}} = O(k^{-\infty})$. There exists a neighbourhood of $z_{0}$ where $T_{k}$ is elliptic. Hence, by symbolic calculus, we can find a Toeplitz operator $S_{k} \in  \mathcal{T}_{0}$ such that $S_{k}T_{k} \sim \Pi_{k}^0$ near $(z_{0},z_{0})$. 
Thus, there exists a neighbourhood $\Omega$ of $z_{0}$ such that $S_{k} T_{k} u_{k} \sim u_{k}$ on $\Omega$; this implies that $\| S_{k} T_{k} u_{k} \|_{L^2(\Omega)} = \| u_{k} \|_{L^2(\Omega)} + O(k^{-\infty})$. But, since $S_{k}$ is bounded $\Barg_{k} \rightarrow \Barg_{k}$ by a constant $C > 0$ which does not depend on $k$, one has $\| S_{k} T_{k} u_{k} \|_{L^2(\Omega)} \leq C \| T_{k} u_{k} \|_{\Barg_{k}}$; this yields that $\| u_{k} \|_{L^2(\Omega)}$ is $O(k^{-\infty})$. Lemma \ref{lm:l2neg} then gives the negligibility of $u_{k}$ on $\Omega$.
\end{proof}
\begin{dfn}
A sequence $(u_{k})_{k \geq 1}$ of elements of $\SBar_{k}$ is said to be
\begin{itemize}
\item $\SBar_{k}$-\textit{admissible} if there exists $N$ in $\Z$ such that every $\SBar_{k}$ semi-norm of $u_{k}$ is $O(k^{N})$,
\item $\SBar_{k}$-\textit{negligible} if it is $\SBar_{k}$-admissible and every $\SBar_{k}$ semi-norm of $u_{k}$ is $O(k^{-\infty})$. We write $u_{k} = O_{\SBar_{k}}(k^{-\infty})$.
\end{itemize}
\end{dfn}
\begin{lm}
\label{lm:Toepneg}
Let $(u_{k})_{k \geq 1}$ be an admissible sequence of elements of $\Barg_{k}$ and $z_{0} \in \C$. Then $z_{0} \notin \mathrm{MS}(u_{k})$ if and only if there exists a Toeplitz operator $T_{k} \in \mathcal{T}_{0}$, elliptic at $z_{0}$, such that $T_{k}u_{k}  = O_{\SBar_{k}}(k^{-\infty})$.
\end{lm}
\begin{proof}
The proof is nearly the same as the one of lemma \ref{lm:pseudoneg}. One can show that if $z_{0} \notin \text{MS}(u_{k})$, there exists a Toeplitz operator $T_{k} \in \mathcal{T}_{0}$, elliptic at $z_{0}$, such that for every polynomial function $P(z)$ of $z$ only, $\underset{z \in \C}{\sup} |P(z)(T_{k}u_{k})(z)| = O(k^{-\infty})$, using the fact that the multiplication by $P(z)$ is a Toeplitz operator.
\end{proof}

\section{The Bargmann transform}

\subsection{Definition and first properties}

The Bargmann transform is the unitary operator $B_k:L^2(\R) \rightarrow \Barg_k$ defined by
\begin{equation*}(B_{k}f)(z) = \left( \left( \frac{k}{\pi} \right)^{1/4} \int_{\R} \exp\left( k \left( -\frac{1}{2}(z^2 + t^2) + \sqrt{2}zt \right) \right) f(t) \ dt \right) \psi^k(z).\end{equation*}

The subspace $\SBar_{k}$ of $\Barg_{k}$ defined in (\ref{eq:sbark}) is the analog of the Schwartz space on the Bargmann side. The case $k=1$ is treated by the following theorem, due to Bargmann.
\begin{thm}[{\cite[theorem 1.7]{Bar2}}]
The Bargmann transform $B_{1}$ is a bijective, bicontinuous mapping between $\Sch(\R)$and $\SBar_{1}$. 
\label{thm:bargmannschwartz}\end{thm}
This allows us to handle the general case.
\begin{prop}
\label{prop:SbarSch}
The Bargmann transform $B_{k}$ is a bijection between $\Sch(\R)$ and $\SBar_{k}$.
\end{prop}
\begin{proof}
If $f$ belongs to $\Sch(\R)$, one has for $z$ in $\C$ 
\begin{equation*} (B_{k}f)(z) = \left( \frac{k}{\pi} \right)^{1/4} \int_{\R} \exp\left( k \left( -\frac{1}{2}(z^2 + t^2) + \sqrt{2}zt \right) \right) f(t) \ dt; \end{equation*}
introducing the variables $u$ and $w$ such that $z = k^{-1/2}w$ and $t = k^{-1/2}u$, this reads 
\begin{equation*} (B_{k}f)(z) = \left( k \pi \right)^{-1/4} \int_{\R} \exp\left(  -\frac{1}{2}(w^2 + u^2) + \sqrt{2}wu \right) f(k^{-1/2}u) \ du. \end{equation*}
Hence, we have $(B_{k}f)(z) =  \left( k \pi \right)^{-1/4} (B_{1}g)(k^{1/2}z)$, where $g(t) = f(k^{-1/2}t)$. Obviously, the function $g$ belongs to $\Sch(\R)$; thus, by the previous theorem, $B_{1}g$ belongs to $\mathfrak{S}_{1}$. Hence, for $j \in \N$, there exists a constant $C_{j} > 0$ such that for every complex variable $w$
\begin{equation*} \left| (B_{1}g)(w) \exp\left( -\frac{1}{2} |w|^2 \right) \right| \leq C_{j}  \left( 1 + |w|^2 \right)^{-j/2}. \end{equation*} 
This implies that for every $z$ in $\C$,
\begin{equation*} \left| (B_{k}f)(z) \exp\left( -\frac{k}{2} |z|^2 \right) \right| \leq C_{j} k^{-j/2} \left( 1 + k |z|^2 \right)^{-j/2} \end{equation*} 
and since $k \geq 1$, this yields
\begin{equation*} \left| (B_{k}f)(z) \exp\left( -\frac{k}{2} |z|^2 \right) \right| \leq C_{j}  \left( 1 +  |z|^2 \right)^{-j}, \end{equation*}
which means that $B_{k}f$ belongs to $\SBar_{k}$. The converse is proved in the same way, using the explicit form of the inverse mapping:
\begin{equation*} (B_{k}^*g)(t) = \left( \frac{k}{\pi} \right)^{1/4} \int_{\R} \exp\left( k \left( -\frac{1}{2}(\bar{z}^2 + t^2) + \sqrt{2}\bar{z}t -|z|^2 \right) \right) g(z) \ d\mu(z) \end{equation*}
for $g$ in $\SBar_{k}$ and $t \in \R$.
\end{proof}

\subsection{Action on Toeplitz operators}          

The Bargmann transform has the good property to conjugate a Toeplitz operator to a pseudodifferential operator, and conversely.                                                                                                                                                                                                       
\begin{lm}
\label{lm:actionBargmann}
Let $T_{k}$ be a Toeplitz operator in the class $\mathcal{T}_{j}$, with contravariant symbol $\sigma_{\text{cont}}(T_{k}) = f(.,\hbar)$; then $B_{k}^*T_{k}B_{k}$ is a pseudodifferential operator with Weyl symbol
\begin{equation*} \sigma^{W}(B_{k}^*T_{k}B_{k})(x,\xi) = I(f(.,\hbar))(x,\xi) = \frac{1}{\pi \hbar} \int_{\C} \exp(-2\hbar^{-1}|w|^2) f(w + z, \hbar) d\lambda(w), \end{equation*}
where $z = \frac{1}{\sqrt{2}}(x-i\xi)$. The map $I$ is continuous $\mathcal{S}_{j} \rightarrow \mathcal{S}_{j}$. Moreover, for any $f(.,\hbar) \in \mathcal{S}_{j}$ and all $p \geq 1$,
\begin{equation} I(f(.,\hbar)) = \sum_{j=0}^{p-1} \left( \frac{\hbar}{2} \right)^j \frac{\Delta^j f(.,\hbar)}{j!} + h^p R_{p}(f(.,\hbar)).  \label{eq:symbole}\end{equation}
where $R_{p}$ is a continuous map from $\mathcal{S}_{j}$ to $\mathcal{S}_{j}$.
\end{lm}
\begin{proof}
Thanks to \cite[theorem $5.2$]{SanLau}, we know that the result holds when $T_{k} = \Pi_{k}^0 f \Pi_{k}^0$, $f$ being a bounded function on $\C$ not depending on $k$. Now, using the stationary phase method, one can prove that the map $I$ is continuous $\mathcal{S}_{j} \rightarrow \mathcal{S}_{j}$ with the asymptotic expansion (\ref{eq:symbole}), and conclude by a density argument. 
\end{proof}

\subsection{Microlocalization and Bargmann transform}

\begin{lm}
\label{lm:BargSchwarz}
\begin{enumerate}
\item $B_{k}$ maps $\Sch$-admissible functions to $\SBar_{k}$-admissible sections, and $B_{k}^*$ maps $\SBar_{k}$-admissible sections to $\Sch$-admissible functions.
\item $B_{k}$ maps $O_{\Sch}(k^{-\infty})$ into $O_{\mathfrak{S}_{k}}(k^{-\infty})$, and $B_{k}^*$ maps $O_{\mathfrak{S}_{k}}(k^{-\infty})$ into $O_{\Sch}(k^{-\infty})$.
\end{enumerate}
\end{lm}
\begin{proof}
These results are proved by performing a change of variables, as in proposition \ref{prop:SbarSch}.
\end{proof}
We are now able to prove the link between the wavefront set and the microsupport \textit{via} the Bargmann transform.
\begin{prop}
\label{prop:microBarg}
Let $u_{k}$ be an admissible sequence of elements of $\Sch(\R)$. Then $(x_{0},\xi_{0}) \notin \mathrm{WF}(u_{k})$ if and only if $z_{0} =\frac{1}{\sqrt{2}}(x_{0} - i \xi_{0}) \notin \mathrm{MS}(B_{k}u_{k})$.
\end{prop}
\begin{proof}
Assume that $z_{0} = \frac{1}{\sqrt{2}}(x_{0} - i \xi_{0})$ does not belong to $\text{MS}(B_{k}u_{k})$; by lemma \ref{lm:Toepneg}, there exists a Toeplitz operator $T_{k}$, elliptic at $z_{0}$, such that $T_{k}B_{k} u_{k} \psi^k  = O_{\SBar_{k}}(k^{-\infty})$. Thanks to lemma \ref{lm:actionBargmann}, $P_{k} = B_{k}^* T_{k} B_{k}$ is a pseudodifferential operator elliptic at $(x_{0},\xi_{0})$. Furthermore, thanks to lemma \ref{lm:BargSchwarz}, $P_{k} u _{k} = B_{k}^* T_{k} B_{k} u_{k} \psi_{k} = O_{\Sch}(k^{-\infty})$; we conclude by lemma \ref{lm:pseudoneg}.
The proof of the converse follows the same steps.
\end{proof}

\section{The sheaf of microlocal solutions}
\label{sect:regular}

In this section, $T_{k}$ is a self-adjoint Toeplitz operator on $M$, with normalized symbol $f(.,\hbar) = \sum_{\ell \geq 0} \hbar^{\ell} f^{\ell}$. Following V{\~u} Ng{\d{o}}c \cite{SanThese,SanFocus}, we introduce the sheaf of microlocal solutions of the equation $T_k \psi_k=0$.

\subsection{Microlocal solutions}

Let $U$ be an open subset of $M$; we call a sequence of sections $\psi_k \in \classe{\infty}{(U,L^k \otimes K)}$ a local state over $U$.
\begin{dfn}
We say that a local state $\psi_k$ is a microlocal solution of 
\begin{equation} T_k \psi_k = 0 \label{eq:micsol}\end{equation}
on $U$ if it is admissible and for every $x \in U$, there exists a function $\chi \in \classe{\infty}{(M)}$ with support contained in $U$, equal to $1$ in a neighbourhood of $x$ and such that 
\begin{equation*} \Pi_k(\chi \psi_k) = \psi_k + O(k^{-\infty}), \quad T_k(\Pi_k(\chi \psi_k)) = O(k^{-\infty}) \end{equation*}
on a neighbourhood of $x$.
\end{dfn}
One can show that if $\psi_k \in \Hil_k$ is admissible and satisfies $T_k \psi_k = 0$, then the restriction of $\psi_k$ to $U$ is a microlocal solution of (\ref{eq:micsol}) on $U$. Moreover, the set $S(U)$ of microlocal solutions of this equation on $U$ is a $\C_k$-module containing the set of negligible local states as a submodule. We denote by $\text{Sol}(U)$ the module obtained by taking the quotient of $S(U)$ by the negligible local states; the notation $[\psi_k]$ will stand for the equivalence class of $\psi_k \in S(U)$.
\begin{lm}
The collection of $\text{Sol}(U)$, $U$ open subset of $M$, together with the natural restrictions maps $r_{U,V}:\text{Sol}(V) \rightarrow \text{Sol}(U)$ for $U,V$ open subsets of $M$ such that $U \subset V$, define a complete presheaf. 
\end{lm}
Thus, we obtain a sheaf $\text{Sol}$ over $M$, called the sheaf of microlocal solutions on $M$. 

\subsection{The sheaf of microlocal solutions}

One can show that if the principal symbol $f_0$ of $T_k$ does not vanish on $U$, then $\text{Sol}(U) = \{0\}$. Equivalently, if $\psi_k \in \Hil_k$ satisfies $T_k \psi_k = 0$, then its microsupport is contained in the level $\Gamma_0 = f_0^{-1}(0)$. This implies the following lemma.
\begin{lm}
Let $\Omega$ be an open subset of $\Gamma_0$; write $\Omega = U \cap \Gamma_0$ where $U$ is an open subset of $M$. Then the restriction map
\begin{equation*} r_U: \text{Sol}(U) \rightarrow \mathfrak{F}_U(\Omega) = r_U(\text{Sol}(U))  \end{equation*}
is an isomorphism of $\C_k$-modules. 
\end{lm}
We want to define a new sheaf $\mathfrak{F} \rightarrow \Gamma_0$ that still describes the microlocal solutions of (\ref{eq:micsol}). In order to do so, we will check that the module $\mathfrak{F}_U(\Omega)$ does not depend on the open set $U$ such that $\Omega = \Gamma_0 \cap U$. We first prove:
\begin{lm}
Let $U, \widetilde{U}$ be two open subsets of $M$ such that $\Omega = U \cap \Gamma_0 = \widetilde{U} \cap \Gamma_0$. Then there exists an isomorphism between $\text{Sol}(U)$ and $\text{Sol}(\widetilde{U})$ commuting with the restriction maps. 
\end{lm}
\begin{proof}
Assume that $U$ and $\widetilde{U}$ are distinct and set $V = U \cap \widetilde{U}$; of course $\Omega \subset V$. Write $\widetilde{U} = V \cup W$ where the open set $W$ is such that there exists an open set $X \subset V$ containing $\Omega$ such that $W \cap X = \emptyset$. Let $\chi_V,\chi_W$ be a partition of unity subordinate to $\widetilde{U} = V \cup W$; in particular, $\chi_V(x) = 1$ whenever $x \in X$. One can show that the class $F_{\chi_V}(\psi_k) = [\chi_V \psi_k]$ belongs to $\text{Sol}(\widetilde{U})$. We claim that the map $F_{\chi_V}$ is an isomorphism with the required property.  
\end{proof}
From these two lemmas, we deduce the:
\begin{prop}
Let $U, \widetilde{U}$ be two open subsets of $M$ such that $\Omega = U \cap \Gamma_0 = \widetilde{U} \cap \Gamma_0$. Then $\mathfrak{F}_{U}(\Omega) = \mathfrak{F}_{\widetilde{U}}(\Omega)$.
\end{prop}
This allows to define a sheaf $\mathfrak{F} \rightarrow \Gamma_0$, which will be called the sheaf of microlocal solutions over $\Gamma_0$. Let us point out that so far, we have made no assumption on the structure (regularity) of the level $\Gamma_0$. 

\subsection{Regular case}

Consider a point $m \in \Gamma_0$ which is regular for the principal symbol $f_0$. Then there exists a symplectomorphism $\chi$ between a neighbourhood of $m$ in $M$ and a neighbourhood of the origin in $\R^2$ such that $(f_0 \circ \chi^{-1})(x,\xi) = \xi$. We can quantize this symplectomorphism by means of a Fourier integral operator: there exists an admissible sequence of operators $U_k^{(m)}: \classe{\infty}{(\R^2,L_0^k)} \to \classe{\infty}{(M,L^k \otimes K)}$ such that
\begin{equation*} U_k^{(m)} \left(U_k^{(m)}\right)^* \sim \Pi_k \quad \text{near} \ m; \ \left(U_k^{(m)}\right)^* U_k^{(m)} \sim \Pi_k^0, \quad \left(U_k^{(m)}\right)^*T_{k}U_k^{(m)} \sim S_{k} \quad \text{near} \ 0, \end{equation*} 
where $S_k$ is the Toeplitz operator 
\begin{equation*} S_{k} = \frac{i}{\sqrt{2}} \left( z - \frac{1}{k}\frac{d}{dz} \right). \end{equation*}
Consider the element $\Phi_k$ of $\classe{\infty}{(\R^2,L_0^k)}$ given by
\begin{equation*} \Phi_k(z) = \exp\left(kz^2/2 \right) \psi^k(z), \qquad \psi(z) = \exp\left( -\frac{1}{2} |z|^2 \right); \end{equation*}
it satisfies $S_k \Phi_k = 0$. Choosing a suitable cutoff function $\eta$ and setting $\Phi_{k}^{(m)} = \Pi_{k}^0 (\eta \Phi_k)$, we obtain an admissible sequence $\Phi_{k}^{(m)}$ of elements of $\Barg_k$ microlocally equal to $\Phi_k$ near the origin and generating the $\C_k$-module of microlocal solutions of $S_k u_k = 0$ near the origin.
\begin{prop}
\label{prop:dimsol}
The $\C_k$-module of microlocal solutions of equation (\ref{eq:micsol}) near $m$ is free of rank 1, generated by $U_k^{(m)} \Phi_k^{(m)}$.
\end{prop}
This is a slightly modified version of proposition $3.6$ of \cite{Cha4}, in which the normal form is achieved on the torus instead of the complex plane.

Thus, if $\Gamma_0$ contains only regular points of the principal symbol $f_0$, then $\mathfrak{F} \rightarrow \Gamma_0$ is a sheaf of free $\C_k$-modules of rank 1; in particular, this implies that $\mathfrak{F} \rightarrow \Gamma_0$ is a flat sheaf, thus characterised by its \v{C}ech holonomy $\text{hol}_{\mathfrak{F}}$.

\subsection{Lagrangian sections}

In order to compute the holonomy $\text{hol}_{\mathfrak{F}}$, we have to understand the structure of the microlocal solutions. For this purpose, a family of solutions of particular interest is given by Lagrangian sections; let us define these.
Consider a curve $\Gamma \subset \Gamma_0$ containing only regular points, and let $j:\Gamma \rightarrow M$ be the embedding of $\Gamma$ into $M$. Let $U$ be an open set of $M$ such that $U_{\Gamma} = j^{-1}(U)$ is contractible; there exists a flat unitary section $t_{\Gamma}$ of $j^*L \rightarrow U_{\Gamma}$. Now, consider a formal series
\begin{equation*} \sum_{\ell \geq 0} \hbar^{\ell} g_{\ell} \in \classe{\infty}{(U_{\Gamma},j^*K)}[[\hbar]]. \end{equation*}
Let $V$ be an open set of $M$ such that $\overline{V} \subset U$. Then a sequence $\Psi_{k} \in \Hil_{k}$ is a \emph{Lagrangian section} associated to $(\Gamma,t_{\Gamma})$ with symbol $\sum_{\ell \geq 0} \hbar^{\ell} g_{\ell}$ if 
\begin{equation*} \Psi_{k}(m) = \left( \frac{k}{2\pi} \right)^{1/4} F^k(m) \tilde{g}(m,k) \ \text{over} \ V, \end{equation*}
where 
\begin{itemize}
\item F is a section of $L \rightarrow U$ such that 
\begin{equation*} j^*F = t_{\Gamma} \quad \text{and} \quad \bar{\partial} F=0 \end{equation*}
modulo a section vanishing to every order along $j(\Gamma)$, and $|F(m)| < 1$ if $m \notin j(\Gamma)$,
\item $\tilde{g}(.,k)$ is a sequence of $\classe{\infty}{(U,K)}$ admitting an asymptotic expansion $\sum_{\ell \geq 0} k^{-\ell} \tilde{g}_{\ell}$ in the $\classe{\infty}{}$ topology such that 
\begin{equation*} j^* \tilde{g}_{\ell} = g_{\ell} \quad \text{and} \quad \bar{\partial}\tilde{g}_{\ell} = 0 \end{equation*}
modulo a section vanishing at every order along $j(\Gamma)$.
\end{itemize}
Assume furthermore that $\Psi_{k}$ is admissible in the sense that $\Psi_{k}(m)$ is uniformly $O(k^N)$ for some $N$ and the same holds for its successive covariant derivatives. It is possible to construct such a section with given symbol $\sum_{\ell \geq 0} \hbar^{\ell} g_{\ell}$ (see \cite[part $3$]{Cha2}). Furthermore, if $\Psi_k$ is a non-zero Lagrangian section, then the constants $c_k \in \C_k$ such that $c_k \Psi_k$ is still a Lagrangian section are the elements of the form
\begin{equation} c_k = \rho(k) \exp(ik \phi(k)) + O(k^{-\infty}) \label{eq:factLag}\end{equation}
where $\rho(k), \phi(k) \in \R$ admit asymptotic expansions of the form $\rho(k) = \sum_{\ell \geq 0} k^{-\ell} \rho_{\ell}$, $\phi(k) = \sum_{\ell \geq 0} k^{-\ell} \phi_{\ell}$.

Lagrangian sections are important because they provide a way to construct microlocal solutions. Indeed, if $\Psi_k$ is a Lagrangian section over $V$ associated to $(\Gamma,t_{\Gamma})$ with symbol $\sum_{\ell \geq 0} \hbar^{\ell} g_{\ell}$, then $T_{k}\Psi_{k}$ is also a Lagrangian section over $V$ associated to $(\Gamma,t_{\Gamma})$, and one can in principle compute the elements $\hat{g}_\ell$, $\ell \geq 0$ of the formal expansion of its symbol as a function of the $g_{\ell}$, $\ell \geq 0$ (by means of a stationary phase expansion). This allows to solve equation (\ref{eq:micsol}) by prescribing the symbol of $\Psi_k$ so that for every $\ell \geq 0$, $\hat{g}_\ell$ vanishes. Let us detail this for the two first terms. 

Introduce a \textit{half-form bundle} $(\delta,\varphi)$, that is a line bundle $\delta \rightarrow M$ together with an isomorphism of line bundles $\varphi:\delta^2 \rightarrow \Lambda^{2,0} T^*M$. Since the first Chern class of $M$ is even, such a couple exists. Introduce the Hermitian holomorphic line bundle $L_{1}$ such that $K = L_{1} \otimes \delta$. Define the \textit{subprincipal form} $\kappa$ as the 1-form on $\Gamma$ such that 
\begin{equation*}  \kappa(X_{f_{0}}) = - f_{1} \end{equation*}
where $X_{f_{0}}$ stands for the Hamiltonian vector field associated to $f_{0}$. Introduce the connection $\nabla^{1}$ on $j^*L_{1} \rightarrow \Gamma$ defined by
\begin{equation*} \nabla^1 = \nabla^{j^*L_{1}} + \frac{1}{i} \kappa, \end{equation*}
with $\nabla^{j^*L_{1}}$ the connection induced by the Chern connection of $L_1$ on $j^*L_{1}$. Let $\delta_{\Gamma}$ be the restriction of $\delta$ to $\Gamma$; the map
\begin{equation*} \varphi_{\Gamma}: \delta_{\Gamma}^2  \rightarrow  T^*\Gamma \otimes \C, \quad u  \mapsto  j^*\varphi(u) \end{equation*}
is an isomorphism of line bundles. Define a connection $\nabla^{\delta_{\Gamma}}$ on $\delta_{\Gamma}$ by
\begin{equation*} \nabla_{X}^{\delta_{\Gamma}} \sigma = \mathcal{L}_{X}^{\delta_{\Gamma}} \sigma, \end{equation*}
where $\mathcal{L}_{X}^{\delta_{\Gamma}}$ is the first-order differential operator acting on sections of $\delta_{\Gamma}$ such that 
\begin{equation*} \varphi_{\Gamma}\left( \mathcal{L}_{X}^{\delta_{\Gamma}} g \otimes g \right) = \frac{1}{2} \mathcal{L}_{X} \varphi_{\Gamma}\left(  g^{\otimes 2}\right) \end{equation*}
for every section $g$; here, $\mathcal{L}$ stands for the standard Lie derivative of forms. 

Then $T_{k}\Psi_{k}$ is a Lagrangian section over $V$ associated to $t_{\Gamma}$ with symbol $(j^*f_{0})g_{0} + O(\hbar) = O(\hbar)$, so $\Psi_k$ satisfies equation (\ref{eq:micsol}) up to order $O(k^{-1})$. Moreover, the subprincipal symbol of $T_{k}\Psi_{k}$ is
\begin{equation*} (j^*f_{1})g_{0} + \frac{1}{i}\left( \nabla_{X_{f_{0}}}^{j^*L_{1}} \otimes \text{Id} + \text{Id} \otimes \mathcal{L}_{X_{f_{0}}}^{\delta_{\Gamma}} \right)g_{0}. \end{equation*}
Consequently, equation (\ref{eq:micsol}) is satisfied by $\Psi_{k}$ up to order $O(k^{-2})$ if and only if 
\begin{equation} \left( f_{1} + \frac{1}{i}\left( \nabla_{X_{f_{0}}}^{j^*L_{1}} \otimes \text{Id} + \text{Id} \otimes \mathcal{L}_{X_{f_{0}}}^{\delta_{\Gamma}} \right) \right) g_{0} = 0 \quad \text{over} \ V \cap \Gamma. \label{eq:transport}\end{equation}
This can be interpreted as a parallel transport equation: if we endow $j^*L_{1} \otimes \delta_{\Gamma}$ with the connection induced from $\nabla^{1}$ and $\nabla^{\delta_{\Gamma}}$, equation (\ref{eq:transport}) means that $g_{0}$ is flat.

\subsection{Holonomy}
\label{subsect:holoreg}

We now assume that $\Gamma_0$ is connected (otherwise, one can consider connected components of $\Gamma_0$) and contains only regular points; it is then a smooth closed curve embedded in $M$. We would like to compute the holonomy of the sheaf $\mathfrak{F} \rightarrow \Gamma_0$. 
\begin{prop}
\label{prop:holreg}
The holonomy $\mathrm{hol}_{\mathfrak{F}}(\Gamma_0)$ is of the form
\begin{equation} \mathrm{hol}_{\mathfrak{F}}(\Gamma_0) = \exp(ik\Theta(k)) + O(k^{-\infty}) \end{equation}
where $\Theta(k)$ is real-valued and admits an asymptotic expansion of the form $\Theta(k) = \sum_{\ell \geq 0} k^{-\ell} \Theta_{\ell}$.
\end{prop}
In particular, this means that if we consider another set of solutions to compute the holonomy, we only have to keep track of the phases of the transition constants.
\begin{proof}
Cover $\Gamma_0$ by a finite number of open subsets $\Omega_{\alpha}$ in which the normal form introduced before proposition \ref{prop:dimsol} applies, and let $U_k^{\alpha}$ and $\Phi_k^{\alpha}$ be as in this proposition. We obtain a family $u_k^{\alpha}$ of microlocal solutions; observe that for each $\alpha$, $u_k^{\alpha}$ is a Lagrangian section associated to $\Gamma$. Hence, if $\Omega_{\alpha} \cap \Omega_{\beta}$ is non-empty, the unique (modulo $O(k^{-\infty})$) constant $c_k^{\alpha \beta} \in \C_k$ such that $u_k^{\alpha} = c_k^{\alpha \beta} u_k^{\beta}$ on $\Omega_{\alpha} \cap \Omega_{\beta}$ is of the form given in equation (\ref{eq:factLag}):
\begin{equation*} c_k^{\alpha \beta} = \rho^{\alpha \beta}(k) \exp(ik \phi^{\alpha \beta}(k)) + O(k^{-\infty}). \end{equation*}
But if $m$ belongs to $\Omega_{\alpha} \cap \Omega_{\beta}$, then near $m$ we have $u_k^{\alpha} \sim U_k^{\alpha} \Phi_k^{(m)}$ and $u_k^{\beta} \sim U_k^{\beta} \Phi_k^{(m)}$ where $\Phi_{k}^{(m)}$ is an admissible sequence of elements of $\Barg_k$ microlocally equal to $\Phi_k$ near the origin. Therefore, we have
\begin{equation*} c_k^{\alpha \beta} \Phi_{k}^{(m)} = (U_k^{\beta})^{-1} U_k^{\alpha} \Phi_{k}^{(m)} + O(k^{-\infty}), \end{equation*}
and the fact that the operators $U_k^{\alpha}$, $U_k^{\beta}$ are microlocally unitary yields $\left|c_k^{\alpha \beta}\right|^2 = 1 + O(k^{-\infty})$. This implies that for $\ell \geq 1$, $\rho_{\ell} = 0$, which gives the result.
\end{proof}

Let us be more specific and compute the first terms of this asymptotic expansion. 
Consider a finite cover $(\Omega_{\alpha})_{\alpha}$ of $\Gamma_0$ by contractible open subsets and endow each $\Omega_{\alpha}$ with a non-trivial microlocal solution $\Psi_k^{\alpha}$ which is a Lagrangian section. Choose a flat unitary section $t_{\alpha}$ of the line bundle $j^*L \rightarrow j^{-1}(\Omega_{\alpha})$ and write, for $m \in \Omega_{\alpha}$: 
\begin{equation*} \Psi_k^{\alpha}(m) = \left( \frac{k}{2\pi} \right)^{1/4} g_{\alpha}(m,k) t_{\alpha}^k(m) \end{equation*}
where the section $g_{\alpha}(.,k)$ of $j^*K \rightarrow \Omega_{\alpha}$ is the symbol of $\Psi_k^{\alpha}$, whose principal symbol will be denoted by $g_{\alpha}^{(0)}$. Now, assume that $\Omega_{\alpha} \cap \Omega_{\beta} \neq \emptyset$; there exists a unique (up to $O(k^{-\infty})$) $c_k^{\alpha \beta} \in \C_k$ such that $\Psi_k^{\alpha} \sim c_k^{\alpha \beta} \Psi_k^{\beta}$ on $\Omega_{\alpha} \cap \Omega_{\beta}$.
\begin{dfn}
Let $A, B \in M$ and $\gamma$ be a piecewise smooth curve joining $A$ and $B$; denote by $P_{A,B,\gamma}: L_A \to L_B$ the linear isomorphism given by parallel transport from $A$ to $B$ along $\gamma$. Given two sections $s,t$ of $L \rightarrow M$, define the \emph{phase difference} between $s(A)$ and $t(B)$ along $\gamma$ as the number
\begin{equation*} \left(\Phi_{s}(A) - \Phi_{t}(B)\right)_{\gamma} =  \arg(\lambda_{A,B,\gamma}) - c_0([A,B]) \in \R/2\pi\Z, \end{equation*}
where $\lambda_{A,B,\gamma}$ is the unique complex number such that $P_{A,B,\gamma}(s(A)) = \lambda_{A,B,\gamma} t(B)$ and $c_0([A,B])$ is the (phase of the) holonomy of $\gamma$ in $(L,\nabla)$. Define in the same way the phase difference for two sections of $K \rightarrow M$, using the Chern connection of $K$.
\end{dfn} 
Now, consider three points $A,B,C \in M$ and let $\gamma_{1}$ (resp. $\gamma_{2}$ be a piecewise smooth curve joining $A$ and $B$ (resp. $B$ and $C$). Let $\gamma$ be the concatenation of $\gamma_{1}$ and $\gamma_{2}$. It is easily checked that 
\begin{equation*} \left(\Phi_{s}(A) - \Phi_{t}(B)\right)_{\gamma_{1}} + \left(\Phi_{t}(B) - \Phi_{u}(C)\right)_{\gamma_{2}} = \left(\Phi_{s}(A) - \Phi_{u}(C)\right)_{\gamma} \end{equation*}
for three sections $s,t,u$ of $L$. Furthermore, if $\gamma$ is a closed curve and $A$ is a point on $\gamma$, then the phase difference between $s(A)$ and $s(A)$ along $\gamma$ is
\begin{equation*} \left(\Phi_{s}(A) - \Phi_{s}(A)\right)_{\gamma} = 0 \end{equation*}
by definition of the holonomy $c_{0}$. This is why we write this number as a difference.

Coming back to our problem, denote by $\Phi_{\alpha}^{(-1)}(A) - \Phi_{\beta}^{(-1)}(B)$ the phase difference between $t_{\alpha}(A)$ and $t_{\beta}(B)$ along $\Gamma_0$ in $L$, and by $\Phi_{\alpha}^{(0)}(A) - \Phi_{\beta}^{(0)}(B)$ the phase difference between $g_{\alpha}^{(0)}(A)$ and $g_{\beta}^{(0)}(B)$ along $\Gamma_0$ in $K$. Let $\zeta$ be the path in $\Gamma_0$ starting at a point $A \in \Omega_{\alpha}$ and ending at $B \in \Omega_{\alpha} \cap \Omega_{\beta}$. Since $t_{\alpha}$ is flat and the principal symbol $g_0$ of $\Psi_k^{\alpha}$ satisfies equation (\ref{eq:transport}), we have
\begin{equation*}\begin{split} \arg\left(c_k^{\alpha \beta}\right) = k\left( c_0(\zeta) + \Phi_{\alpha}^{(-1)}(A) - \Phi_{\beta}^{(-1)}(B) \right) \\ + c_1(\zeta) + \text{hol}_{\delta_{\Gamma_0}}(\zeta) + \Phi_{\alpha}^{(0)}(A) - \Phi_{\beta}^{(0)}(B) \ + O(k^{-1}).\end{split}\end{equation*}

Thanks to the discussion above, we know that the term $k \left(\Phi_{\alpha}^{(-1)}(A) - \Phi_{\beta}^{(-1)}(B) \right) + \Phi_{\alpha}^{(0)}(A) - \Phi_{\beta}^{(0)}(B)$ is a \v{C}ech coboundary. Let us call $c_1(\Gamma_0)$ the holonomy of $\Gamma_0$ in $(L_1,\nabla^1)$. One can check that $\nabla^{\delta_{\Gamma_0}}$ has holonomy in $\Z/2\Z$, represented by $\epsilon(\Gamma_0) \in \{0,1\}$. We obtain:
\begin{prop}
The first two terms of the asymptotic expansion of the quantity $\Theta(k)$ defined in proposition \ref{prop:holreg} are given by
\begin{equation*} \Theta_{0} = c_{0}(\Gamma_0) \end{equation*}
and
\begin{equation*} \Theta_{1} =  c_{1}(\Gamma_0) + \epsilon(\Gamma_0) \pi.\end{equation*}
\end{prop}
Since one can construct a non-trivial microlocal solution over $\Gamma_0$ if and only if $\Theta(k) \in 2\pi\Z$, we recover the usual Bohr-Sommerfeld conditions. 

Let us give another interpretation of the index $\epsilon$. Consider a smooth closed curve $\gamma$ immersed in $M$. Denote by $\iota:\gamma \rightarrow M$ this immersion, and by $\delta_{\gamma} = \iota^*\delta$ the pullback bundle over $\gamma$. Let $\tilde{\iota}:\delta_{\gamma} \to \delta$ be the natural lift of $\iota$, and define $\tilde{\iota}^2: \delta_{\gamma}^2 \to \delta^2$ by the formula $\tilde{\iota}^2(u \otimes v) = \tilde{\iota}(u) \otimes \tilde{\iota}(v)$. The map
\begin{equation*} \varphi_{\gamma}: \delta_{\gamma}^2  \rightarrow  T^*\gamma \otimes \C, \quad u  \mapsto  \iota^*\varphi(\tilde{\iota}^2(u)) \end{equation*}
is an isomorphism of line bundles. The set
\begin{equation*} \left\{ u \in \delta_{\gamma}; \varphi_{\gamma}(u^{\otimes 2}) > 0 \right\} \end{equation*}
has one or two connected components. In the first case, we set $\epsilon(\gamma) = 1$, and in the second case $\epsilon(\gamma) = 0$. One can check that this definition coincides with the one above when $\gamma$ is a smooth embedded closed curve. Notice that the value of $\epsilon(\gamma)$ only depends on the isotopy class of $\gamma$ in $M$.

\subsection{Spectral parameter dependence}
\label{subsect:BSreg}

For spectral analysis, one has to do the same study as above replacing the operator $T_k$ with $T_k - E$; then it is natural to ask if the previous study can be done taking into account the dependence on the spectral parameter $E$. 

Assume that there exists a tubular neighbourhood $\Omega$ of $\Gamma$ such that for $E$ close enough to $0$, the intersection $\Gamma_{E} \cap \Omega$ is regular. Then we can construct microlocal solutions of $(T_{k} - E)u_{k} = 0$ as Lagrangian sections depending smoothly on a parameter (see \cite[section $2.6$]{Cha4}); these solutions are uniform in $E$. We can then define all the previous objects with smooth dependence in $E$. Proceeding this way, we obtain the parameter dependent Bohr-Sommerfeld conditions, that we describe below.

Let $I$ be an interval of regular values of the principal symbol $f_{0}$ of the operator. For $E \in I$, denote by $\mathcal{C}_{j}(E)$, $1 \leq j \leq N$, the connected components of $f_{0}^{-1}(E)$ in such a way that $E \mapsto \mathcal{C}_{j}(E)$ is smooth. Observe that $\mathcal{C}_{j}(E)$ is a smooth embedded closed curve, endowed with the orientation depending continuously on $E$ given by the Hamiltonian flow of $f_0$. Define the \textit{principal action} $c_{0}^{(j)} \in \classe{\infty}(I)$ in such a way that the parallel transport in $L$ along $\mathcal{C}_{j}(E)$ is the multiplication by $\exp(i c_{0}^{(j)}(E))$. Define the \textit{subprincipal action} $c_{1}^{(j)}$ in the same way, replacing $L$ by $L_{1}$ and using the connection $\nabla^1$ (depending on $E$) described above. Finally, set $\epsilon_{E}^{(j)} = \epsilon(\mathcal{C}_{j}(E))$; in fact, $\epsilon_{E}^{(j)}$ is a constant $\epsilon_{E}^{(j)} = \epsilon^{(j)}$ for $E$ in $I$. The Bohr-Sommerfeld conditions (see \cite{Cha2} for more details) state that there exists $\eta > 0$ such that the intersection of the spectrum of $T_{k}$ with $[E-\eta,E+\eta]$ modulo $O(k^{-\infty})$ is the union of the spectra $\sigma_{j}$, $1 \leq j \leq N$, where the elements of $\sigma_{j}$ are the solutions of 
\begin{equation*} g^{(j)}(\lambda,k) \in 2\pi k^{-1} \Z \label{eq:BSreg}\end{equation*}
where $g^{(j)}(.,k)$ is a sequence of functions of $\classe{\infty}{(I)}$ admitting an asymptotic expansion 
\begin{equation*} g^{(j)}(.,k) = \sum_{\ell = 0}^{+ \infty} k^{-\ell} g_{\ell}^{(j)} \end{equation*}
with coefficients $ g_{\ell}^{(j)} \in \classe{\infty}{(I)}$. Furthermore, one has
\begin{equation*} g_{0}^{(j)}(\lambda) = c_{0}^{(j)}(\lambda) \ \text{and} \ g_{1}^{(j)}(\lambda) = c_{1}^{(j)}(\lambda) + \epsilon^{(j)} \pi. \end{equation*}

\section{Microlocal normal form}
\label{section:formenormale}

\subsection{Normal form on the Bargmann side}

Let $P_{k}$ be the operator defined by $P_{k} =  \frac{i}{2} \left(z^2 - \frac{1}{k^2} \dpar{^2}{z^2}   \right)$ with domain $\C[z] \subset \Barg_{k}$; it is a Toeplitz operator with normalized symbol $p_{0}(x,\xi) = x\xi$. We will use this operator to understand the behaviour of $A_{k}$ near each $s_{j}$, $1 \leq j \leq n$. More precisely, we study the operator $A_{k} - E$, where $E \in \R$ is allowed to vary in a neighbourhood of zero. 

Let $j \in \llbracket 1,n \rrbracket$. The isochore Morse lemma \cite{CdV1} yields a symplectomorphism $\chi_{E}$ from a neighbourhood of $s_{j}$ in $M$ to a neighbourhood of the origin in $\R^2$, depending smoothly on $E$, and a smooth function $g_{j}^E$, again depending smoothly on $E$, such that 
\begin{equation*} ((a_{0} -E)\circ \chi_{E}^{-1})(x,\xi) = g_{j}^E(x\xi) \end{equation*}
and $(g_{j}^E)'(0) \neq 0$. Using a Taylor formula, one can write
\begin{equation*} g_{j}^E(t) = w_{j}^E(t)\left(t - f_{j}(E)\right) \end{equation*}
with $w_{j}^E$ smooth, depending smoothly on $E$, and such that $w_{j}^E(0) \neq 0$, and $f_{j}$ a smooth function of $E$. This symplectic normal form can be quantized to the following semiclassical normal form.
\begin{prop}
\label{prop:fnparam}
Fix $j \in \llbracket 1,n \rrbracket$. There exist a smooth function $f_{j}$, a Fourier integral operator $U_{k}^E: \Barg_{k} \rightarrow \Hil_{k}$, a Toeplitz operator $W_{k}^{E}$ elliptic at $0$ and a sequence of smooth functions $\varepsilon_{j}(.,k)$ admitting an asymptotic expansion $\varepsilon_{j}(E,k) = \sum_{\ell = 0}^{+\infty} k^{-\ell} \varepsilon_{j}^{(\ell)}(E)$ such that 
\begin{equation*} (U_{k}^{E})^* (A_{k} - E) U_{k}^E \sim W_{k}^{E} \left(P_{k} - f_{j}(E) - k^{-1}\varepsilon_{j}(E,k) \right) \end{equation*}
microlocally near $s_{j}$. Furthermore,
\begin{itemize}
\item $U_{k}$ and $W_{k}$ depend smoothly on $E$,
\item $f_{j}(E)$ is the value of $x\xi$ whenever $(x,\xi) = \chi_{E}(m)$ for $m \in \Gamma_{E}$,
\item and the first term of the asymptotic expansion of $\varepsilon_{j}(0,k)$ is given by
\begin{equation*} \varepsilon_{j}^{(0)}(0) = \frac{- a_{1}(s_{j})}{|\det(\mathrm{Hess}(a_{0})(s_{j}))|^{1/2}}, \end{equation*}
where $\mathrm{Hess}(a_{0})(s_{j})$ is the Hessian of $a_{0}$ at $s_{j}$. 
\end{itemize}
\end{prop}
The proof is an adaptation of the one in \cite[section 3]{CdV3} to the Toeplitz setting; see also \cite{LF} for a similar result in the elliptic case.

\subsection{Link with the pseudodifferential setting}

Now, we use the Bargmann transform to understand the structure of the space of microlocal solutions of $P_{k} - E =0$. 
\begin{lm}
For $u \in \Sch(\R)$, one has
\begin{equation*} B_{k}^*P_{k}B_{k}u = \frac{1}{ik} \left( x \partial_{x} + 1  \right) u. \end{equation*}
\end{lm}
From now on, we will denote by $S_{k}$ the pseudodifferential operator $\frac{1}{ik} \left( x \partial_{x} + 1  \right)$.
This correspondence will allow us to understand the space of microlocal solutions of $P_{k} - E$ on a neighbourhood of the origin. Let us recall the results of \cite{CdV3,CdV4} that will be useful to our study. 
\begin{prop}[{\cite[proposition $3$]{CdV3}}]
\label{prop:basesPseudo}
Let $E$ be such that $|E| < 1$. The space of microlocal solutions of $(S_{k} - E)u_{k} = 0$ on $Q = [-1,1]^2$ is a free $\C_{k}$-module of rank $2$. 
\end{prop}
Moreover, we know two bases of this module. Indeed, the tempered distributions $v_{k,E}^{(j)}$, $j \in \llbracket 1,4 \rrbracket$ defined as:
\begin{equation*} v_{k,E}^{(1),(2)}(x) = \mathbf{1}_{\R^{\pm*}}(x) \exp\left( (-\tfrac{1}{2} + ikE) \ln(|x|)  \right), \end{equation*}
\begin{equation*} \hat{v}_{k,E}^{(3),(4)}(\xi) = \mathbf{1}_{\R^{\pm*}}(\xi) \exp\left(( -\tfrac{1}{2} + ikE ) \ln(|\xi|)  \right). \end{equation*}
are exact solutions of the equation $(S_{k} - E) v_{k,E}^{(j)} = 0$; better than that, the couple $(v_{k,E}^{(1)},v_{k,E}^{(2)})$ (resp. $(v_{k,E}^{(3)},v_{k,E}^{(4)})$) forms a basis of the space of solutions of this equation. Now, choose a compactly supported function $\chi \in \classe{\infty}{(\R)}$ with constant value $1$ on $I$ and vanishing outside $2I$. Define the pseudodifferential operator $\Pi_{Q}$ by 
\begin{equation*} \Pi_{Q}u(x) = \frac{k}{2\pi} \int_{\R^2} \exp\left( ik(x-y)\xi \right) \chi(\xi) \chi(y) u(y) \ dy d\xi. \end{equation*}
Then $\Pi_{Q}$ maps $\mathscr{D}'(\R)$ into $\Sch(\R)$, and $\Pi_{Q} \sim \text{Id}$ on $Q$. Set
\begin{equation*} w_{k,E}^{(j)} = \Pi_{Q} v_{k}^{(j)}; \end{equation*}
then the $w_{k,E}^{(j)}$, $j \in \llbracket 1,4 \rrbracket$, belong to $\Sch(\R)$, and are microlocal solutions of $(S_{k} - E) w_{k,E}^{(j)} = 0$ on $Q$. The matrix of the change of basis from $(w_{k,E}^{(3)},w_{k,E}^{(4)})$ to $(w_{k,E}^{(1)},w_{k,E}^{(2)})$ is given by
\begin{equation} M_{k} = \mu_{k}(E) \begin{pmatrix} 1 & -i\exp(-i \pi kE ) \\ -i\exp(-i \pi kE ) & 1 \end{pmatrix} + O(k^{-\infty}) \label{eq:matrice}\end{equation}
with
\begin{equation*} \mu_{k}(E) = \frac{1}{\sqrt{2\pi}} \Gamma\left( \tfrac{1}{2} -ikE \right) \exp\left( \frac{\pi}{4} \left( 2kE + i \right) - ikE\ln(k) \right).  \end{equation*}

\subsection{Microlocal solutions of $(P_{k} - E)u_{k} = 0$}
\label{subsect:microPk}

Now, consider the Bargmann transforms of the sequences $w_{k,E}^{(j)}$: $u_{k,E}^{(j)} = B_{k} w_{k,E}^{(j)}$. Propositions \ref{prop:basesPseudo} and \ref{prop:microBarg} yield the:
\begin{prop}
\label{prop:dim2Toep}
For $E$ such that $|E| < 1$, the space of microlocal solutions of $(P_{k} - E) u_{k} = 0$ on $Q = [-1,1]^2 \subset \C$ is a free $\C_{k}$-module of rank $2$. Moreover, the couples $\left(u_{k,E}^{(1)},u_{k,E}^{(2)}\right)$ and $\left(u_{k,E}^{(3)},u_{k,E}^{(4)}\right)$ are two bases of this module; the transfer matrix is given by equation (\ref{eq:matrice}). 
\end{prop}
\rmq The sections $u_{k,E}^{(j)}$, $j=1,\ldots,4$, can be written in terms of parabolic cylinder functions. In the article \cite{NonVo}, Nonnenmacher and Voros studied these functions in order to understand the behaviour of the generalized eigenfunctions of $P_{k}$; the result of this subtle analysis, based on Stokes lines techniques, was not exactly what we needed here, and this is partly why we chose to use the microlocal properties of the Bargmann transform instead.

\section{Bohr-Sommerfeld conditions}
\label{sect:BS}

To obtain the Bohr-Sommerfeld conditions, we will recall the reasoning of Colin de Verdi\`ere and Parisse \cite{CdVP}, and will also refer to the work of Colin de Verdi\`ere and V{\~u} Ng{\d{o}}c \cite{SanCdV}. Since the general approach is the same, we only recall the main ideas and focus on what differs in the Toeplitz setting.

\subsection{The sheaf of standard basis}

As in section \ref{sect:regular}, introduce the sheaf $(\mathfrak{F}, \Gamma_{0})$ of microlocal solutions of $A_k \psi_k = 0$ over $\Gamma_0$; we recall that a global non-trivial microlocal solution corresponds to a global non-trivial section of this sheaf. However, since the topology of $\Gamma_0$ is much more complicated than in the regular case, the condition for the existence of such a section is not as simple as saying that a holonomy must be trivial. In particular, we have to handle what happens at critical points. To overcome this difficulty, the idea is to introduce a new sheaf over $\Gamma_0$ that will contain all the information we need to construct a global non-trivial microlocal solution; roughly speaking, this new sheaf can be thought of as the limit of the sheaf $\mathfrak{F} \rightarrow \Gamma_E$ of microlocal solutions over regular levels as $E$ goes to 0.     

Following Colin de Verdi\`ere and Parisse \cite{CdVP}, we introduce a sheaf $(\mathfrak{L},\Gamma_{0})$ of free $\C_{k}$-modules of rank one over $\Gamma_{0}$ as follows: to each point $m \in \Gamma_{0}$, associate the free module $\mathfrak{L}(m)$ generated by \textit{standard basis} at $m$. If $m$ is a regular point, a standard basis is any basis of the space of microlocal solutions near $m$. At a critical point $s_{j}$, we define a standard basis in the following way. The $\C_{k}$-module of microlocal solutions near $s_{j}$ is free of rank 2; moreover, it is the graph of a linear application. Indeed, number the four local edges near $s_{j}$ with cyclic order $1, 3, 2, 4$, so that the edges $e_{1}, e_{2}$ are the ones that leave $s_{j}$. Let us denote by $\text{Sol}(e_{1}e_{2})$ (resp. $\text{Sol}(e_{3}e_{4})$) the module of microlocal solutions over the disjoint union of the local unstable (resp. stable) edges $e_{1}, e_{2}$ (resp. $e_{3}, e_{4}$). 
$\text{Sol}(e_{1}e_{2})$ and $\text{Sol}(e_{3}e_{4})$ are free modules of rank 2, and there exists a linear map $T_{j}: \text{Sol}(e_{3}e_{4}) \rightarrow \text{Sol}(e_{1}e_{2})$ such that $u$ is a solution near $s_{j}$ if and only if its restrictions satisfy $u_{|\text{Sol}(e_{1}e_{2})} = T_{j} u_{|\text{Sol}(e_{3}e_{4})}$. Equivalently, given two solutions on the entering edges, there is a unique way to obtain two solutions on the leaving edges by passing the singularity. One can choose a basis element for each $\mathfrak{F}(e_{i})$, $i \in \llbracket 1,4 \rrbracket$, and express $T_{j}$ as a $2 \times 2$ matrix (defined modulo $O(k^{-\infty})$); one can show that the entries of this matrix are all non-vanishing. An argument of elementary linear algebra shows that once the matrix $T_{j}$ is chosen, the basis elements of the modules $\mathfrak{F}(e_{i})$ are fixed up to multiplication by the same factor. Moreover, we saw that there exists a choice of basis elements such that $T_{j}$ has the following expression:
\begin{equation} T_{j} = \exp\left(-\frac{i\pi}{4}\right) \mathcal{E}_{k}(\varepsilon_{j}(0,k)) \begin{pmatrix} 1 & i \exp(-\pi \varepsilon_{j}(0,k)) \\ i \exp(-\pi \varepsilon_{j}(0,k)) & 1 \end{pmatrix},  \label{eq:Tj}\end{equation}
where 
\begin{equation} \mathcal{E}_{k}(t) = \frac{1}{\sqrt{2\pi}} \Gamma\left( \frac{1}{2} + it \right) \exp\left(t\left( \frac{\pi}{2} - i \ln k \right)\right). \label{eq:funcEk}\end{equation}
This allows us to call the choice of the basis elements of $\mathfrak{F}(e_{i})$ a standard basis whenever $T_{j}$ is given by equation (\ref{eq:Tj}). 

$(\mathfrak{L},\Gamma_{0})$ is a locally free sheaf of rank one $\C_{k}$-modules, and its transition functions are constants. Hence, it is flat, thus characterised by its holonomy
\begin{equation*} \text{hol}_{\mathfrak{L}}: \text{H}_{1}(\Gamma_{0}) \rightarrow \C_{k}. \end{equation*}
In terms of \v{C}ech cohomology, if $\gamma$ is a cycle in $\Gamma_{0}$, and $\Omega_{1}, \ldots, \Omega_{\ell}$ is an ordered sequence of open sets covering the image of $\gamma$, each $\Omega_{i}$ being equipped with a standard basis $u_{i}$, then 
\begin{equation} \text{hol}_{\mathfrak{L}}(\gamma) = x_{1,2} \ldots x_{\ell - 1, \ell} x_{\ell,1}, \end{equation}
where $x_{i,j} \in \C_{\hbar}$ is such that $u_{i} = x_{i,j} u_{j}$ on $\Omega_{i} \cap \Omega_{j}$.

Now, cut $n + 1$ edges of $\Gamma_{0}$, each one corresponding to a cycle $\gamma_{i}$ in a basis $(\gamma_{1}, \ldots, \gamma_{n + 1})$ of $\text{H}_{1}(\Gamma_{0},\Z)$, in such a way that the remaining graph is a tree $T$. Then the sheaf $(\mathfrak{L},T)$ has a non-trivial global section. The conditions to obtain a non-trivial global section of the sheaf $(\mathfrak{F},\Gamma_{0})$ of microlocal solutions on $\Gamma_{0}$ are given in the following theorem. They were already present in the work of Colin de Verdi\`ere and Parisse in the case of pseudodifferential operators, but the fact that they extend to our setting is a consequence of the results obtained in the previous sections.
\begin{thm}
\label{thm:cdvp}
The sheaf $(\mathfrak{F}, \Gamma_{0})$ has a non-trivial global section if and only if the following linear system of $3n+1$ equations with $3n+1$ unknowns $(x_{\alpha} \in \C_{k})_{\alpha \in \{\mathrm{edges \ of \ } T\}}$ has a non-trivial solution:  
\begin{enumerate}
\item if the edges $(\alpha_{1}, \alpha_{2}, \alpha_{3}, \alpha_{4})$ connect at $s_{j}$ (with the same convention as before for the labeling of the edges), then
\begin{equation*} \begin{pmatrix} x_{\alpha_{3}} \\ x_{\alpha_{4}} \end{pmatrix} = T_{j} \begin{pmatrix} x_{\alpha_{1}} \\ x_{\alpha_{2}}, \end{pmatrix} \end{equation*}
\item if $\alpha$ and $\beta$ are the extremities of a cut cycle $\gamma_{i}$, then 
\begin{equation*} x_{\alpha} = \mathrm{hol}_{\mathfrak{L}}(\gamma_{i}) x_{\beta}, \end{equation*}
where the following orientation is assumed: $\gamma_{i}$ can be represented as a closed path starting on the edge $\alpha$ and ending on the edge $\beta$.
\end{enumerate}
\end{thm}
\begin{proof}
It follows from propositions \ref{prop:dimsol} and \ref{prop:dim2Toep} that the proof can be directly adapted from the one of \cite[theorem 2.7]{SanCdV}.
\end{proof}

\subsection{Singular invariants}
\label{subsect:singular}

Of course, in order to use this result, it remains to compute the holonomy $\text{hol}_{\mathfrak{L}}$. For this purpose, let us introduce some geometric quantities close from the ones used to express the regular Bohr-Sommerfeld conditions. Let $\gamma$ be a cycle in $\Gamma_{0}$, and denote by $s_{j_{m}}$, $m=1,\ldots,p$ the critical points contained in $\gamma$. 
\begin{dfn}[singular subprincipal action] 
\label{dfn:singact}
Decompose $\gamma$ as a concatenation of smooth paths and paths containing exactly one critical point; if $A$ and $B$ are the ordered endpoints of a path, we will call it $[A,B]$. Define the subprincipal action $\tilde{c}_{1}(\gamma)$ as the sum of the contributions of these paths, given by the following rules:
\begin{itemize}
\item if $[A,B]$ contains only regular points, its contribution to the singular subprincipal action is 
\begin{equation*} \tilde{c}_{1}([A,B]) = c_{1}([A,B]) \end{equation*}
as in the regular case,
\item if $[A,B]$ contains the singular point $s$ and is smooth at $s$, then
\begin{equation*} \tilde{c}_{1}([A,B]) =  \lim_{a,b \rightarrow s}  \left(  c_{1}([A,a]) + c_1([b,B])  \right) \end{equation*}
where $a$ (resp. $b$) lies on the same branch as $A$ (resp. $B$),
\item if $[A,B]$ contains the singular point $s$ and is not smooth at $s$, we set
\begin{equation} \tilde{c}_{1}([A,B]) =  \lim_{a,b \rightarrow s}  \left(  c_{1}([A,a]) + c_1([b,B]) \pm \varepsilon_{s}^{(0)} \ln \left| \int_{P_{a,b}} \omega \right| \right) \end{equation}
where $P_{a,b}$ is the parallelogram (defined in any coordinate system) built on the vectors $\overrightarrow{sa}$ and $\overrightarrow{sb}$, $\pm = +$ if $[A,B]$ is oriented according to the flow of $X_{a_{0}}$, $ \pm = -$ otherwise and
\begin{equation*} \varepsilon_{s}^{(0)} = \frac{- a_{1}(s) }{\left| \det(\text{Hess}(a_{0})(s)) \right|^{1/2}}, \end{equation*}
as before.                                                                                                                                         
\end{itemize}
\end{dfn}
\begin{figure}[H]
\begin{center}
\includegraphics[scale=0.8]{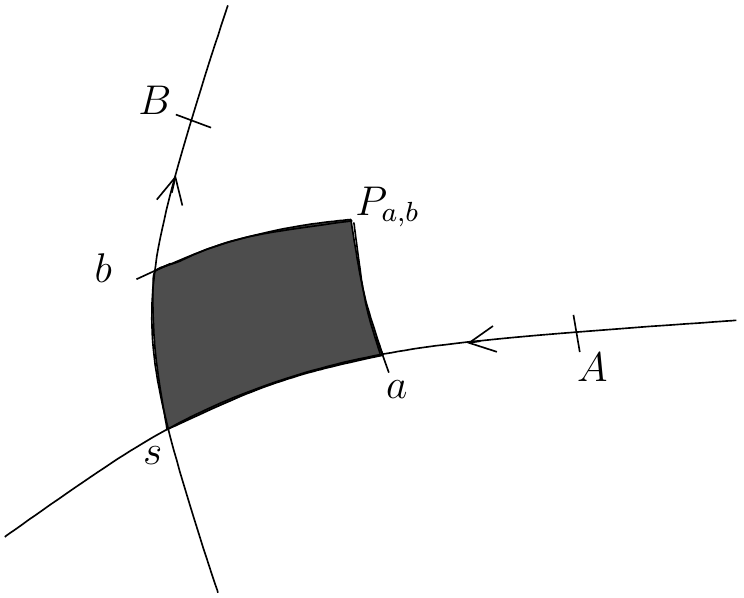}
\end{center}
\caption{Computation of $\tilde{c}_1([A,B])$}
\label{fig:Pab}
\end{figure}
\begin{dfn}[singular index] 
\label{dfn:singeps}
Let $(\gamma_{t})_{t}$ be a continuous family of immersed closed curves such that $\gamma_{0} = \gamma$ and $\gamma_{t}$ is smooth for $t > 0$. Then the function $t \mapsto \epsilon(\gamma_{t})$, $t > 0$, is constant; we denote by $\epsilon$ its value. We define the singular index $\tilde{\epsilon}(\gamma)$ by setting 
\begin{equation} \tilde{\epsilon}(\gamma) = \epsilon + \sum_{m=1}^p  \frac{\rho_{m}}{4} \end{equation}
where $\rho_m = 0$ if $\gamma$ is smooth at $s_{j_m}$, $\rho_{m} = +1$ if at $s_{j_{m}}$, $\gamma$ turns in the direct sense with respect to the cyclic order $(1,3,2,4)$ of the local edges, and $\rho_{m} = -1$ otherwise.
\end{dfn}
Observe that both $\tilde{c}_{1}$ and $\tilde{\epsilon}$ define $\Z$-linear maps on $\text{H}_{1}(\Gamma_{0},\Z)$. 
\begin{thm}
\label{thm:holsing}
Let $\gamma$ be a cycle in $\Gamma_{0}$. Then the holonomy $\text{hol}_{\mathfrak{L}}(\gamma)$ of $\gamma$ in $\mathfrak{L}$ has the form
\begin{equation} \mathrm{hol}_{\mathfrak{L}}(\gamma) = \exp(ik\theta(\gamma,k)) \end{equation}
where $\theta(\gamma,k)$ admits an asymptotic expansion in non-positive powers of $k$. Moreover, if we denote by $\theta(\gamma,k) = \sum_{\ell \geq 0} k^{-\ell} \theta_{\ell}(\gamma)$ this expansion, the first two terms are given by the formulas:
\begin{equation} \theta_{0}(\gamma) = c_{0}(\gamma), \quad \theta_{1}(\gamma) = \tilde{c}_{1}(\gamma) + \tilde{\epsilon}(\gamma) \pi \label{eq:invsing}.\end{equation}
\end{thm}
\begin{proof}
We just prove here that the holonomy has the claimed behaviour. It is enough to show that one can choose a finite open cover $(\Omega_{\alpha})_{\alpha}$ of $\gamma$ and a section $u_k^{\alpha}$ of $\mathfrak{L} \rightarrow \Omega_{\alpha}$ for which the transition constants $c_k^{\alpha \beta}$ have the required form. On the edges of $\gamma$, this follows from the analysis of section \ref{sect:regular}. At a vertex, we choose the standard basis $U_{k}^0 u_{k,\varepsilon_{j}(0,k)}^{(j)}$, where $u_{k,E}^{(j)}$ is defined in section \ref{subsect:microPk} and $U_{k}^E$ is the operator of proposition \ref{prop:fnparam}; to conclude, we observe that the restrictions of these sections to the corresponding edge are Lagrangian sections.
\end{proof}

\subsection{Computation of the singular holonomy}

This section is devoted to the proof of the second part of theorem \ref{thm:holsing}. We use the method of \cite{SanCdV}, but of course, our case is simpler, because in the latter, the authors investigated the case of singularities in (real) dimension 4 (for pseudodifferential operators). Let us work on microlocal solutions of the equation
\begin{equation} (A_{k} - E) u_{k} = 0 \label{eq:microE}\end{equation} 
where $E$ varies in a small connected interval $I$ containing the critical value $0$. The critical value separates $I$ into two open sets $I^+$ and $I^-$, with the convention $I^{\pm} = I \cap \{ \pm a_{0} > 0 \}$. Let $D^{\pm} = I^{\pm} \cup \{ 0 \}$, and let $\mathfrak{C}^{\pm}$ be the set of connected components of the open set $a_{0}^{-1}(I^{\pm})$. The smooth family of circles in the component $p^{\pm}$ is denoted by $\mathcal{C}_{p^{\pm}}(E)$, $E \in I^{\pm}$.

As in section \ref{sect:regular}, for $E \neq 0$, we denote by $(\mathfrak{F},\Gamma_{E})$ the sheaf of microlocal solutions of (\ref{eq:microE}) on $\Gamma_{E}$; remember that it is a flat sheaf of rank 1 $\C_k$-modules, characterised by its \v{C}ech holonomy $\text{hol}_{\mathfrak{F}}$. The idea is to let $E$ go to 0 and compare this holonomy to the holonomy of the sheaf $\mathfrak{L} \rightarrow \Gamma_0$.
\begin{dfn}
\label{def:holosing}
Near each critical point $s_{j}$, consider two families of points $A_{j}(E)$ and $B_{j}(E)$ in $\classe{\infty}{(D^{\pm}, \bar{p}^{\pm} \setminus \{s_{j}\})}$ lying on $\mathcal{C}_{p^{\pm}}(E)$ and such that $A_{j}(0)$ and $B_{j}(0)$ lie respectively in the stable or unstable manifold. Endow a small neighbourhood of $A_{j}$ (resp. $B_{j}$) with a microlocal solution $u_{A_{j}}$ (resp. $u_{B_{j}}$) of (\ref{eq:microE}) which is a Lagrangian section uniform in $E \in D^{\pm}$. Define the quantity $\Theta([A_{j}(E),B_{j'}(E)],k)$ as the phase of the \v{C}ech holonomy of the path $[A_{j}(E),B_{j'}(E)]$ joining $A_{j}(E)$ and $B_{j'}(E)$ in the sheaf $(\mathfrak{F},\Gamma_{E})$ computed with respect to $u_{A_{j}}$ and $u_{B_{j'}}$. 
\end{dfn}

Note that if we change the sections $u_{A_{j}}$ and $u_{B_{j'}}$, the phase of the holonomy is modified by an additive term admitting an asymptotic expansion in $k \, \classe{\infty}{(D^{\pm})}[[k^{-1}]]$. The singular behaviour of the holonomy is thus preserved; moreover, the added term is a \v{C}ech coboundary, and hence does not change the value of the holonomy along a closed path.

Then, we consider continuous families of paths $(\zeta_{E})_{E \in D^{\pm}}$ drawn on a circle $\mathcal{C}_{p^{\pm}}(E)$ and whose endpoints are some of the $A_{j}(E)$ and $B_{j'}(E)$ of the previous definition. We say that $\zeta_{E}$ is
\begin{itemize}
\item \emph{regular} if $\zeta_{0}$ does not contain any of the critical points $s_{j}$,
\item \emph{local} if $\zeta_{0}$ contains exactly one critical point,
\end{itemize}
and we consider only these two types of paths. We write $\zeta_{E} = [A_{j}(E),B_{j'}(E)]$. The following proposition implies that a path that is local in the above sense can always be assumed to be local in the sense that it is included in a small neighbourhood of the critical point that it contains.
\begin{figure}[h]
\begin{center}
\includegraphics[scale=0.7]{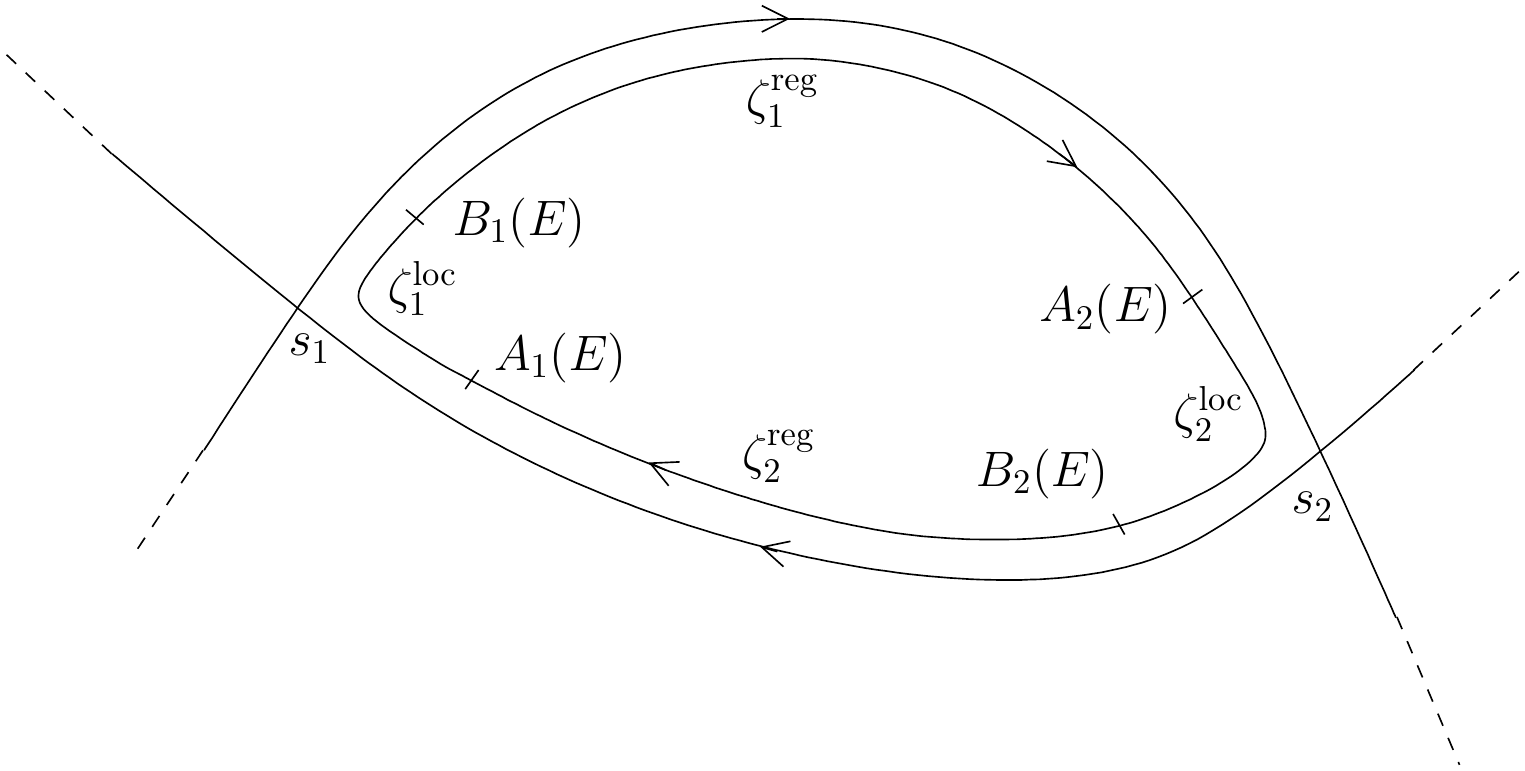}
\end{center}
\caption{Regular and local paths}
\end{figure}
\begin{prop}
\label{prop:regular}
If $\zeta_{E} = [B_{j}(E),A_{j'}(E)]$ is a regular path, then the map $E \mapsto \Theta(\zeta_{E},k)$ belongs to $\classe{\infty}{(D^{\pm})}$ and admits an asymptotic expansion in $k \, \classe{\infty}{(D^{\pm})}[[k^{-1}]]$. This expansion starts as follows:
\begin{equation} \begin{split} \Theta(\zeta_{E},k) = k\left( c_0(\zeta_{E}) + \Phi_{B_{j}(E)}^{(-1)}(B_{j}(E)) -  \Phi_{A_{j'}(E)}^{(-1)}(A_{j'}(E))\right) \\ + c_1(\zeta_{E}) + \mathrm{hol}_{\delta_{\zeta_{E}}}(\zeta_{E}) + \Phi_{B_{j}(E)}^{(0)}(B_{j}(E)) -  \Phi_{A_{j'}(E)}^{(0)}(A_{j'}(E)) + O(k^{-1}); \end{split} \end{equation} 
see section \ref{subsect:holoreg} for the notations.
\end{prop}
In order to study the behaviour of the holonomy of a local path with respect to $E$, we use the parameter dependent normal form given by proposition \ref{prop:fnparam}.
Using the notations of this proposition, we will write $\mathfrak{e}_{j}(E,k) = f_{j}(E) + k^{-1}\varepsilon_{j}(E,k)$. Introduce the Bargmann transform $w_{k,E}^i$ of $v_{k,E}^i$, where
\begin{equation*} \begin{array}{l} v_{k,E}^{1,2}(x) =  1_{\pm x > 0} |x|^{-1/2} \exp(i k \mathfrak{e}_{j}(E,k) \ln|x|); \\
\\
v_{k,E}^{3,4}(\xi) = \mathcal{F}_{k}^{-1} \left( 1_{\pm \xi > 0} |\xi|^{-1/2} \exp(i k \mathfrak{e}_{j}(E,k) \ln|\xi|) \right). \end{array} \end{equation*}
Let $\tilde{w}_{k,E}^i$ be a sequence having microsupport in a sufficiently small neighbourhood of the origin and microlocally equal to $w_{k,E}^i$ on it; then $\tilde{w}_{k,E}^i$ is a basis of the module of microlocal solutions of $P_{k} - \mathfrak{e}_{j}(E,k)$ near the image of the edge with label $i$ by the symplectomorphism $\chi_{E}$.
Consequently, the section $\phi_{k,E}^{(i)} = U_k^E \tilde{w}_{k,E}^i$, where $U_k^E$ is the operator used for the normal form, is a basis of the module of microlocal solutions of equation (\ref{eq:microE}) near the edge $e_i$. Moreover, it displays a good behaviour with respect to the spectral parameter.
\begin{lm}
\label{lm:smoothlag}
The restriction of $\phi_{k,E}^{(i)}$ to a neighbourhood of the edge number $i$ is a Lagrangian section uniformly for $E \in D^{\pm}$. 
\end{lm}
\begin{proof}
First, we prove using a parameter dependant stationary phase lemma that $w_{k,E}^i$ is a Lagrangian section associated to the image of the $i$-th edge, uniformly in $E \in D^{\pm}$. We conclude by the fact that the image of a Lagrangian section depending smoothly on a parameter by a Fourier integral operator is a Lagrangian section depending smoothly on this parameter.
\end{proof}
We also recall the following useful lemma.
\begin{lm}[{\cite[lemma 2.18]{SanCdV}}]
Set $\beta_{j}(E,k) = \frac{1}{2} + i k \mathfrak{e}_{j}(E,k)$ and 
\begin{equation*} \begin{array}{l} \nu_{j}^+ = \left( \frac{k}{2\pi} \right)^{1/2} \Gamma(\beta_{j}) \exp\left(-\beta_{j} \ln k - i \beta_{j} \frac{\pi}{2} \right);\\
\\
\nu_{j}^- = \left( \frac{k}{2\pi} \right)^{1/2} \Gamma(\beta_{j}) \exp\left(-\beta_{j} \ln k + i \beta_{j} \frac{\pi}{2} \right).
\end{array} \end{equation*}
Then for any $E \in I^{\pm}$, 
\begin{equation*} - i \ln \nu_{j}^{\pm} = k \left(f_{j}(E) \ln|f_{j}(E)| - f_{j}(E)\right) + \varepsilon_{j}^{(0)}(E) \ln|f_{j}(E)| \mp \frac{\pi}{4} + O_{E}(k^{-1}). \end{equation*}
\end{lm}
The following proposition shows that the holonomy $\Theta(\zeta_{E},k)$, which has a singular behaviour as $E$ tends to $0$, can be regularized.
\begin{prop}
\label{prop:regularisation}
Fix a component $p^{\pm} \in \mathfrak{C}^{\pm}$, and let $\zeta_{E} = [A_{j}(E),B_{j}(E)]$ be a local path near the critical point $s_{j}$. Assume moreover that $\zeta_{E}$ is oriented according to the flow of $a_{0}$. Then there exists a sequence of $\R/2\pi\Z$-valued functions $g_{\zeta}(.,k) \in \classe{\infty}{(D^{\pm})}$, $E \mapsto g_{\zeta_{E}}(k)$, admitting an asymptotic expansion in $k \, \classe{\infty}{(D^{\pm})}[[k^{-1}]]$ of the form  
\begin{equation*} g_{\zeta}(E,k) = \sum_{\ell=-1}^{+\infty} k^{-\ell} g_{\zeta}^{(\ell)}(E), \end{equation*}
such that
\begin{equation*} \forall E \in I^{\pm}, \quad g_{\zeta}(E,k) = \Theta(\zeta_{E},k) - i \ln(\nu_{j}^{\pm}(E)) \quad (\mathrm{mod} \ 2\pi\Z). \end{equation*}
The first terms of the asymptotic expansion of $g_{\zeta}(.,k)$ are given, for $E \in I^{\pm}$, by:
\begin{equation} g_{\zeta}^{(-1)}(E) =c_{0}(\zeta_{E}) + \left( f_{j}(E) \ln|f_{j}(E)| - f_{j}(E) \right) + \Phi_{B_{j}(E)}^{(-1)}(B_{j}(E)) -  \Phi_{A_{j'}(E)}^{(-1)}(A_{j'}(E))) \label{eq:g0}\end{equation}
and 
\begin{equation} g_{\zeta}^{(0)}(E) = c_{1}(\zeta_{E}) + \mathrm{hol}_{\delta_{\zeta_{E}}}(\zeta_{E}) \mp \frac{\pi}{4} + \varepsilon_{j}^0(E) \ln|f_{j}(E)| + \Phi_{B_{j}(E)}^{(0)}(B_{j}(E)) -  \Phi_{A_{j'}(E)}^{(0)}(A_{j'}(E)). \label{eq:g1}\end{equation}
\end{prop}
\begin{proof}
We can assume that the paths $\zeta_{E}$, $E \in D^{\pm}$ all entirely lie in the open set $\Omega_{s_{j}}$ where the normal form of proposition \ref{prop:fnparam} is valid. Endow each edge $e_{i}$ with the section $\phi_{k,E}^{(i)}$ defined earlier; by lemma \ref{lm:smoothlag}, these sections can be used to compute $\Theta(\zeta_{E},k)$. But we know how the different sections $\phi_{k,E}^{(i)}$ are related: equation (\ref{eq:matrice}) shows that $\Theta(\zeta_{E},k) - i \ln \nu_{j}^{\pm}$ microlocally vanishes. Now, if we choose another set of microlocal solutions, we only add a term admitting an asymptotic expansion in $k \,\classe{\infty}{(D^{\pm})}[[k^{-1}]]$.
\end{proof}

Since the sections $\phi_{k,E}^{(i)}$, $i=1\ldots4$ form a standard basis at $s_j$, they can also be used to compute the holonomy $\text{hol}_{\mathfrak{L}}$. Of course, for this choice of sections, one has $\text{hol}_{\mathfrak{L}}(\zeta_{j}^{\text{loc}}(0)) = 1$ and $g_{\zeta_{j}^{\text{loc}}}(0) = 0$; this allows to obtain the following result.
\begin{prop}
Let $\gamma$ be a cycle in $\Gamma_{0}$, oriented according to the Hamiltonian flow of $a_{0}$, and of the form $\gamma = \zeta_{1}^{\text{loc}}(0) \zeta_{1}^{\text{reg}}(0) \zeta_{2}^{\text{loc}}(0) \zeta_{2}^{\text{reg}}(0) \ldots \zeta_{p}^{\text{loc}}(0) \zeta_{p}^{\text{reg}}(0)$, where $\zeta_{j}^{\text{loc}}$ and $\zeta_{j}^{\text{reg}}$ are respectively local and regular paths in the sense introduced earlier. Define 
\begin{equation*} g(0,k) \sim \sum_{\ell = -1}^{+\infty} g^{(\ell)}(0) k^{-\ell} \end{equation*}
as the sum
\begin{equation*} g(0,k) = g_{\zeta_{1}^{\text{loc}}}(0) + g_{\zeta_{1}^{\text{reg}}}(0) + \ldots + g_{\zeta_{p}^{\text{loc}}}(0) + g_{\zeta_{p}^{\text{reg}}}(0),\end{equation*}
where $g_{\zeta_{j}^{\text{loc}}}$ is given by proposition \ref{prop:regularisation} and $g_{\zeta_{j}^{\text{reg}}} = \Theta(\zeta_{j}^{\text{reg}}(E),k)$. Then
\begin{equation*} \mathrm{hol}_{\mathfrak{L}}(\gamma) = \exp\left( i g(0,k) \right) + O(k^{-\infty}). \end{equation*}
\end{prop}
This is enough to prove the second part of theorem \ref{thm:holsing}.  
\begin{cor}
The first two terms in the asymptotic expansion of the phase of $\mathrm{hol}_{\mathfrak{L}}(\gamma)$ are given by formula (\ref{eq:invsing}).
\end{cor}
\begin{proof}
We start by the case of a cycle $\gamma$ oriented according to the Hamiltonian flow of $a_{0}$. Since $\mathfrak{e}_{j}^0(0) = 0$, formula (\ref{eq:g0}) gives for $j \in \llbracket 1,p \rrbracket$
\begin{equation*} g_{\zeta_{j}^{\text{loc}}}^{(-1)}(0) = c_{0}\left(\zeta_{j}^{\text{loc}}(0)\right) + \Phi_{A_{j}}^{(-1)}(A_{j}(0)) -  \Phi_{B_{j}}^{(-1)}(B_{j}(0)) \end{equation*}
while proposition \ref{prop:regular} shows that (identifying $j = p+1$  with $j = 1$)
\begin{equation*} g_{\zeta_{j}^{\text{reg}}}^{(-1)}(0) = c_{0}\left(\zeta_{j}^{\text{reg}}(0)\right) + \Phi_{B_{j}}^{(-1)}(B_{j}(0)) - \Phi_{A_{j+1}}^{(-1)}(A_{j+1}(0)). \end{equation*}
Consequently, 
\begin{equation*} g^{(-1)}(0) = c_0(\gamma). \end{equation*}

Let us now compute the subprincipal term $g_{\zeta_{j}^{\text{loc}}}^{(0)}(0)$. Recall that it is equal to the limit of 
\begin{equation*} c_{1}(\zeta_{j}^{\text{loc}}(E))  + \mathrm{hol}_{\delta_{\zeta_{j}^{\text{loc}}(E)}}(\zeta_{j}^{\text{loc}}(E)) \mp \frac{\pi}{4} + \varepsilon_{j}^{(0)}(E) \ln|f_{j}(E)| + \Phi_{A_{j}(E)}^{(0)}(A_{j}(E)) - \Phi_{B_{j}(E)}^{(0)}(B_{j}(E))  \end{equation*}
as $E$ goes to $0$, which is equal to
\begin{equation*} \Phi_{A_{j}(0)}^{(0)}(A_{j}(0)) - \Phi_{B_{j}(0)}^{(0)}(B_{j}(0)) \, \mp \, \frac{\pi}{4}  + \lim_{E \rightarrow 0} \left( c_{1}(\zeta_{j}^{\text{loc}}(E))  + \mathrm{hol}_{\delta_{\zeta_{j}^{\text{loc}}(E)}}(\zeta_{j}^{\text{loc}}(E)) + \varepsilon_{j}^{(0)}(E) \ln|f_{j}(E)|  \right). \end{equation*}
First, we show that
\begin{equation} \lim_{E \rightarrow 0} \left( c_{1}(\zeta_{j}^{\text{loc}}(E))  + \varepsilon_{j}^{(0)}(E) \ln|f_{j}(E)|  \right) = \tilde{c}_{1}( \zeta_{j}^{\text{loc}}(0)). \label{eq:c1lim}\end{equation}
Decompose 
\begin{equation*} c_{1}(\zeta_{j}^{\text{loc}}(E)) = \int_{\zeta_{j}^{\text{loc}}(E)} \nu + \int_{\zeta_{j}^{\text{loc}}(E)} \kappa_{E}, \end{equation*}
where we recall that $-i\nu$ stands for the local connection $1$-form associated to the Chern connection of $L_{1}$, and $\kappa_{E}$ is such that $\kappa_{E}(X_{a_{0}}) = - a_{1}$. Of course, the term $ \int_{\zeta_{j}^{\text{loc}}(E)} \nu$ converges to $ \int_{\zeta_{j}^{\text{loc}}(0)} \nu$ as $E$ tends to $0$. Moreover, we have seen that there exists a symplectomorphism $\chi_{E}$ and a smooth function $g_{j}^E$ such that $(g_{j}^E)'(0) \neq 0$ and 
\begin{equation} (a_{0} \circ \chi_{E}^{-1})(x,\xi) - E = g_{j}^E(x\xi).\label{eq:gj}\end{equation}
Hence, if we denote by $\tilde{a}_{0}$ (resp. $\tilde{a}_{1}$, $\tilde{\kappa}_{E}$) the pullback of $a_{0}$ (resp. $a_{1}$, $\kappa_{E}$) by $\chi_{E}^{-1}$, we have 
\begin{equation*} X_{\tilde{a}_{0}}(x,\xi) = (g_{j}^{E})'(x\xi) X_{x\xi}(x,\xi),\end{equation*}
so that $\tilde{\kappa}_{E}$ is characterized by 
\begin{equation*} \tilde{\kappa}_{E}(X_{x\xi}) = \frac{-\tilde{a}_{1}(x,\xi)}{(g_{j}^{E})'(x\xi)}. \end{equation*}
Since $(g_{j}^E)'(0) \neq 0$, the function $b(x,\xi) = \frac{-\tilde{a}_{1}(x,\xi)}{(g_{j}^{E})'(x\xi)}$ is smooth (considering a smaller neighbourhood of $s_{j}$ for the definition of $\zeta_{j}^{\text{loc}}$ if necessary). Moreover, from equation (\ref{eq:gj}), one finds that $(g_{j}^{E})'(0) = \left| \det(\text{Hess}(a_{0})(s)) \right|^{-1/2}$, which yields the fact that $b(0) = \varepsilon_{j}^{(0)}(0)$. Using a known result (see \cite[theorem 2, p.175]{GS} for instance), we can construct smooth functions $F:\R^2 \rightarrow \R$ and $K:\R \rightarrow \R$ such that
\begin{equation*} b(x,\xi) = K(x\xi) - L_{X_{x\xi}}F(x,\xi); \end{equation*}
since $x\xi = f_{j}(E)$ whenever $\chi_{E}^{-1}(x,\xi)$ belongs to $\Gamma_{E}$, this can be written $b(x,\xi) = K(f_{j}(E)) - L_{X_{x\xi}}F(x,\xi)$. Therefore, the function
\begin{equation*} G = K(f_{j}(E)) \ln|x| - F \quad (\text{or }  -K(f_{j}(E)) \ln|\xi| - F \text{ where } x=0) \end{equation*}
restricted to $\chi(\Gamma_{E})$ is a primitive of $\tilde{\kappa}_{E}$.This yields
\begin{equation} \int_{\zeta_{j}^{\text{loc}}(E)} \kappa_{E} = G(\tilde{B}_{j}) - G(\tilde{A}_{j}) = K(f_{j}(E)) \left( \ln|x_{B_{j}}| - \ln|x_{A_{j}}| \right) + F((\tilde{A}_{j}) - F(\tilde{B}_{j}) \label{eq:intkappa}\end{equation}
where $\tilde{m} = \chi_{E}(m)$ for any point $m \in M$, and $(x_{m},\xi_{m})$ are the coordinates of $\tilde{m}$ ($E$ being implicit to simplify notations). Writing $\ln|x_{B_{j}}| - \ln|x_{A_{j}}| = \ln|x_{B_{j}}\xi_{A_{j}}| - \ln|x_{A_{j}} \xi_{A_{j}}|$, we obtain 
\begin{equation*}\begin{split} \int_{\zeta_{j}^{\text{loc}}(E)} \kappa_{E} + \varepsilon_{j}^{(0)}(E) \ln|f_{j}(E)| = F(\tilde{A}_{j}) - F(\tilde{B}_{j}) + K(f_{j}(E)) \ln|x_{B_{j}}\xi_{A_{j}}| \\ + \left(\varepsilon_{j}^{(0)}(E) - K(f_{j}(E)) \right) \ln|f_{j}(E)|. \end{split} \end{equation*}
By definition of $K$, $b(0) - K(0) = 0$, hence $K(f_{j}(E)) = b(0) + O(f_{j}(E)) = \varepsilon_{j}^{(0)}(0) + O(f_{j}(E))$. Thus, the term $\left(\varepsilon_{j}^{(0)}(E) - K(f_{j}(E)) \right) \ln|f_{j}(E)|$ tends to zero as $E$ tends to zero; this induces
\begin{equation*} \lim_{E \rightarrow 0} \int_{\zeta_{j}^{\text{loc}}(E)} \kappa_{E} + \varepsilon_{j}^{(0)}(E) \ln|f_{j}(E)| = F(\tilde{A}_{j}) - F(\tilde{B}_{j}) + K(f_{j}(E)) \ln|x_{B_{j}}\xi_{A_{j}}| \end{equation*}
(one must keep in mind that in this formula, we should write $\tilde{A_{j}} = \tilde{A_{j}}(0)$, etc.). Now, if $a,b$ are points on $\zeta_{j}^{\text{loc}}(0)$ located respectively in $[A_{j},s_{m_{j}}]$ and $[s_{m_{j}},B_{j}]$, then the term on the right hand side of the previous equation is equal to
\begin{equation*} I = \lim_{a,b \rightarrow s_{j}} \left( F(\tilde{A}_{j}) - F(\tilde{a}) + F(\tilde{b}) - F(\tilde{B}_{j}) + K(f_{j}(E)) \ln|x_{B_{j}}\xi_{A_{j}}| \right).  \end{equation*}
Using equation (\ref{eq:intkappa}), it is easily seen that
\begin{equation*} I = \lim_{a,b \rightarrow s_{j}} \left( \int_{[A_{j},a]} \kappa_{E} + \int_{[b,B_{j}]} \kappa_{E} + \varepsilon_{j}^{(0)}(0) \ln|x_{b} \xi_{a}| \right).  \end{equation*}
Remembering definition \ref{dfn:singact}, this proves equation (\ref{eq:c1lim}).
Since $g_{\zeta_{j}^{\text{loc}}}^{(0)}$ and the quantities $\Phi_{A_{j}}^{(-1)}(A_j) - \Phi_{B_{j}}^{(-1)}(B_j)$ and $\Phi_{A_{j}}^{(0)}(A_j) - \Phi_{B_{j}}^{(0)}(B_j)$ are continuous at $E=0$, the term 
\begin{equation*} \mathrm{hol}_{\delta_{\zeta_{j}^{\text{loc}}(E)}}(\zeta_{j}^{\text{loc}}(E))  \end{equation*}
is continuous at $E=0$. 
Hence, if we sum up all the contributions from regular and local paths, we finally obtain
\begin{equation*} g^{(0)}(\gamma) = \tilde{c}_{1}(\gamma) + \sum_{m=1}^p  \frac{\rho_{m} \pi}{4} + \ell(\gamma)  \end{equation*}
where $\rho_{m}$ was introduced in definition \ref{dfn:singeps} and $\ell(\gamma)$ is the quantity
\begin{equation*} \ell(\gamma) = \sum_{j=1}^p \left( \mathrm{hol}_{\delta_{\zeta_{j}^{\text{reg}}(0)}}(\zeta_{j}^{\text{reg}}(0)) + \lim_{E \rightarrow 0} \mathrm{hol}_{\delta_{\zeta_{j}^{\text{loc}}(E)}}(\zeta_{j}^{\text{loc}}(E)) \right) ;\end{equation*}
it is not hard to show that $\ell(\gamma)$ is independent of the choice of the local and regular paths. 
Furthermore, let $\epsilon$ be the index of any smooth embedded cycle which is a continuous deformation of $\gamma$. If the regular and local paths can be chosen so that they all lie in the same connected component of $\Gamma_{E}$, it is clear that $\ell(\gamma) = \epsilon$. If it is not the case, we remove a small path $\eta_{j}(E)$ of $\zeta_{j}^{\text{reg}}(E)$ at any point $A_{j}$ (resp. $B_{j}$) where there is a change of connected component, and replace it by a smooth path $\nu_{j}(E)$ connecting $\zeta_{j}^{\text{reg}}(E)$ and $\zeta_{j}^{\text{loc}}(E)$ (see figure \ref{fig:limite}). We obtain a smooth path $\tilde{\gamma}(E)$; on the one hand, one has $\epsilon(\tilde{\gamma}(E)) = \epsilon$.
\begin{figure}[h]
\begin{center}
\includegraphics[scale=0.7]{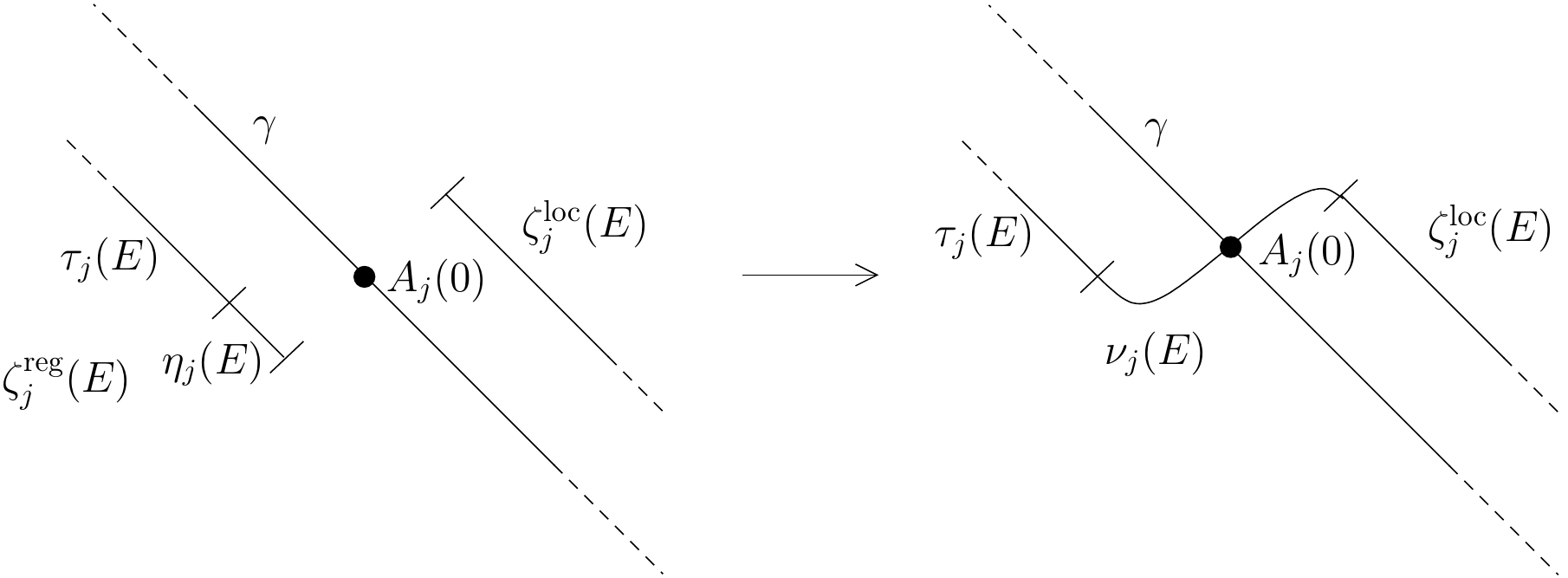}
\end{center}
\caption{Computation of $\ell(\gamma)$}
\label{fig:limite}
\end{figure}
On the other hand, $\epsilon(\tilde{\gamma}(E))$ is the sum of the holonomies of the paths composing $\tilde{\gamma}(E)$. But, if we denote by $\tau_{j}(E)$ the part of $\zeta_{j}^{\text{reg}}(E)$ that remains when we remove $\eta_{j}(E)$, we have
\begin{equation*}  \text{hol}_{\delta_{\tau_{j}(E)}}(\tau_{j}(E))  = \text{hol}_{\delta_{\zeta_{j}^{\text{reg}}(E)}}(\zeta_{j}^{\text{reg}}(E)) - \text{hol}_{\delta_{\eta_{j}(E)}}(\eta_{j}(E))  \end{equation*}
which implies
\begin{equation*} \text{hol}_{\delta_{\tau_{j}(E)}}(\tau_{j}(E)) + \text{hol}_{\delta_{\nu_{j}(E)}}(\nu_{j}(E)) \underset{E \rightarrow 0}{\longrightarrow} \mathrm{hol}_{\delta_{\zeta_{j}^{\text{reg}}(0)}}(\zeta_{j}^{\text{reg}}(0)) \end{equation*}
because
\begin{equation*} \text{hol}_{\delta_{\nu_{j}(E)}}(\nu_{j}(E)) - \text{hol}_{\delta_{\eta_{j}(E)}}(\eta_{j}(E)) \underset{E \rightarrow 0}{\longrightarrow} 0. \end{equation*}
This shows that $\ell(\gamma) = \epsilon$, which concludes the proof for this first case.

If the orientation of the cycle $\gamma$ is opposite to the flow of $X_{a_0}$, just change the sign of the holonomy. 

It remains to investigate the case where $\gamma$ can be smooth at some critical point $s$. We can use the analysis above by introducing two local paths $\zeta_1^{\text{loc}}$ and $\zeta_2^{\text{loc}}$ at $s$ as in figure \ref{fig:lisse} (we make a small move forwards and backwards on an edge added to $\gamma$); one can obtain the claimed result by looking carefully at the obtained holonomies, remembering that the two paths have opposite orientation on the added edge.
\begin{figure}[H]
\begin{center}
\includegraphics[scale=0.7]{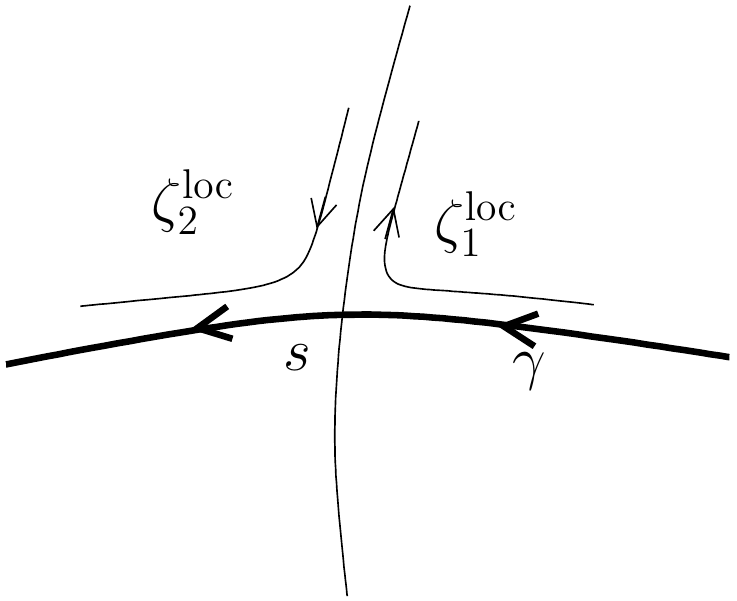}
\end{center}
\caption{Case of a cycle $\gamma$ smooth at $s$}
\label{fig:lisse}
\end{figure}
Note that the choice of the added edge does not change the result.
\end{proof}

\subsection{Derivation of the Bohr-Sommerfeld conditions}

The previous results allow to compute the spectrum of $A_{k}$ in an interval of size $O(1)$ around the singular energy. Indeed, let $\gamma_{E}$, $E \in I^{\pm}$ be a connected component of the level $a_{0}^{-1}(E)$ and $\gamma$ be the cycle in $\Gamma_{0}$ obtained by letting $E$ go to 0. Then one can choose the local and regular paths used to compute the holonomy $\mathrm{hol}_{\mathfrak{L}}(\gamma)$ so that they all lie on $\gamma_{E}$, and define $g(E,k)$ as the sum
\begin{equation*} g(E,k) = g_{\zeta_{1}^{\text{loc}}}(E) + g_{\zeta_{1}^{\text{reg}}}(E) + \ldots + g_{\zeta_{p}^{\text{loc}}}(E) + g_{\zeta_{p}^{\text{reg}}}(E). \end{equation*}
Furthermore, the matrix of change of basis associated to the sections $\phi_{k,E}^{(i)}$ is given by
\begin{equation*} T_{j}(E) =  \exp\left(-\frac{i\pi}{4}\right) \mathcal{E}_{k}(k \mathfrak{e}_{j}(E,k) ) \begin{pmatrix} 1 & i \exp(-k \pi \mathfrak{e}_{j}(E,k)) \\ i \exp(-k \pi \mathfrak{e}_{j}(E,k)) & 1 \end{pmatrix}, \end{equation*}
where the function $\mathcal{E}_{k}$ is defined in equation (\ref{eq:funcEk}). To compute eigenvalues near $E$, apply theorem \ref{thm:cdvp} where $T_{j}$ is replaced by $T_{j}(E)$ and $\mathrm{hol}_{\mathfrak{L}}(\gamma)$ by $\exp\left( ig(E,k) \right)$. Applying Stirling's formula, we obtain  
\begin{equation*} T_{j}(E) =  \exp \left( ik \theta(E,k)  \right) \begin{pmatrix} 1 & 0 \\ 0 & 1 \end{pmatrix} + O(k^{-1}), \quad f_{j}(E) > 0 \end{equation*}
and
\begin{equation*} T_{j}(E) =  \exp \left( ik \theta(E,k)  \right) \begin{pmatrix} 0 & i \\ i & 0 \end{pmatrix} + O(k^{-1}), \quad f_{j}(E) < 0. \end{equation*}
with $\theta(E,k) =  f_{j}(E) \ln|f_{j}(E)| - f_{j}(E) + k^{-1}\left( \varepsilon_{j}^{(0)}(E) \ln|f_{j}(E)| -\frac{\pi}{4} \right) $.
Together with equations (\ref{eq:g0}) and (\ref{eq:g1}), this ensures that we recover the usual Bohr-Sommerfeld conditions away from the critical energy.

In the rest of the paper, we will look for eigenvalues of the form $k^{-1} e + O(k^{-2})$, where $e$ is allowed to vary in a compact set. Hence, we have to replace $A_{k}$ by $A_{k} - k^{-1}e$; this operator still has principal symbol $a_{0}$, but its subprincipal symbol is $a_{1} - e$. Thanks to theorem \ref{thm:holsing}, we are able to compute the singular holonomy and the invariants $\varepsilon_j$ up to $O(k^{-2})$; hence, we approximate the spectrum up to an error of order $O(k^{-2})$.

\subsection{The case of a unique saddle point}
\label{subsection:unique}

In the case of a unique saddle point in $\Gamma_{0}$, it is not difficult to write the Bohr-Sommerfeld conditions in a more explicit form. The critical level $\Gamma_{0}$ looks like a figure eight. We choose the convention for the cut edges and cycles as in figure \ref{fig:unique}.
\begin{figure}[h]
\begin{center}
\includegraphics[scale=0.7]{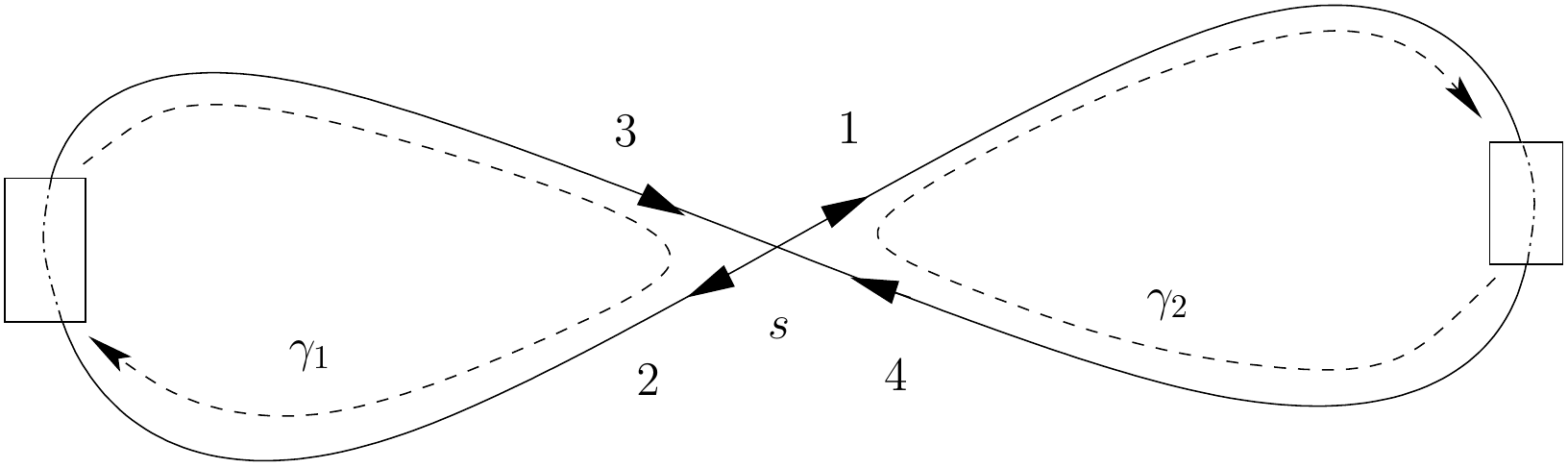}
\end{center}
\caption{The singular level $\Gamma_{0} = a_{0}^{-1}(0)$ and the choice of cut edges and cycles.}
\label{fig:unique}
\end{figure}

Let $s$ be the saddle point, and let $\varepsilon(e,k)$ be the invariant associated to the operator $A_{k} - k^{-1}e$ at $s$; one has $\varepsilon^{(0)}(e) = \varepsilon^{(0)}(0) + e  |\det(\text{Hess}(a_{0})(s))|^{-1/2}$.
Denote by $h_{j}(e) = \exp(i \theta_{j}(e))$ the holonomy of the loop $\gamma_{j}$ in $\mathfrak{L}$; remember that $\theta_{j}$ is given by
\begin{equation*} \theta_{j}(e) = k \, c_{0}(\gamma_{j}) +  \tilde{c}_{1}(\gamma_{j}) + \tilde{\epsilon}(\gamma_{j}) \pi + O(k^{-1}).  \end{equation*}
The Bohr-Sommerfeld conditions are given by the holonomy equations 
\begin{equation*} x_{4} = h_{2} x_{1}, \quad x_{3} = h_{1} x_{2} \end{equation*}
and by the transfer relation at the critical point
\begin{equation*} \begin{pmatrix} x_{3} \\ x_{4} \end{pmatrix} = T \begin{pmatrix} x_{1} \\ x_{2} \end{pmatrix}  \end{equation*}
where $T = T(\varepsilon)$ is defined in equation (\ref{eq:Tj}). Using lemma 2 of \cite{CdV4}, the quantization rule can in fact be written as a real scalar equation.
\begin{prop}
The equation $A_{k}u_{k} = k^{-1}e \, u_{k} + O(k^{-\infty})$ has a normalized eigenfunction if and only if $e$ satisfies the condition
\begin{equation}\label{eq:BSunique} \frac{1}{\sqrt{1 + \exp(2\pi \varepsilon)}}\cos\left( \frac{\theta_{1} - \theta_{2}}{2} \right) = \sin\left( \frac{\theta_{1} + \theta_{2}}{2} + \frac{\pi}{4} + \varepsilon \ln(k) - \arg\left( \Gamma\left( \tfrac{1}{2} + i \varepsilon \right) \right).  \right) \end{equation}
\end{prop}

\section{Examples}

We conclude by investigating two examples on the torus and one on the sphere; these examples present various topologies. More precisely, using the terminology of Bolsinov, Fomenko and Oshemkov \cite{Oshe,BolFom} for atoms (neighbourhoods of singular levels of Morse functions), we provide an example of a type $B$ atom--the only type in complexity 1 (here, complexity means the number of critical points on the singular level) in the orientable case--and two examples of atoms of complexity 2: one is of type $C_2$ ($xy$ on the sphere $S^2$) and the other is of type $C_1$ (Harper's Hamiltonian on the torus $\mathbb{T}^2$). It is a remarkable fact that these two examples are natural not only as the canonical realization of the atom on a surface but also because they come from the simplest possible Toeplitz operators with critical level of given type. 

Note that there are two other types of atoms of complexity 2 in the orientable case; it would be interesting to realize each of them as a hyperbolic level of the principal symbol of a selfadjoint Toeplitz operator and to complete this study. Note that in the context of pseudodifferential operators, Colin de Verdi\`ere and Parisse \cite{CdVP} treated the case of a type $D_{1}$ atom (the triple well potential) among some other examples. More generally, one could use the classification of Bolsinov, Fomenko and Oshemkov to write the Bohr-Sommerfeld conditions for all cases in low complexity ($\leq 3$ for instance); however, the case of two critical points already gives rise to rather tedious computations.    

The details of the quantization of the torus and the sphere are quite standard. Nevertheless, for the sake of completeness, we will recall a few of them at the beginning of each paragraph.

\subsection{Height function on the torus}

Firstly, we consider the quantization of the height function on the torus. This is one of the first examples in Morse theory, perhaps because this is the simplest and most intuitive example with critical points of each type. In particular, the description of the two hyperbolic levels is quite simple.

\begin{figure}[h]
\begin{center}
\includegraphics[scale=0.7]{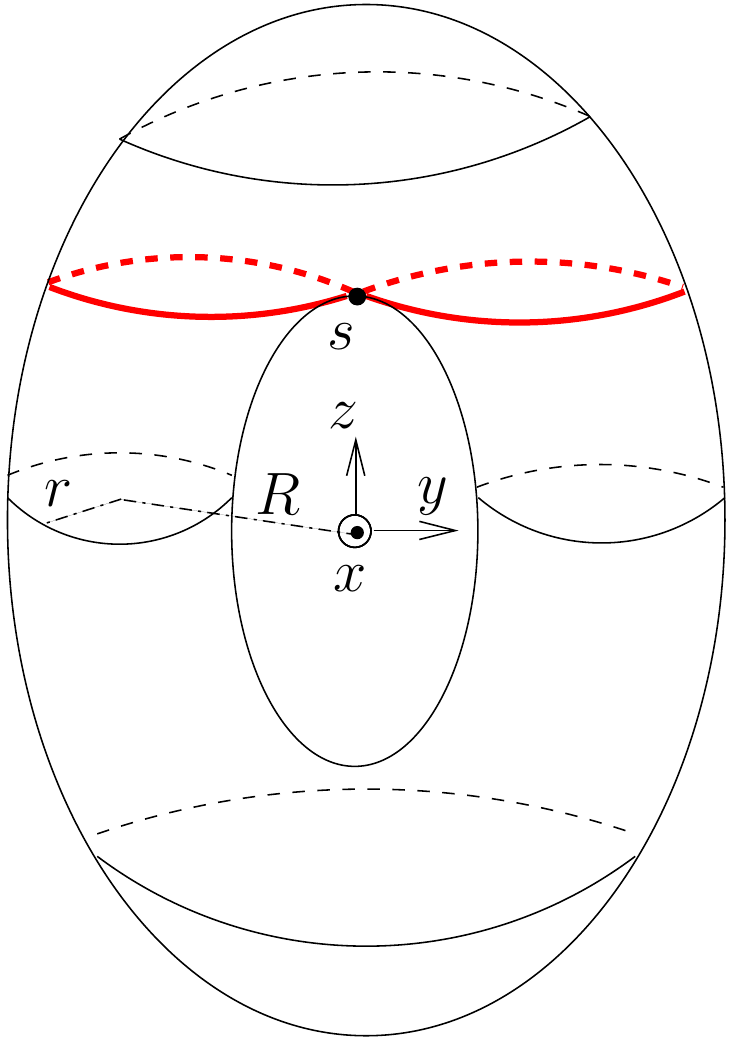}
\end{center}
\caption{Height function on the torus}
\label{fig:hauteur}
\end{figure}

Endow $\R^2$ with the linear symplectic form $\omega_{0}$ and let $L_{0} \rightarrow \R^2$ be the complex line bundle with Hermitian form and connection defined in section \ref{subsubsect:toepcomp}. Let $K$ be the canonical line of $\R^2$ with respect to its standard complex structure. Choose a half-form line, that is a complex line $\delta$ with an isomorphism $\varphi: \delta^{\otimes 2} \rightarrow K$. $K$ has a natural scalar product such that the square of the norm of $\alpha$ is $i\alpha \wedge \bar{\alpha}/\omega_{0}$; endow $\delta$ with the scalar product $\scal{.}{.}_{\delta}$ such that $\varphi$ is an isometry. The half-form bundle we work with, that we still denote by $\delta$, is the trivial line bundle with fiber $\delta$ over $\R^2$.

Consider a lattice $\Lambda$ with symplectic volume $4\pi$. The Heisenberg group $H = \R^2 \times U(1)$ with product
\begin{equation*}(x,u).(y,v) = \left(x+y,uv\exp\left(\frac{i}{2}\omega_{0}(x,y)\right) \right) \end{equation*}
acts on the bundle $L_0 \rightarrow \R^2$, with action given by the same formula. This action preserves the prequantum data, and the lattice $\Lambda$ injects into $H$; therefore, the fiber bundle $L_{0}$ reduces to a prequantum bundle $L$ over $\mathbb{T}^2 = \R^2 \slash \Lambda$. The action extends to the fiber bundle $L_{0}^k$ by
\begin{equation*} (x,u).(y,v) = \left(x+y,u^k v\exp\left(\frac{ik}{2}\omega_{0}(x,y)\right) \right). \end{equation*}
We let the Heisenberg group act trivially on $\delta$. We obtain an action
\begin{equation*} T^*:  \Lambda  \rightarrow  \text{End}\left(\classe{\infty}{\left(\R^2,L_{0}^k \otimes \delta \right)}\right), \quad u  \mapsto  T_{u}^*. \end{equation*}
The Hilbert space $\Hil_{k} = H^0(M,L^k)$ can naturally be identified to the space $\Hil_{\Lambda,k}$ of holomorphic sections of $L_{0}^k \otimes \delta \rightarrow \R^2$ which are invariant under the action of $\Lambda$, endowed with the hermitian product
\begin{equation*} \langle \varphi,\psi \rangle = \int_{D} \scal{\varphi }{\psi}_{\delta} \ |\omega_{0}| \end{equation*}
where $D$ is the fundamental domain of the lattice. Furthermore, $\Lambda \slash 2k$ acts on $\Hil_{\Lambda,k}$. Let $e$ and $f$ be generators of $\Lambda$ satisfying $\omega_{0}(e,f) = 4 \pi$; one can show that there exists an orthonormal basis $(\psi_{\ell})_{\ell \in \Z \slash 2k\Z}$ of $\Hil_{\Lambda,k}$ such that
\begin{equation*} \forall \ell \in \Z \slash 2k\Z \qquad \left\{\begin{array}{c} T^*_{e/2k} \psi_{\ell} = w^{\ell} \psi_{\ell}  \\ T^*_{f/2k} \psi_{\ell} = \psi_{\ell + 1} \end{array}\right. \end{equation*}
with $w = \exp\left( \frac{i \pi}{k} \right)$. The sections $\psi_{\ell}$ can be expressed in terms of $\Theta$ functions.

Set $M_{k} = T^*_{e/2k}$ and $L_{k} = T^*_{f/2k}$. Let $(q,p)$ be coordinates on $\R^2$ associated to the basis $(e,f)$ and $\left[q,p\right]$ be the equivalence class of $(q,p)$. Both $M_{k}$ and $L_{k}$ are Toeplitz operators, with respective principal symbols $ \left[q,p\right]  \mapsto  \exp(2 i \pi p)$ and $ \left[q,p\right] \mapsto \exp(2 i \pi q) $, and vanishing subprincipal symbols.

It is a well-known fact that $\T^2$ is diffeomorphic to the surface shown in figure \ref{fig:hauteur} above, which is obtained by rotating a circle of radius $r$ around a circle of radius $R > r$ contained in the $yz$ plane; the diffeomorphism is given by the explicit formulas
\begin{equation*} x = r \sin(2\pi q), \ y = (R + r \cos(2\pi q)) \cos(2\pi p), \ z = (R + r \cos(2\pi q)) \sin(2\pi p). \end{equation*}
Hence, the Hamiltonian that we consider is 
\begin{equation*} a_{0}(q,p) = (R + r \cos(2\pi q)) \sin(2\pi p) \end{equation*}
on the fundamental domain $D$. We try to quantize it, \textit{id est} find a Toeplitz operator $A_{k}$ with principal symbol $a_{0}$. The Toeplitz operators
\begin{equation*} B_{k} = \frac{1}{2i}(M_{k} - M_{k}^*), \quad C_{k} = R \Pi_{k} + \frac{r}{2}(L_{k} + L_{k}^*) \end{equation*}
are selfadjoint and
\begin{equation*} \sigma_{\text{norm}}(B_{k}) = \sin(2\pi p) + O(\hbar^2), \quad \sigma_{\text{norm}}(C_{k}) = R + r\cos(2\pi q) + O(\hbar^2). \end{equation*}
Hence $A_{k} = \tfrac{1}{2}(B_{k}C_{k} + C_{k}B_{k})$ is a selfadjoint Toeplitz operator with normalized symbol $a_{0} + O(\hbar^2)$. Its matrix in the basis $(\psi_{\ell})_{\ell \in \Z/2k\Z}$ writes
\begin{equation} \label{eq:mathauteur} \begin{pmatrix} R \alpha_{0} & \frac{r}{4}(\alpha_{0} + \alpha_{1}) & 0 & \ldots & 0 & \frac{r}{4}(\alpha_{2k-1} + \alpha_{0}) \\ 
\frac{r}{4}(\alpha_{0} + \alpha_{1}) & R \alpha_{1} & \frac{r}{4}(\alpha_{1} + \alpha_{2}) & 0  & \ldots  & 0 \\
0 & \frac{r}{4}(\alpha_{1} + \alpha_{2}) & R \alpha_{2} & \ddots & \ddots & \vdots \\
\vdots & \ddots & \ddots  & \ddots & \ddots & 0 \\
\vdots & \ddots & \ddots  & \ddots & R \alpha_{2k-2} & \frac{r}{4}(\alpha_{2k-2} + \alpha_{2k-1}) \\
\frac{r}{4}(\alpha_{0} + \alpha_{2k-1}) & 0 & \ldots & 0 & \frac{r}{4}(\alpha_{2k-2} + \alpha_{2k-1}) & R \alpha_{2k-1}
\end{pmatrix} \end{equation}
with $\alpha_{\ell} = \sin(\ell \pi/k)$.

The level $\Gamma_{R-r} = a_{0}^{-1}(R-r)$ contains one hyperbolic point $s=(1/2,1/4)$. It is the union of the two branches
\begin{equation*} p = \frac{1}{2\pi} \arcsin\left( \frac{R-r}{R + r\cos(2\pi q)} \right) \ \text{and} \ p = \frac{1}{2} - \frac{1}{2\pi} \arcsin\left( \frac{R-r}{R + r\cos(2\pi q)} \right). \end{equation*}
The Hamiltonian vector field associated to $a_{0}$ is given by
\begin{equation*} X_{a_{0}}(q,p) = \frac{1}{2} \left( R + r \cos(2\pi q) \right) \cos(2\pi p) \frac{\partial}{\partial q} + \frac{r}{2} \sin(2\pi q) \sin(2\pi p) \frac{\partial}{\partial p}. \end{equation*}
Moreover, one has
\begin{equation} \varepsilon^{(0)} = \frac{e}{\pi\sqrt{r(R-r)}} \end{equation}

We choose the cycles $\gamma_{1}$ and $\gamma_{2}$ with the convention given in section \ref{subsection:unique}. We have to compute the principal and subprincipal actions of $\gamma_{1},\gamma_{2}$ and their indices $\tilde{\epsilon}$. Let us detail the calculations in the case of $\gamma_{1}$.

We parametrize $\gamma_{1}$ by $q \mapsto \left(q, \frac{1}{2} - \frac{1}{2\pi} \arcsin\left( \frac{R-r}{R + r\cos(2\pi q)} \right) \right)$. The principal action is given by
\begin{equation} c_{0}(\gamma_{1}) = 2I(R,r) - 2\pi, \end{equation}
where $I(R,r)$ is the integral
\begin{equation*} I(R,r) = \int_{0}^{1} \arcsin\left( \frac{R-r}{R+r\cos(2\pi q)} \right) \ dq; \end{equation*} 
unfortunately, we do not know any explicit expression for this integral, so for numerical computations, once fixed the radii $R$ and $r$, we obtain the value of $I(R,r)$ thanks to numerical integration routines.

On $\gamma_{1}$, the subprincipal form reads
\begin{equation*} \kappa_{0} = \frac{-2e \ dq}{\sqrt{(R + r\cos(2\pi q))^2 - (R-r)^2 )}}. \end{equation*}
One can obtain an explicit primitive thanks to any computer algebra system. 
Furthermore, some computations show that the symplectic area of the parallelogram $R_{a,b}$ is equal to
\begin{equation*} \int_{R_{a,b}} \omega =  8\pi \sqrt{\frac{r}{R-r}} \left(q_{a} - \frac{1}{2} \right) \left(\frac{1}{2} - q_{b} \right). \end{equation*}
This yields the following value for the subprincipal action:
\begin{equation} \tilde{c}_{1}(\gamma_{1}) = \varepsilon^{(0)} \ln\left( \frac{32}{\pi} \sqrt{\frac{r}{R}\left(1-\frac{r}{R}\right)} \right). \end{equation}
Finally, the index associated to half-forms is $\tilde{\epsilon}(\gamma_{1}) = 1/4$. For $\gamma_{2}$, one can check that
\begin{equation} c_{0}(\gamma_{2}) = 2 I(R,r), \quad  \tilde{c}_{1}(\gamma_{2}) = \varepsilon^{(0)} \ln\left( \frac{32}{\pi} \sqrt{\frac{r}{R}\left(1-\frac{r}{R}\right)} \right), \quad \tilde{\epsilon}(\gamma_{2}) = 1/4.\end{equation} 

With this data, one can test the Bohr-Sommerfeld condition for different couples $(R,r)$. We illustrate this with $(R,r) = (4,1)$ (note that we have tested several couples). We compare the eigenvalues obtained numerically from the matrix (\ref{eq:mathauteur}) and the ones derived from the Bohr-Sommerfeld conditions (\ref{eq:BSunique}) in the interval $I=[R-r - 10 k^{-1}, R - r + 10 k^{-1}]$. In figure \ref{fig:comparaisonBShauteur}, we plotted the theoretical and numerical eigenvalues; figure \ref{fig:erreurhauteur} shows the error between the eigenvalues and the solutions of the Bohr-Sommerfeld conditions for fixed $k$, while figure \ref{fig:logerreurhauteur} is a graph of the logarithm of the maximal error in the interval $I$ as a function of $\log(k)$. 

\begin{figure}[H]
\subfigure[$k=10$]{\includegraphics[scale=0.32]{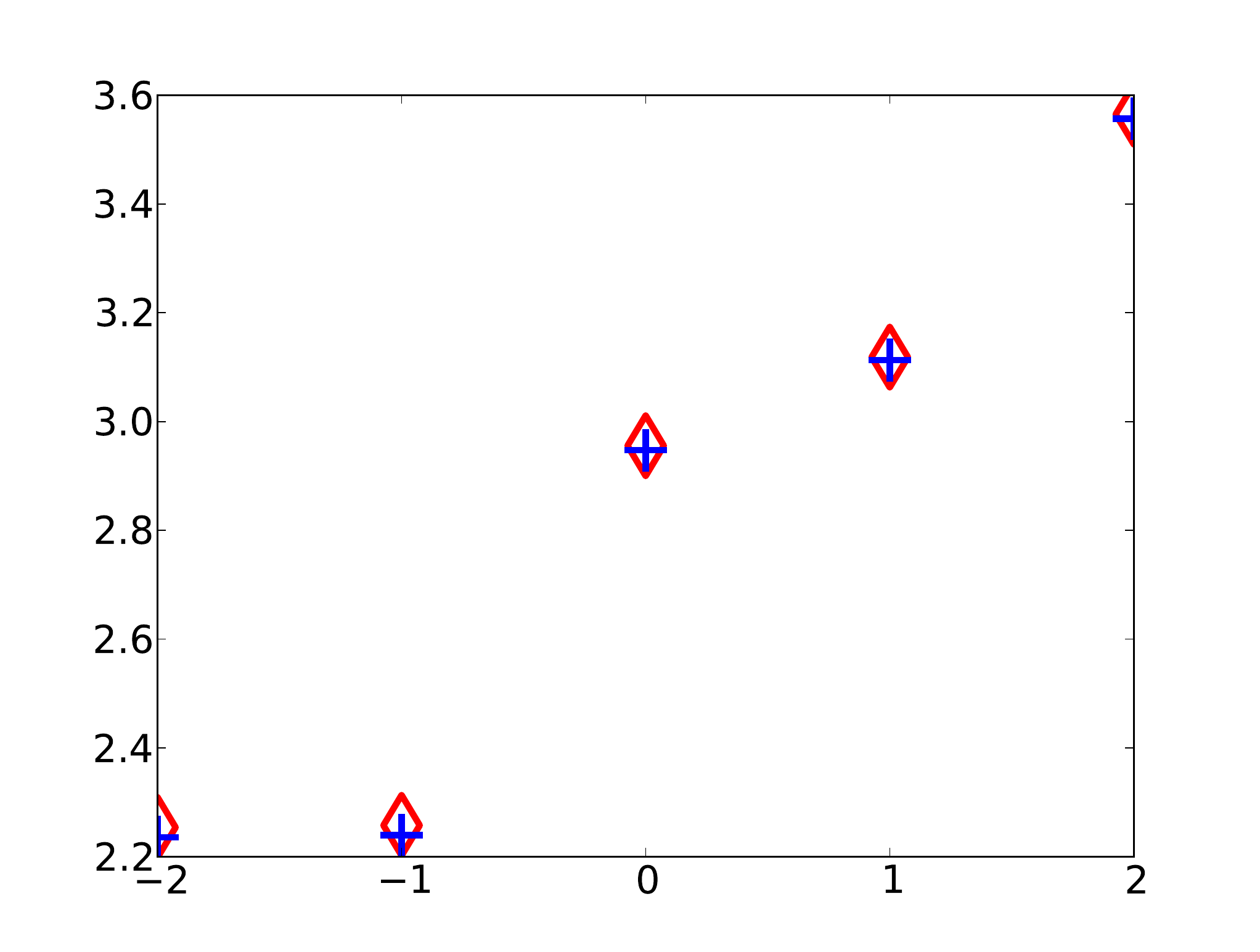} } 
\subfigure[$k=100$]{\includegraphics[scale=0.32]{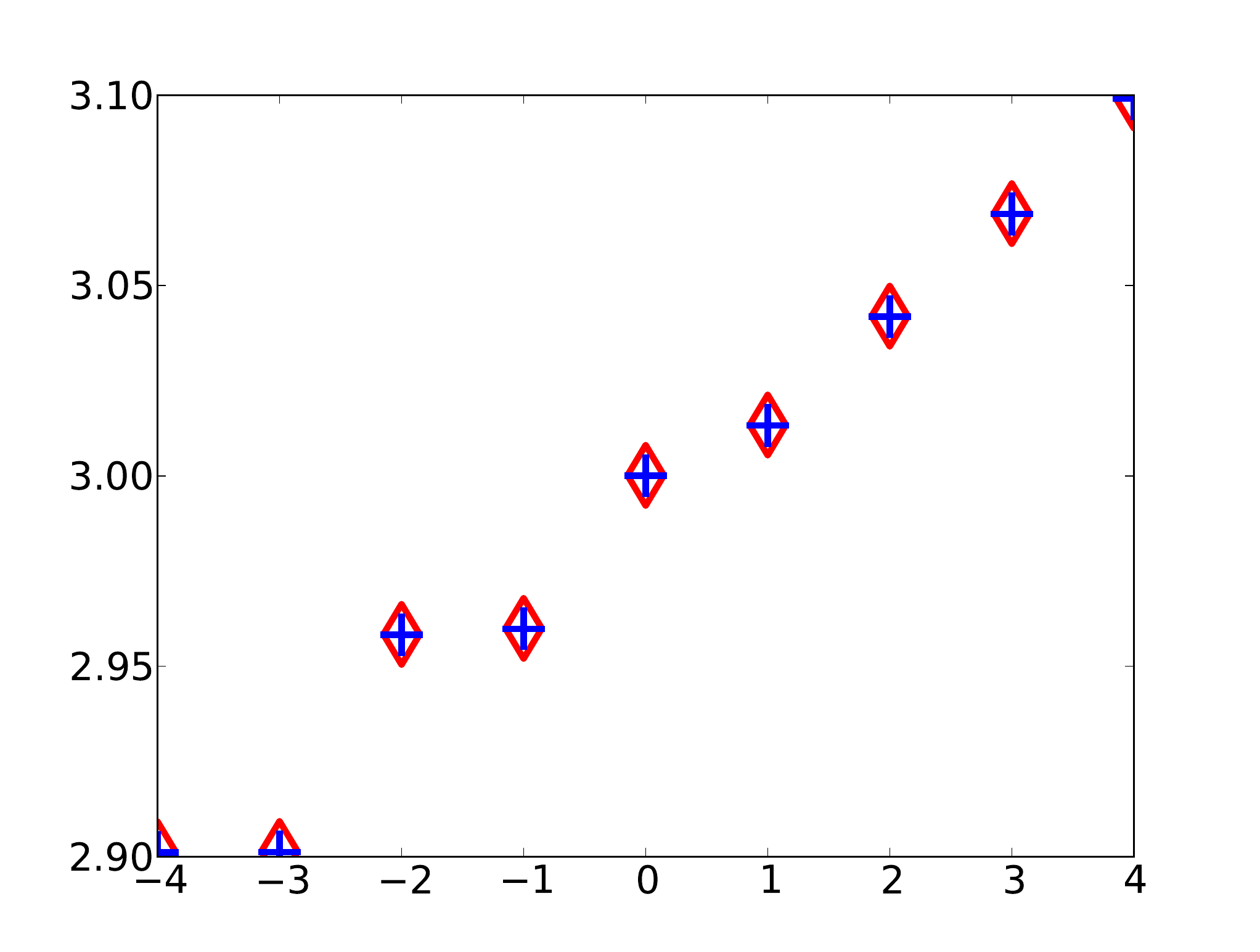} }
\caption{Eigenvalues in $[R-r - 10 k^{-1}, R - r + 10 k^{-1}]$; in red diamonds, the eigenvalues of $A_{k}$ obtained numerically; in blue crosses, the theoretical eigenvalues derived from the Bohr-Sommerfeld conditions. The results are indexed with respect to the eigenvalue closest to the critical energy, labeled as $0$.  Observe that even for $k=10$, the method is very precise.}
\label{fig:comparaisonBShauteur}
\end{figure}
\begin{figure}[H]
\subfigure[$k=10$]{\includegraphics[scale=0.32]{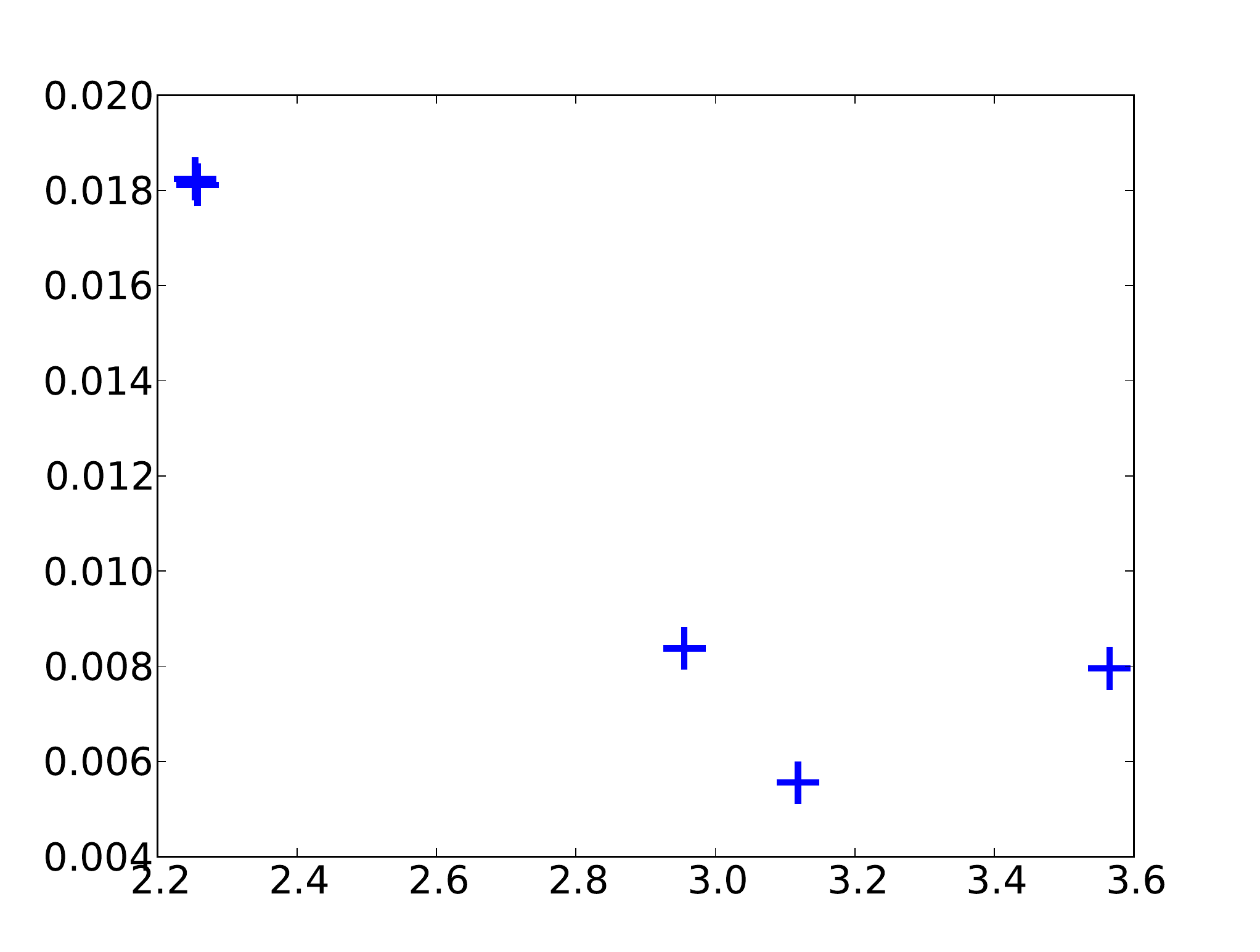} } 
\subfigure[$k=100$]{\includegraphics[scale=0.32]{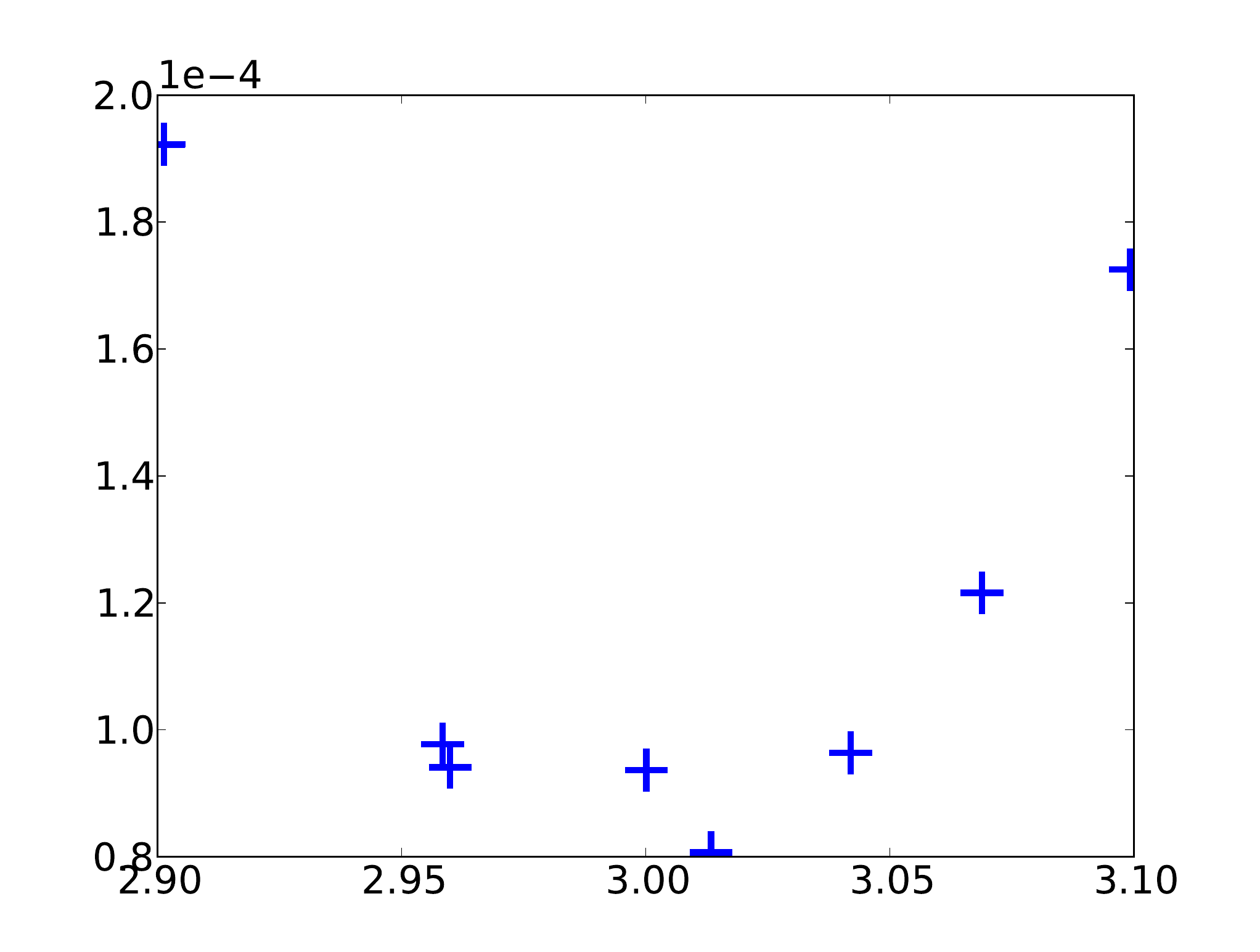} } 
\caption{Absolute value of the difference between the numerical and theoretical eigenvalues; the error is smaller near the critical energy ($R-r = 3$ in this case).}
\label{fig:erreurhauteur}
\end{figure}

\begin{figure}[H]
\begin{center}
\includegraphics[scale=0.4]{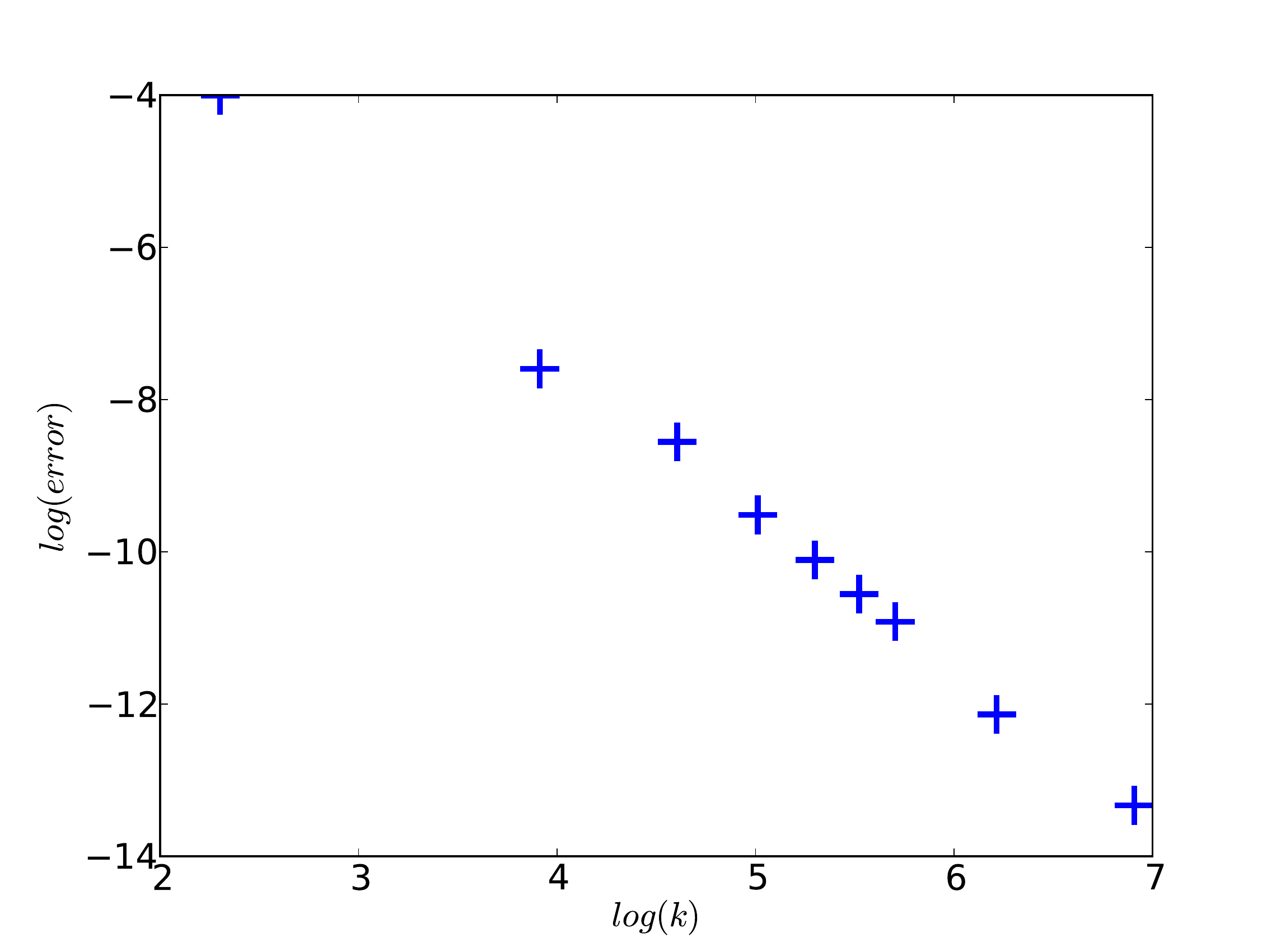} 
\end{center}
\caption{Logarithm of the maximal error as a function of the logarithm of $k$; the error displays a behaviour in $O(k^{-2})$, as expected.}
\label{fig:logerreurhauteur}
\end{figure}

\subsection{$xy$ on the 2-sphere}

Let us consider another simple example, but this time with two saddle points on the critical level. We will quantize the Hamiltonian $a_{0}(x,y,z) = xy$ on the sphere $S^2$. Let us briefly recall the details of the quantization of this surface. 

Start from the complex projective plane $\C\P^1$ and let $L = \mathcal{O}(1)$ be the dual bundle of the tautological bundle
\begin{equation*} \mathcal{O}(-1) = \left\{ (u,v) \in \C\P^1 \times \C^2; \quad v \in u \right\} \end{equation*}
with natural projection. $L$ is a Hermitian, holomorphic line bundle; let us denote by $\nabla$ its Chern connection. The $2$-form $\omega = i \ \text{curv}(\nabla)$ is the symplectic form on $\C\P^1$ associated to the Fubini-Study K\"ahler structure, and $L \rightarrow \C\P^1$ is a prequantum bundle. Moreover, the canonical bundle naturally identifies to $\mathcal{O}(-2)$, hence one can choose the line bundle $\delta = \mathcal{O}(-1)$ as a half-forms bundle. The state space $\Hil_k = H^0(\C\P^1,L^k \otimes \delta)$ can be identified with the space $\C_{p_{k}}[z_1,z_2]$ of homogeneous polynomials of degree $p_{k} = k-1$ in two variables. The polynomials  
\begin{equation*} P_{\ell}(z_1,z_2) = \sqrt{\frac{(p_{k}+1) {p_{k} \choose \ell}}{2\pi}} \ z_1^{\ell}z_2^{p_{k}-\ell}, \quad 0 \leq \ell \leq p_{k},  \end{equation*}
form an orthonormal basis of $\Hil_k$. The sphere $S^2 = \{(x,y,z) \in \R^3; x^2+y^2+z^2 = 1 \}$ is diffeomorphic to $\C\P^1$ \textit{via} the stereographic projection (from the north pole to the plane $z=0$). The symplectic form $\omega$ on $\C\P^1$ is carried to the symplectic form $\omega_{S^2} = -\frac{1}{2}\Omega$, with $\Omega$ the usual area form on $S^2$ (the one which gives the area $4\pi$). The operator $A_{k}$ acting on the basis $(P_{\ell})_{0 \leq \ell \leq p_{k}}$ by
\begin{equation*} A_k P_{\ell} = \frac{i}{p_{k}^2} \left( \alpha_{\ell,k} P_{\ell - 2} - \beta_{\ell,k}P_{\ell + 2} \right), \end{equation*}
with
\begin{equation*} \alpha_{\ell,k} = \sqrt{\ell(\ell-1)(p_{k}-\ell+1)(p_{k}-\ell+2)} \end{equation*}
and
\begin{equation*} \beta_{\ell,k} = \sqrt{(\ell+1)(\ell+2)(p_{k}-\ell-1)(p_{k}-\ell)}, \end{equation*}
is a Toeplitz operator with principal symbol $a_0(x,y,z) = xy$ and vanishing subprincipal symbol (for more details, one can consult \cite[section 3]{BloGol} for instance). Note that $\alpha_{\ell,k} = \beta_{p_{k}-\ell,k}$, which implies that if $\lambda$ is an eigenvalue of $A_{k}$, then $-\lambda$ also is.

The level $a_{0}^{-1}(0)$ is critical, and contains two saddle points: the poles $N$ (north) and $S$ (south). It is the union of the two great circles $x=0$ and $y=0$. We choose the cut edges and cycles as indicated in figure \ref{fig:cyclessphere}.
\begin{figure}[h]
\begin{center}
\includegraphics[scale=0.7]{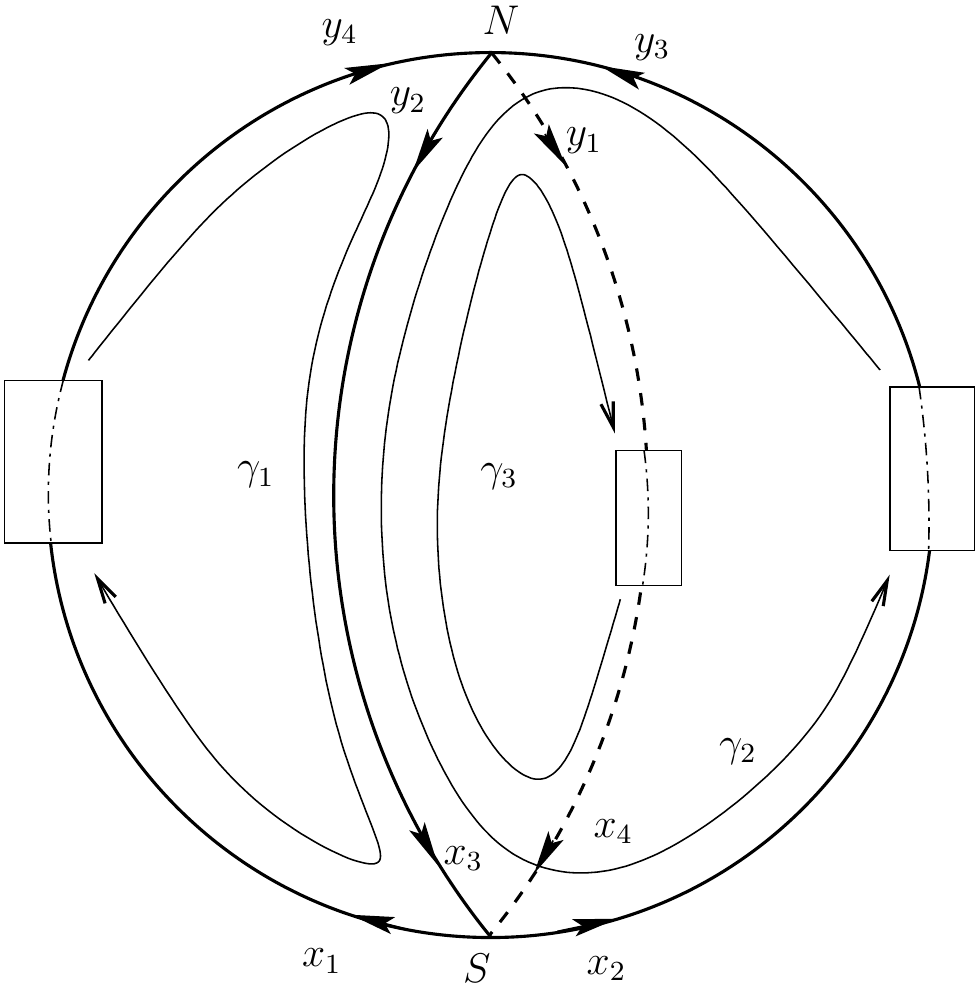}
\end{center}
\caption{Choice of the cycles and cut edges}
\label{fig:cyclessphere}
\end{figure}
Set $h_j = \text{hol}_{\mathfrak{L}}(\gamma_j) = \exp(i \theta_j)$; remember that $\theta_j = k c_0(\gamma_j) + \tilde{c}_1(\gamma_j) + \tilde{\epsilon}(\gamma_j)\pi + O(k^{-1})$. The holonomy equations read
\begin{equation} y_2 = x_3, \quad y_4 = h_1 x_1, \quad y_3 = h_2 x_2, \quad x_4 = h_3 y_1 \label{eq:holsphere}\end{equation}
while the transfer equations are given by
\begin{equation} \begin{pmatrix} x_3 \\ x_4 \end{pmatrix} = T_S \begin{pmatrix} x_1 \\ x_2 \end{pmatrix}, \quad \begin{pmatrix} y_3 \\ y_4 \end{pmatrix} = T_N \begin{pmatrix} y_1 \\ y_2 \end{pmatrix}.  \label{eq:transsphere}\end{equation}
The system (\ref{eq:holsphere}) + (\ref{eq:transsphere}) has a solution if and only if the matrix
\begin{equation*} U = T_S \begin{pmatrix} 0 & \exp(-i\theta_1) \\ \exp(-i\theta_2) & 0 \end{pmatrix} T_N \begin{pmatrix} 0 & \exp(-i\theta_3) \\ 1 & 0 \end{pmatrix}  \end{equation*}
admits 1 as an eigenvalue. The matrix $U$ is unitary, and if we write $U = \begin{pmatrix} a & b \\ c & d \end{pmatrix}$, a straightforward computation shows that \begin{equation*} |a|^2 = |d|^2 = \frac{1 - 2\cos(\theta_2 - \theta_1)\exp(-\pi(\varepsilon_S + \varepsilon_N)) + \exp(-2\pi(\varepsilon_S + \varepsilon_N))}{(1+\exp(-2\pi \varepsilon_S))(1+\exp(-2\pi \varepsilon_N))};\end{equation*}
hence, by lemma 2 of \cite{CdV4}, $1$ is an eigenvalue of $U$ if and only if
\begin{equation*} |a| \sin\left( \frac{\arg(ad) - \pi}{2} - \arg(a) \right) = \sin\left( \frac{\arg(ad) - \pi}{2} \right). \end{equation*}
This amounts to the equation
\begin{eqnarray*} \lefteqn{|a|\cos\left( \frac{\arg(z) - \arg(w)}{2} \right)} \\ && = \sin\left( \frac{\arg(z) + \arg(w)}{2} + \arg\left( \Gamma(\tfrac{1}{2} + i \ \varepsilon_N) \right) + \arg\left( \Gamma(\tfrac{1}{2} + i \ \varepsilon_S) \right) - (\varepsilon_S + \varepsilon_N)\ln(k) \right)  \end{eqnarray*}
with 
\begin{equation*} z = \exp(-i(\theta_2 + \theta_3)) - \exp(-\pi(\varepsilon_S + \varepsilon_N) - i(\theta_1 + \theta_3)) \end{equation*}
and
\begin{equation*} w = \exp(-i\theta_1) - \exp(-\pi(\varepsilon_N + \varepsilon_S) - i\theta_2) .\end{equation*}
One has
\begin{equation} \varepsilon_S^{(0)} = \varepsilon_N^{(0)} = \varepsilon^{(0)} = \frac{e}{2}. \end{equation}
Moreover, the principal actions are
\begin{equation} c_0(\gamma_1) = - \frac{\pi}{2}, \quad c_0(\gamma_2) = \frac{\pi}{2}, \quad c_0(\gamma_3) = \pi. \end{equation}
Then, one finds for the subprincipal actions
\begin{equation} \tilde{c}_1(\gamma_1) = 2 \varepsilon^{(0)} \ln 2 = \tilde{c}_1(\gamma_2), \quad \tilde{c}_1(\gamma_3) = 0. \end{equation}
Finally, the indices $\tilde{\epsilon}$ are the following:
\begin{equation} \tilde{\epsilon}(\gamma_1) = \frac{3}{2}, \quad \tilde{\epsilon}(\gamma_2) = \frac{1}{2}, \quad \tilde{\epsilon}(\gamma_3) = 1. \end{equation}
 
\begin{figure}[H]
\subfigure[$k=10$]{\includegraphics[scale=0.32]{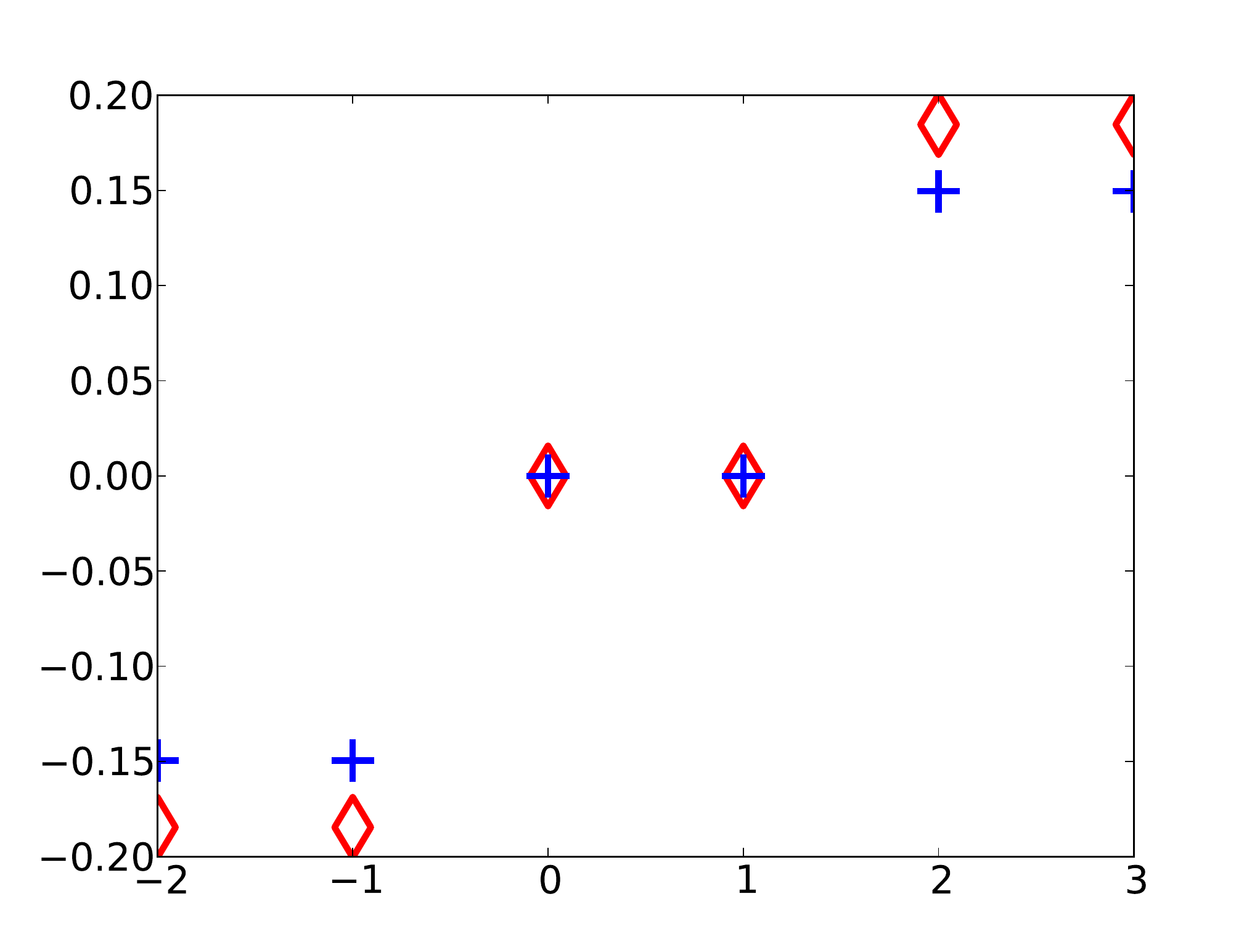} } 
\subfigure[$k=100$]{\includegraphics[scale=0.32]{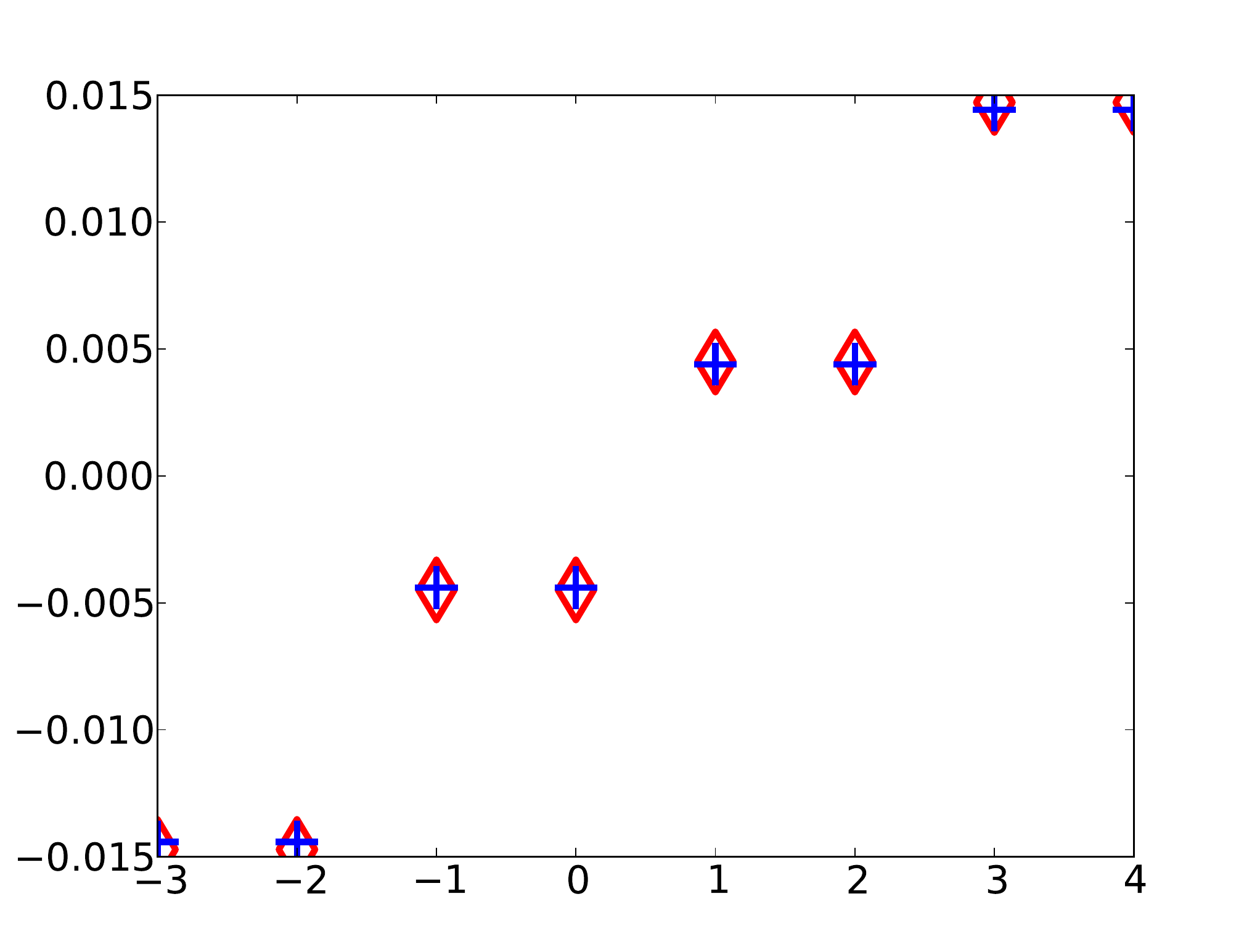} }
\caption{Eigenvalues in $[-2 k^{-1}, 2 k^{-1}]$; in red diamonds, the eigenvalues of $A_{k}$ obtained numerically; in blue crosses, the theoretical eigenvalues derived from the Bohr-Sommerfeld conditions.}
\label{fig:comparaisonBSsphere}
\end{figure}

\subsection{Harper's Hamiltonian on the torus} 

Keeping the conventions and notations of the first example, we consider the Hamiltonian (sometimes known as Harper's Hamiltonian since it is related to Harper's equation \cite{HfSj})
\begin{equation*} a_0(q,p) = 2(\cos(2\pi p) + \cos(2\pi q)) \end{equation*}
on the torus. The operator $A_k = M_k + M_k^* + L_k + L_k^*$ is a Toeplitz operator with principal symbol $a_0$ and vanishing subprincipal symbol. Its matrix in the basis $(\psi_{\ell})_{\ell \in \Z \slash 2k\Z}$ is 
\begin{equation*}  \begin{pmatrix} 2 \alpha_0 & 1 & 0 & \ldots & 0 & 1 \\ 1 & \ddots & \ddots & \ddots &  & 0\\ 0 & \ddots & \ddots & \ddots & \ddots & \vdots \\ \vdots & \ddots & \ddots & \ddots & \ddots & 0 \\ 0 &  & \ddots & \ddots  & \ddots & 1 \\ 1 & 0 &  \ldots & 0 & 1 & 2 \alpha_{2k-1}  \end{pmatrix}  \end{equation*}
where
\begin{equation*} \alpha_{\ell} = \cos\left( \frac{\ell \pi}{k} \right), \quad 0 \leq \ell \leq 2k-1. \end{equation*}
The critical level $\Gamma_{0} = a_0^{-1}(0)$ contains two hyperbolic points: $s_1=(0,1/2)$ and $s_2=(1/2,0)$. On the fundamental domain, it is the union of the four segments described in figure \ref{fig:fondharper}; hence, its image on the torus it is the union of two circles that intersect at two points. 
\begin{figure}[h]
\begin{center}
\includegraphics[scale=0.7]{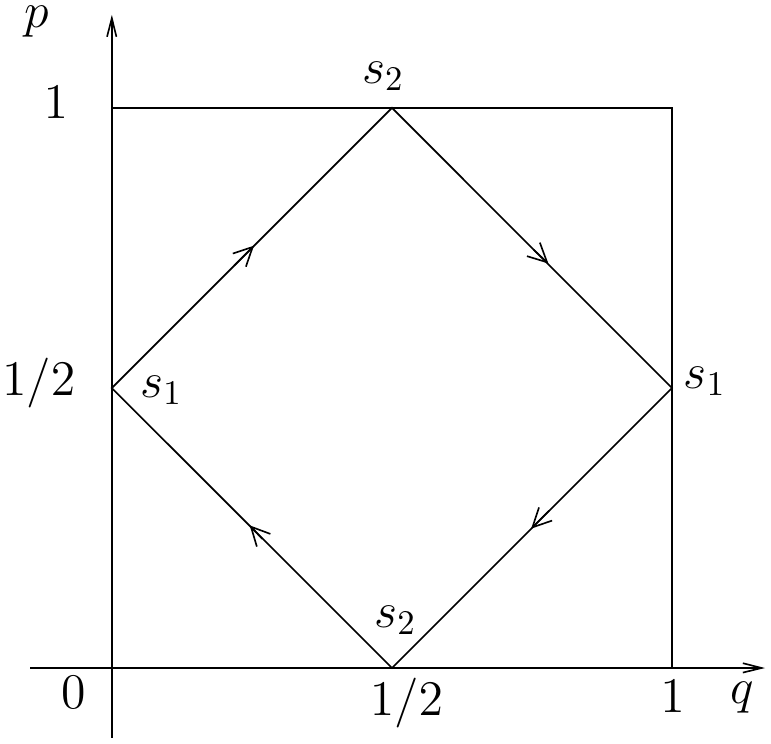}
\end{center}
\caption{Critical level $\Gamma_{0}$ on the fundamental domain; the arrows indicate the direction of the Hamiltonian flow of $a_{0}$.}
\label{fig:fondharper}
\end{figure}
We choose the cycles and cut edges as in figure \ref{fig:harper} (for a representation of the two circles in a two-dimensional view) and \ref{fig:cycles} (for a representation of the cycles on the fundamental domain).
\begin{figure}[h]
\begin{center}
\includegraphics[scale=0.7]{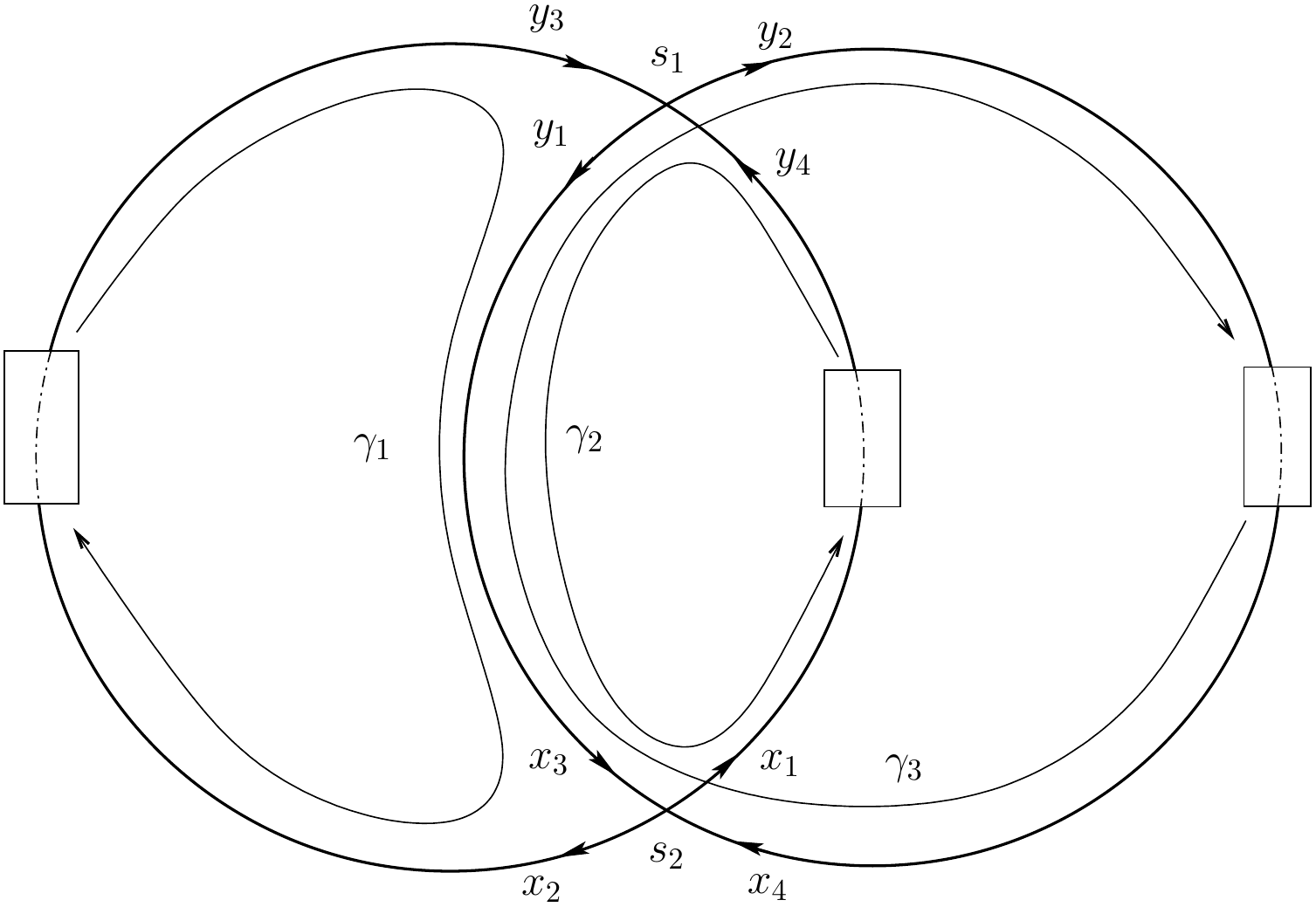}
\end{center}
\caption{Choice of the cycles and cut edges}
\label{fig:harper}
\end{figure}
\begin{figure}[h]
\begin{center}
\includegraphics[scale=0.7]{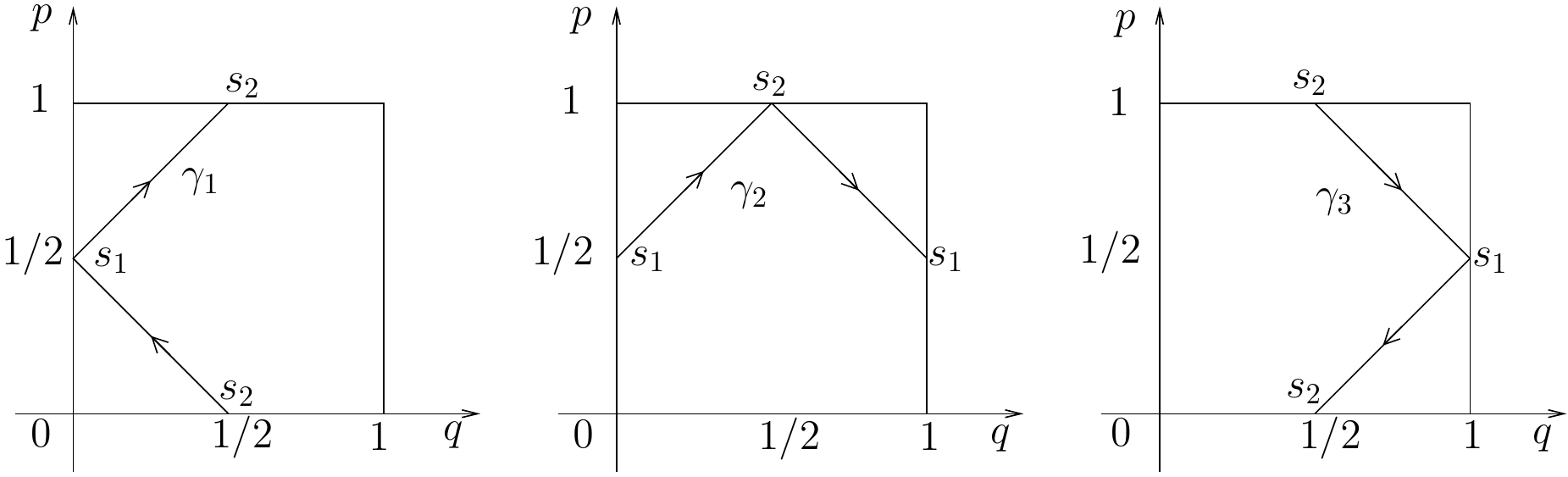}
\end{center}
\caption{Cycles on the fundamental domain}
\label{fig:cycles}
\end{figure}
We write the holonomy equations 
\begin{equation} y_1 = x_3, \quad y_3 = h_1 x_2, \quad y_4 = h_2 x_1, \quad x_4 = h_3 y_1 \label{eq:holotore}\end{equation}
and the transfer equations
\begin{equation} \begin{pmatrix} x_3 \\ x_4 \end{pmatrix} = T_2 \begin{pmatrix} x_1 \\ x_2 \end{pmatrix}, \quad \begin{pmatrix} y_3 \\ y_4 \end{pmatrix} = T_1 \begin{pmatrix} y_1 \\ y_2 \end{pmatrix}  \label{eq:transtore} \end{equation}
where $h_j = \text{hol}_{\mathfrak{L}}(\gamma_j) = \exp(i \theta_j)$.
Following the same steps as in the previous example, one can show that the system (\ref{eq:holotore}) + (\ref{eq:transtore}) has a solution if and only if $e$ is a solution of the scalar equation
\begin{eqnarray*} \lefteqn{|a|\cos\left( \frac{\arg(w) - \arg(z)}{2} \right)} \\ && = \cos\left( \frac{\arg(z) + \arg(w)}{2} + \arg\left( \Gamma(\tfrac{1}{2} + i \ \varepsilon_1) \right) + \arg\left( \Gamma(\tfrac{1}{2} + i \ \varepsilon_2) \right) - (\varepsilon_1 + \varepsilon_2)\ln(k) \right)  \end{eqnarray*}
with
\begin{equation*} |a|^2= \frac{\exp(-2\pi\varepsilon_1) + \exp(-2\pi\varepsilon_2) + 2\cos(\theta_2 - \theta_1)\exp(-\pi(\varepsilon_1 + \varepsilon_2))}{(1+\exp(-2\pi\varepsilon_1))(1+\exp(-2\pi\varepsilon_2))}, \end{equation*}
\begin{equation*} w = \exp(-\pi\varepsilon_2 -i(\theta_2+\theta_3)) + \exp(-\pi\varepsilon_1-i(\theta_1+\theta_3)) \end{equation*}
and
\begin{equation*} z = \exp(-\pi\varepsilon_1-i\theta_2) + \exp(-\pi\varepsilon_2-i\theta_1). \end{equation*}
Moreover, one has
\begin{equation} \varepsilon_{1}^{(0)} = \varepsilon_{2}^{(0)} = \frac{e}{2\pi} := \varepsilon^{(0)}. \end{equation}
It remains to compute the quantities $\theta_{j}$ (up to $O(k^{-1})$). The principal actions are easily computed:
\begin{equation} c_{0}(\gamma_{1}) = -\pi, \quad c_{0}(\gamma_{2}) = 3\pi, \quad c_{0}(\gamma_{1}) = -2\pi. \end{equation}
Furthermore, one can check that the subprincipal actions are given by
\begin{equation} \tilde{c}_{1}(\gamma_{1}) = 2\varepsilon^{(0)} \ln\left( \frac{8}{\pi} \right) = \tilde{c}_{1}(\gamma_{2}), \quad \tilde{c}_{1}(\gamma_{3}) = 0. \end{equation}
Finally, one has 
\begin{equation} \tilde{\epsilon}(\gamma_{1}) = \tilde{\epsilon}(\gamma_{2}) = \tilde{\epsilon}(\gamma_{3}) = 0. \end{equation}

\begin{figure}[H]
\subfigure[$k=10$]{\includegraphics[scale=0.32]{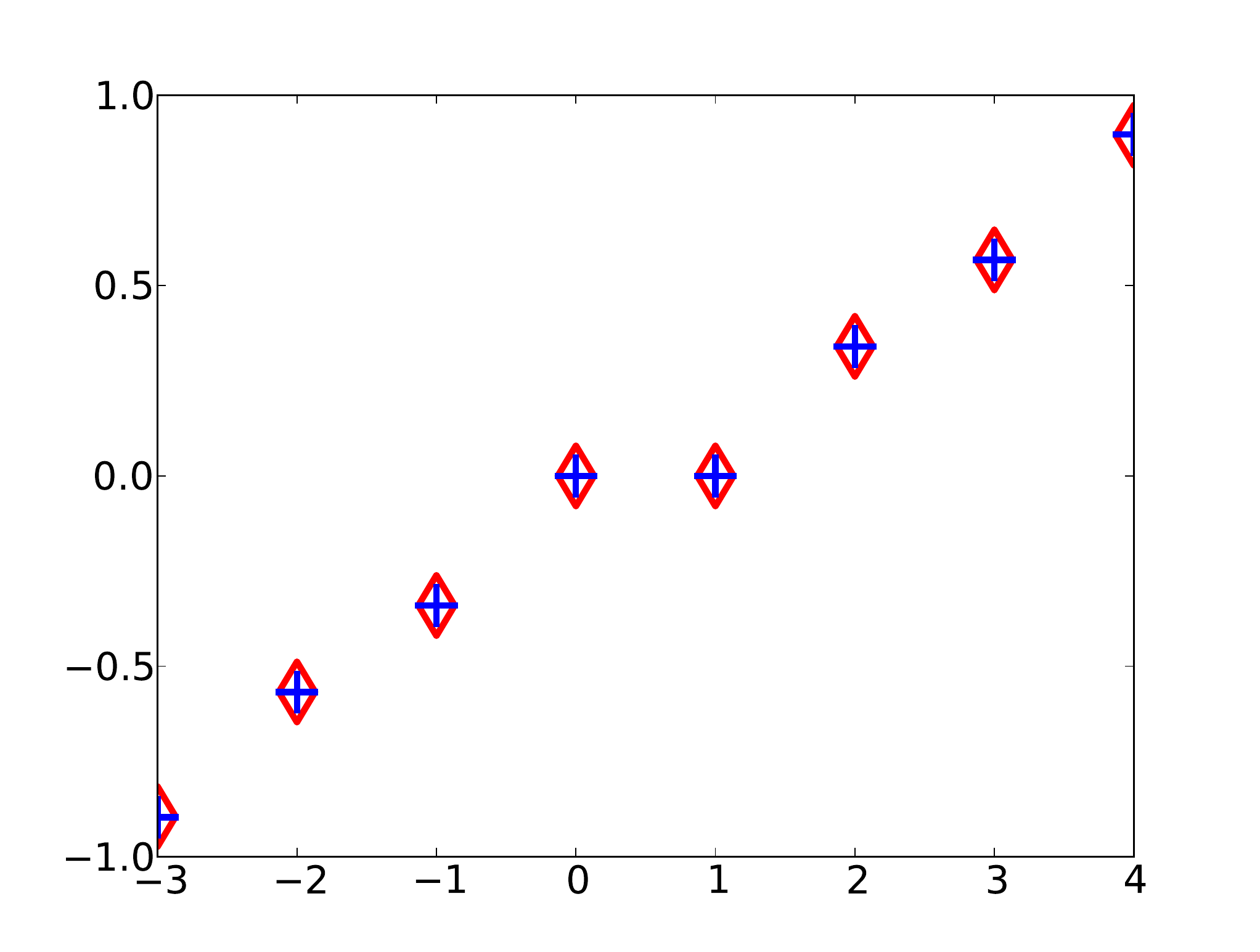} } 
\subfigure[$k=100$]{\includegraphics[scale=0.32]{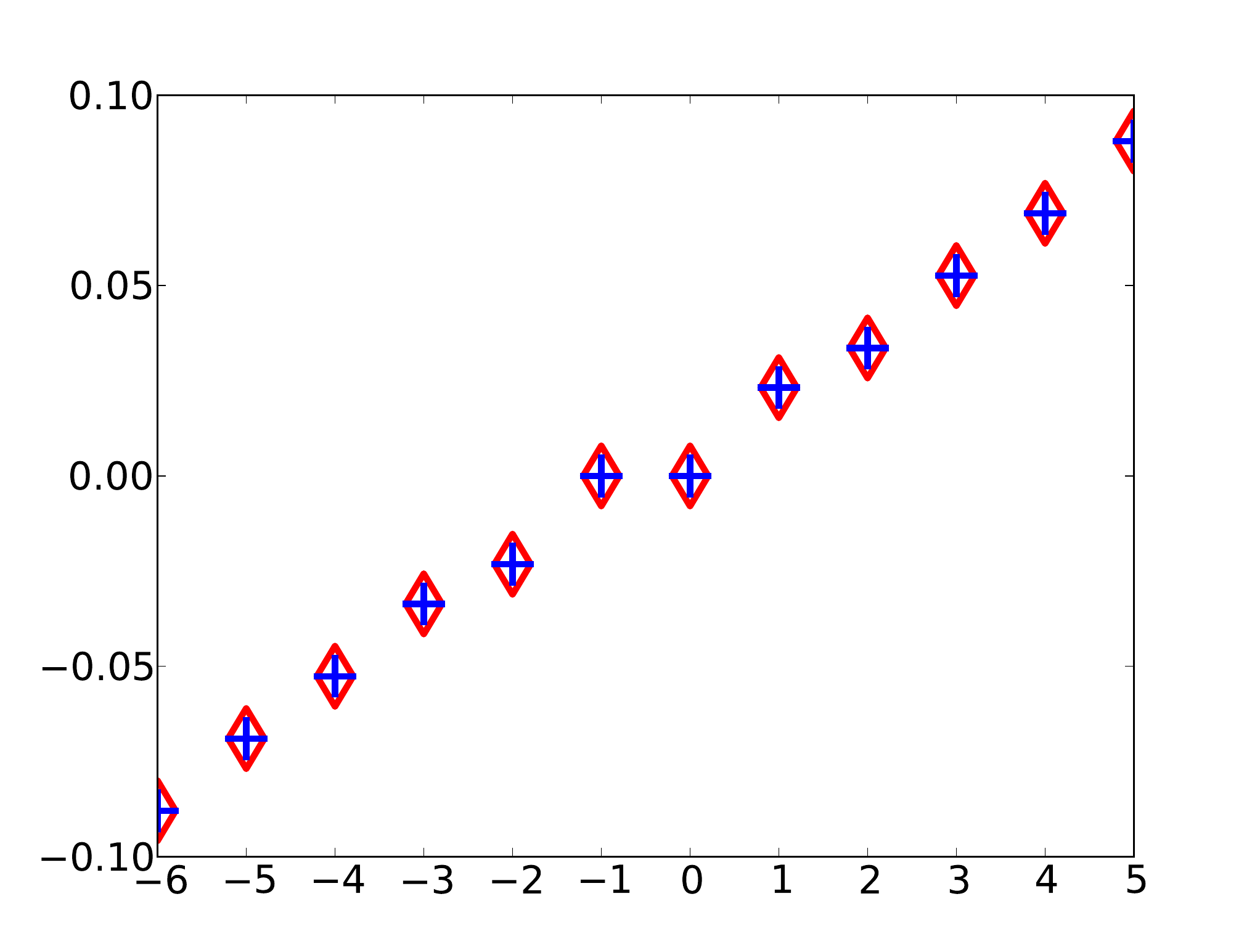} }
\caption{Eigenvalues in $[-10 k^{-1}, 10 k^{-1}]$; in red diamonds, the eigenvalues of $A_{k}$ obtained numerically; in blue crosses, the theoretical eigenvalues derived from the Bohr-Sommerfeld conditions.}
\label{fig:comparaisonBStore}
\end{figure}

\section*{Acknowledgements}
I would like to thank San V\~u Ng\d{o}c and Laurent Charles for their helpful remarks and suggestions.

\bibliographystyle{abbrv}
\bibliography{BShyp}

\end{document}